\documentclass{dis}

\usepackage{amssymb}
\usepackage{amsmath}
\usepackage{amsthm}

\allowdisplaybreaks

\newtheorem{theorem}{Theorem}[chapter]
\newtheorem{proposition}[theorem]{Proposition}
\newtheorem{corollary}[theorem]{Corollary}
\newtheorem{lemma}[theorem]{Lemma}

\theoremstyle{definition}
\newtheorem{definition}[theorem]{Definition}

\newtheorem{remark}[theorem]{Remark}

\def\ls{\lesssim}
\def\gs{\gtrsim}
\def\fz{\infty}

\renewcommand{\r}{\right}

\def\supp{{\mathop\mathrm{\,supp\,}}}

\def\rr{{\mathbb R}}
\def\rn{{{\rr}^n}}
\def\zz{{\mathbb Z}}
\def\nn{{\mathbb N}}
\def\cc{{\mathbb C}}

\newcommand{\wz}{\widetilde}
\newcommand{\oz}{\overline}

\newcommand{\cb}{{\mathcal B}}

\newcommand{\cd}{{\mathcal D}}

\newcommand{\cg}{{\mathcal G}}

\newcommand{\cp}{{\mathcal P}}

\newcommand{\cs}{{\mathcal S}}
\newcommand{\cy}{{\mathcal Y}}
\def\az{\alpha}
\def\lz{\lambda}
\def\blz{\Lambda}

\def\bfai{\Phi}
\def\dz {\delta}
\def\ez {\eta}
\def\epz{\epsilon}

\def\bz{\beta}
\def\rz{\rho}
\def\pz{\psi}
\def\xz{\xi}

\def\gz{{\gamma}}

\def\tz{\theta}
\def\sz{\sigma}
\def\zez{\zeta}
\def\wz{\widetilde}

\def\ls{\lesssim}
\def\gs{\gtrsim}

\def\boz{\Omega}
\def\oz{\omega}

\def\gfz{\genfrac{}{}{0pt}{}}
\def\pat{\partial}

\def\fin{{\mathop\mathrm{fin}}}

\def\esup{\mathop\mathrm{\,esssup\,}}
\def\einf{\mathop\mathrm{\,essinf\,}}

\def\bmo{{{\mathop\mathrm {bmo}}}}
\def\bbmo{{{\mathop\mathrm {BMO}}}}

\def\hs{\hspace{0.3cm}}

\def\gfz{\genfrac{}{}{0pt}{}}

\def\supp{{\mathop\mathrm{\,supp\,}}}
\def\diam{{\mathop\mathrm{\,diam\,}}}
\def\dist{{\mathop\mathrm{\,dist\,}}}
\def\loc{{\mathop\mathrm{loc\,}}}
\def\lfz{\lfloor}
\def\rf{\rfloor}
\numberwithin{equation}{chapter}

\begin{document}

\keywords{local weight, local Orlicz-Hardy space, atom, local grand
maximal function, quasi-Banach space, $\mathop\mathrm{BMO}$-type
space, duality, local Riesz transform, local fractional integral,
pseudo-differential operator.} \mathclass{Primary 46E30; Secondary
42B35, 42B30, 42B25, 42B20, 35S05, 47G30, 47B06.}
\thanks{The first author is supported by the National
Natural Science Foundation (Grant No. 10871025) of China
and Program for Changjiang Scholars and Innovative
Research Team in University of China.

\medskip

Both authors would like to thank Professor Lin Tang
for some helpful discussions on the subject of this paper and
they would also like to thank the referee and the copy editor
for their valuable remarks which made this article
more readable.}

\abbrevauthors{D. Yang and S. Yang}
\abbrevtitle{Weighted local Orlicz-Hardy spaces}

\title{{\vspace{-5cm}\bf Dissertationes Math. 
(Rozprawy Mat.) (to appear)}\\
\vspace{4.5cm} Weighted local Orlicz-Hardy spaces
with applications to pseudo-differential operators}

\author{Dachun Yang}
\address{School of Mathematical Sciences\\ Beijing Normal
University\\ Laboratory of Mathematics and Complex Systems\\Ministry
of Education\\ Beijing 100875\\People's Republic of China\\
E-mail: dcyang@bnu.edu.cn}

\author{Sibei Yang}
\address{School of Mathematical Sciences\\ Beijing Normal
University\\ Laboratory of Mathematics and Complex Systems\\Ministry
of Education\\ Beijing 100875\\People's Republic of China\\
E-mail: yangsibei@mail.bnu.edu.cn}

\maketitledis

\tableofcontents
\begin{abstract}
Let $\Phi$ be a concave function on $(0,\infty)$ of strictly lower
type $p_{\Phi}\in(0,1]$ and $\omega\in
A^{\mathop\mathrm{loc}}_{\infty}(\mathbb{R}^n)$ (the class of local
weights introduced by V. S. Rychkov). We introduce the weighted
local Orlicz-Hardy space $h^{\Phi}_{\omega}(\mathbb{R}^n)$ via the
local grand maximal function. Let $\rho(t)\equiv
t^{-1}/\Phi^{-1}(t^{-1})$ for all $t\in(0,\infty)$. We also
introduce the $\mathop\mathrm{BMO}$-type space
$\mathop\mathrm{bmo}_{\rho,\,\omega}(\mathbb{R}^n)$ and establish
the duality between $h^{\Phi}_{\omega}(\mathbb{R}^n)$ and
$\mathop\mathrm{bmo}_{\rho,\,\omega}(\mathbb{R}^n)$.
Characterizations of $h^{\Phi}_{\omega}(\mathbb{R}^n)$, including
the atomic characterization, the local vertical and the local
nontangential maximal function characterizations, are presented.
Using the atomic characterization, we prove the existence of finite
atomic decompositions achieving the norm in some dense subspaces of
$h^{\Phi}_{\omega}(\mathbb{R}^n)$, from which, we further deduce
that for a given admissible triplet $(\rho,\,q,\,s)_{\omega}$ and a
$\beta$-quasi-Banach space $\mathcal{B}_{\beta}$ with
$\beta\in(0,1]$, if $T$ is a $\mathcal{B}_{\beta}$-sublinear
operator, and maps all $(\rho,\,q,\,s)_{\omega}$-atoms and
$(\rho,\,q)_{\omega}$-single-atoms with $q<\infty$ (or all
continuous $(\rho,\,q,\,s)_{\omega}$-atoms with $q=\infty$) into
uniformly bounded elements of $\mathcal{B}_{\beta}$, then $T$
uniquely extends to a bounded $\mathcal{B}_{\beta}$-sublinear
operator from $h^{\Phi}_{\omega}(\mathbb{R}^n)$ to
$\mathcal{B}_{\beta}$. As applications, we show that the local Riesz
transforms are bounded on $h^{\Phi}_{\omega}(\mathbb{R}^n)$, the
local fractional integrals are bounded from
{\normalsize$h^p_{\omega^p}(\mathbb{R}^n)$} to
{\normalsize$L^q_{\omega^q}(\mathbb{R}^n)$} when $q>1$ and from
{\normalsize$h^p_{\omega^p}(\mathbb{R}^n)$} to
{\normalsize$h^q_{\omega^q}(\mathbb{R}^n)$} when $q\le 1$, and some
pseudo-differential operators are also bounded on both
$h^{\Phi}_{\omega}(\mathbb{R}^n)$. All results for any general
$\Phi$ even when $\omega\equiv 1$ are new.
\end{abstract}
\makeabstract

\chapter{Introduction}\label{s1}

It is well known that the theory of the classical local Hardy
spaces, originally introduced by Goldberg \cite{Go79}, plays an
important role in partial differential equations and harmonic
analysis; see, for example, \cite{Go79,B81,R01,Tay91,Tr83,Tr92} and
their references. In particular, pseudo-differential operators are
bounded on local Hardy spaces $h^p (\rn)$ with $p\in(0,1]$, but they
are not bounded on Hardy spaces $H^p (\rn)$ with $p\in(0,1];$ see
\cite{Go79} (also \cite{Tr83, Tr92}). In \cite{B81}, Bui studied the
weighted version $h^p_{\oz}(\rn)$ of the local Hardy space $h^p
(\rn)$  with $\oz\in A_{\fz}(\rn)$, where and in what follows,
$A_q(\rn)$ for $q\in[1,\fz]$ denotes the \emph{class of Muckenhoup's
weights}; see, for example, \cite{gr85} for their definitions and
properties.

Rychkov \cite{R01} introduced and studied a class of local weights,
denoted by $A^{\loc}_{\fz}(\rn)$ (see also Definition \ref{d2.1}
below), and the weighted Besov-Lipschitz spaces and Triebel-Lizorkin
spaces with weights belonging to $A^{\loc}_{\fz}(\rn)$, which
contains $A_{\fz}(\rn)$ weights and exponential weights introduced
by Schott \cite{S98} as special cases. In particular, Rychkov
\cite{R01} generalized some of the theory of weighted local Hardy
spaces developed by Bui \cite{B81} to $A^{\loc}_{\fz}(\rn)$ weights.
In fact, Rychkov established  the local vertical and the local
nontangential maximal function characterizations of weighted local
Hardy spaces with $A^{\loc}_{\fz}(\rn)$ weights. Very recently, Tang
\cite{Ta1} established the weighted atomic decomposition
characterization of the weighted local Hardy space $h^p_{\oz}(\rn)$
with $\oz\in A^{\loc}_{\fz}(\rn)$ via the local grand maximal
function. Motivated by \cite{blyz08}, Tang also established some
criterions for boundedness of $\cb_{\bz}$-sublinear operators on
$h^p_{\oz}(\rn)$ (see Section \ref{s6} for the notion of
$\cb_{\bz}$-sublinear operators, which was first introduced in
\cite{yz08}). As applications, Tang \cite{Ta1, Ta2} proved that some
strongly singular integrals, pseudo-differential operators and their
commutators are bounded on $h^p_{\oz}(\rn)$. It is worth pointing
out that in recent years, many papers are focused on criterions for
boundedness of (sub)linear operators on various Hardy spaces with
different underlying spaces (see, for example,
\cite{B05,msv08,yz09,blyz08, gly08,yz08,rv,Ta1}), and on entropy and
approximation numbers of embeddings of function spaces with
Muckenhoupt weight (see, for example, \cite{hs08,hs1,hs2, hs3}).

It is also well known that the classical $\bbmo$ space (the {\it
space of functions with bounded mean oscillation}) originally
introduced by John and Nirenberg \cite{jn61} and the classical
Morrey space originally introduced by Morrey \cite{M67} play an
important role in the study of partial differential equations and
harmonic analysis; see, for example, \cite{fs72,ddy05,dxy07,N06}. In
particular, Fefferman and Stein \cite{fs72} proved that $\bbmo(\rn)$
is the dual space of the Hardy space $H^1 (\rn)$. Moveover, Goldberg
\cite{Go79} introduced the space $\bmo(\rn)$ and proved that
$\bmo(\rn)$ is the dual space of the local Hardy space $h^1 (\rn)$.

On the other hand,  as the generalization of $L^p (\rn)$, the Orlicz
space was introduced by Birnbaum-Orlicz in \cite{bo31} and Orlicz in
\cite{O32}, since then, the theory of the Orlicz spaces themselves
has been well developed and these spaces have been widely used in
probability, statistics, potential theory, partial differential
equations, as well as harmonic analysis and some other fields of
analysis; see, for example, \cite{rr91,rr00,byz08,mw08,io02}.
Moreover, Orlicz-Hardy spaces are also suitable substitutes of the
Orlicz spaces in dealing with many problems of analysis; see, for
example, \cite{J80,Vi87,Str79,jy10}. Recall that Orlicz-Hardy spaces
and their dual spaces were studied by Janson \cite{J80} on $\rn$ and
Viviani \cite{Vi87} on spaces of homogeneous type in the sense of
Coifman and Weiss \cite{cw71}. Recently, Orlicz-Hardy spaces
associated with some differential operators and their dual spaces
were introduced and studied in \cite{jyz09,jy10}.

Let $\oz\in A^{\loc}_{\fz}(\rn)$, $\bfai$ be a concave function on
$(0,\fz)$ of strictly lower type $p_{\bfai}\in(0,1]$ (see
\eqref{2.6} below for the definition) and
$$\rz(t)\equiv t^{-1}/\bfai^{-1}(t^{-1})$$
for all $t\in(0,\fz)$, where $\bfai^{-1}$ is the inverse function of
$\bfai$. A typical example of such Orlicz functions is
$\bfai(t)\equiv t^p$ for all $t\in(0,\fz)$ and $p\in(0,1]$.
Motivated by \cite{R01,Ta1,Go79,jyz09,jy10,blyz08}, in this paper,
we introduce the weighted local Orlicz-Hardy space
$h^{\Phi}_{\omega}(\mathbb{R}^n)$ via the local grand maximal
function. We also introduce the $\mathop\mathrm{BMO}$-type space
$\mathop\mathrm{bmo}_{\rho,\,\omega}(\mathbb{R}^n)$ and establish
the duality between $h^{\Phi}_{\omega}(\mathbb{R}^n)$ and
$\mathop\mathrm{bmo}_{\rho,\,\omega}(\mathbb{R}^n)$.
Characterizations of $h^{\Phi}_{\omega}(\mathbb{R}^n)$, including
the atomic characterization, the local vertical and the local
nontangential maximal function characterizations, are presented.
Using the atomic characterization, we prove the existence of finite
atomic decompositions achieving the norm in some dense subspaces of
$h^{\Phi}_{\omega}(\mathbb{R}^n)$, from which, we further deduce
that for a given admissible triplet $(\rho,\,q,\,s)_{\omega}$ and a
$\beta$-quasi-Banach space $\mathcal{B}_{\beta}$ with
$\beta\in(0,1]$, if $T$ is a $\mathcal{B}_{\beta}$-sublinear
operator, and maps all $(\rho,\,q,\,s)_{\omega}$-atoms and
$(\rho,\,q)_{\omega}$-single-atoms with $q<\infty$ (or all
continuous $(\rho,\,q,\,s)_{\omega}$-atoms with $q=\infty$) into
uniformly bounded elements of $\mathcal{B}_{\beta}$, then $T$
uniquely extends to a bounded $\mathcal{B}_{\beta}$-sublinear
operator from $h^{\Phi}_{\omega}(\mathbb{R}^n)$ to
$\mathcal{B}_{\beta}$. As applications, we show that the local Riesz
transforms are bounded on $h^{\Phi}_{\omega}(\mathbb{R}^n)$, the
local fractional integrals are bounded from $h^p_{\oz^p}(\rn)$ to
$L^q_{\oz^q}(\rn)$ when $q>1$ and from $h^p_{\oz^p}(\rn)$ to
$h^q_{\oz^q}(\rn)$ when $q\le 1$, and some pseudo-differential
operators are also bounded on both $h^{\Phi}_{\omega}(\mathbb{R}^n)$
and $\mathop\mathrm{bmo}_{\rho,\,\omega}(\mathbb{R}^n)$. We point
out that the Orlicz-Hardy spaces $h^{\bfai}_\oz(\rn)$ include the
classical local Hardy spaces of Goldberg \cite{Go79}, the weighted
local Hardy spaces of Bui \cite{B81} and the weighted local Hardy
spaces of Tang \cite{Ta1} as special cases. Moreover, the method of
obtaining atomic decompositions used in this paper (see the proof of
Theorem \ref{t5.6} below) is different from the classical methods in
\cite{Go79,B81}. Indeed, following Bownik \cite{B03} (see also
\cite{blyz08,Ta1}), we give a direct proof for the weighted atomic
characterization of $h^{\bfai}_\oz(\rn)$, without invoking the
atomic characterization of $H^\Phi_\oz(\rn)$. All results of this
paper for any general $\Phi$ even when $\oz\equiv 1$ are new.

Precisely, this paper is organized as follows. In Section \ref{s2},
we first recall some definitions and notation concerning local
weights introduced in \cite{R01,Ta1}, then describe some basic
assumptions and present some properties of Orlicz functions
considered in this paper.

In Section \ref{s3}, we first introduce the weighted local
Orlicz-Hardy space $h^{\bfai}_{\oz,\,N}(\rn)$ via the local grand
maximal function, and then the weighted atomic local Orlicz-Hardy
space $h^{\rz,\,q,\,s}_{\oz}(\rn)$ for any admissible triplet
$(\rz,\,q,\,s)_{\oz}$ (see Definition \ref{d3.4} below). We point
out that when $\bfai(t)\equiv t^p$ for all $t\in(0,\fz)$ and
$p\in(0,1]$, the weighted local Orlicz-Hardy space
$h^{\bfai}_{\oz,\,N}(\rn)$ is just the weighted local Hardy space
$h^p_{\oz,\,N}(\rn)$ introduced by Tang in \cite{Ta1}. Next, we
establish the local vertical and the local nontangential maximal
function characterizations of $h^{\bfai}_{\oz,\,N}(\rn)$ via a local
Calder\'on reproducing formula and some useful estimates established
by Rychkov \cite{R01}, which generalizes \cite[Theorem 2.24]{R01} by
taking $\bfai(t)\equiv t^p$ for all $t\in(0,\fz)$ and $p\in(0,1]$;
see Theorems \ref{t3.12} and \ref{t3.14} and Remark \ref{r3.13}
below. Finally, we present some properties of these weighted local
Orlicz-Hardy spaces $h^{\bfai}_{\oz,\,N}(\rn)$ and weighted atomic
local Orlicz-Hardy spaces $h^{\rz,\,q,\,s}_{\oz}(\rn)$.

Throughout the whole paper, as usual, $\cd(\rn)$ denotes the {\it
set of all $C^\fz(\rn)$ functions on $\rn$ with compact support},
endowed with the inductive limit topology, and $\cd'(\rn)$ its {\it
topological dual space}, endowed with the weak $\ast$-topology. Let
$\lfz r\rf$ for any $r\in\rr$ denote the \emph{maximal integer not more
than $r$}. In Section \ref{s4}, for any given $f\in\cd'(\rn)$, integer
$$s\ge\lfz n(q_{\oz}/p_{\bfai}-1)\rf$$
and $\lz>\inf_{x\in\rn}\cg_{N,\,\wz{R}} (f)(x)$, where $q_{\oz}$,
$p_{\bfai}$ and $\cg_{N,\,\wz{R}}(f)$ are respectively as in
\eqref{2.4}, \eqref{2.6} and \eqref{3.2} below, and
$\wz{R}=2^{3(10+n)}$, following \cite{St93,B03,blyz08,Ta1}, via a
Whitney decomposition of $\boz_{\lz}$ in \eqref{4.1}, we obtain the
Calder\'on-Zygmund decomposition $f\equiv g+\sum_i b_i$ in
$\cd'(\rn)$ of degree $s$ and height $\lz$ associated with the local
grand maximal function $\cg_{N,\,\wz{R}}(f)$. The main task of
Section \ref{s4} is to establish some subtle estimates for $g$ and
$\{b_i\}_i$. Precisely, Lemmas \ref{l4.2} through \ref{l4.5} are
estimates on $\{b_i\}_i$, the bad part of $f$, while Lemmas
\ref{l4.6} and \ref{l4.7} on $g$, the good part of $f$. As an
application of these estimates, we obtain the density of
$L^q_{\oz}(\rn)\cap h^{\bfai}_{\oz,\,N}(\rn)$ in
$h^{\bfai}_{\oz,\,N}(\rn)$, where $q\in(q_{\oz},\fz)$ (see Corollary
\ref{c4.8} below). Differently from the proof of
\cite[Lemma\,4.9]{Ta1}, via an application of the boundedness of the
local vector-valued Hardy-Littlewood maximal operator obtained by
Rychkov \cite{R01} (see also Lemma \ref{l3.10} below), our Lemma
\ref{l4.7} below improves \cite[Lemma\,4.9]{Ta1} even when
$\bfai(t)\equiv t^p$ for all $t\in(0,\fz)$ and $p\in(0,1]$, which is
necessary for Corollary \ref{c4.8}.

In Section \ref{s5}, we prove that for any given admissible triplet
$(\rz,\,q,\,s)_{\oz}$,
$$h^{\rz,\,q,\,s}_{\oz}(\rn)=h^{\bfai}_{\oz,\,N}(\rn)$$
with equivalent norms when positive integer $N\ge N_{\bfai,\,\oz}$ (see
\eqref{3.25} below for the definition of $N_{\bfai,\,\oz}$), by
using the Calder\'on-Zygmund decomposition and some subtle estimates
obtained in Section \ref{s4}, which completely covers
\cite[Theorem\,5.1]{Ta1} by taking $\bfai(t)\equiv t^p$ for all
$p\in(0,1]$ and $t\in(0,\fz)$; see Theorem \ref{t5.6} and Remark
\ref{r5.7} below. It is worth pointing out that we show Theorem
\ref{t5.6} by a way different from the methods in \cite{Go79,B81},
but close to those in \cite{B03,blyz08,Ta1}. For simplicity, in the
rest of this introduction, we denote by $h^{\bfai}_{\oz}(\rn)$ the
{\it weighted local Orlicz-Hardy space $h^{\bfai}_{\oz,\,N}(\rn)$
with} $N\ge N_{\bfai,\,\oz}$.

Assume that $(\rz,\,q,\,s)_{\oz}$ is an admissible triplet. Let
$h^{\rz,\,q,\,s}_{\oz,\,\fin}(\rn)$ be the {\it space of all finite
linear combinations of $(\rz,\,q,\,s)_{\oz}$-atoms and
$(\rz,\,q)_{\oz}$-single-atoms} (see Definition \ref{d6.1} below),
and $h^{\rz,\,\fz,\,s}_{\oz,\,\fin,\,c}(\rn)$ the {\it space of all
$f\in h^{\rz,\,\fz,\,s}_{\oz,\,\fin}(\rn)$ with compact support}. In
Section \ref{s6}, we prove that
$\|\cdot\|_{h^{\rz,\,q,\,s}_{\oz,\,\fin}(\rn)}$ and
$\|\cdot\|_{h^{\bfai}_{\oz}(\rn)}$ are equivalent quasi-norms on
$h^{\rz,\,q,\,s}_{\oz,\,\fin}(\rn)$ when $q <\fz$, and
$\|\cdot\|_{h^{\rz,\,\fz,\,s}_{\oz,\,\fin}(\rn)}$ and
$\|\cdot\|_{h^{\bfai}_{\oz}(\rn)}$ are equivalent quasi-norms on
$h^{\rz,\,\fz,\,s}_{\oz,\,\fin,\,c}(\rn)\cap C(\rn)$ when $q=\fz$
(see Theorem \ref{t6.2} below). As an application, we prove that for
a given admissible triplet $(\rho,\,q,\,s)_{\omega}$ and a
$\beta$-quasi-Banach space $\mathcal{B}_{\beta}$ with
$\beta\in(0,1]$, if $\bfai$ has a upper type $\wz{p}$ satisfying
$0<\wz{p}\le\beta$, and $T$ is a $\cb_{\bz}$-sublinear operator
mapping all $(\rz,\,q,\,s)_{\oz}$-atoms and
$(\rz,\,q)_{\oz}$-single-atoms with $q\in(q_{\oz},\fz)$ (or all
continuous $(\rho,\,q,\,s)_{\omega}$-atoms with $q=\infty$) into
uniformly bounded elements of $\cb_{\bz}$, then $T$ uniquely extends
to a bounded $\cb_{\bz}$-sublinear operator from
$h^{\bfai}_{\oz}(\rn)$ to $\cb_{\bz}$ which coincides with $T$ on
these $(\rz,\,q,\,s)_{\oz}$-atoms and
$(\rz,\,q)_{\oz}$-single-atoms; see Theorem \ref{t6.4} below. We
remark that this extends both the results of
Meda-Sj\"ogren-Vallarino \cite{msv08} and Yang-Zhou \cite{yz09} to
the setting of weighted local Orlicz-Hardy spaces. We also point out
that Theorem \ref{t6.2}(i) and Theorem \ref{t6.4}(i) below
completely cover \cite[Theorem\,6.1]{Ta1} and \cite[Theorem
\,6.2]{Ta1}, respectively, by taking $\bfai(t)\equiv t^p$ for all
$t\in(0,\fz)$ and $p\in(0,1]$; and Theorem \ref{t6.2}(ii) and
Theorem \ref{t6.4}(ii) below are new even for $\bfai(t)\equiv t^p$
with $t\in(0,\fz)$ and $p\in(0,1];$ see Remark \ref{r6.5} below.

Let $(\rz,\,q,\,s)_{\oz}$ be an admissible triplet, $q'$ the dual
exponent of $q$ and  $q_{\oz}$ the critical index of $\oz$. In
Section \ref{s7}, we introduce the $\bbmo$-type space
$\bmo_{\rz,\,\oz}^{q'}(\rn)$ and prove that
$$\left[h^{\rz,\,q,\,s}_{\oz} (\rn)\r]^{\ast}=\bmo_{\rz,\,\oz}^{q'}
(\rn),$$ where $\left[h^{\rz,\,q,\,s}_{\oz} (\rn)\r]^{\ast}$ denotes
the dual space of $h^{\rz,\,q,\,s}_{\oz} (\rn)$; see Theorem
\ref{t7.5} below. Denote $\bmo^{1}_{\rz,\,\oz}(\rn)$ simply by
$\bmo_{\rz,\,\oz}(\rn)$. As applications of Theorems \ref{t5.6} and
\ref{t7.5}, we see that for $q\in[1,\frac{q_{\oz}}{q_{\oz}-1})$,
$\bmo^{q}_{\rz,\,\oz}(\rn)=\bmo_{\rz,\,\oz}(\rn)$ with equivalent
norms and
$$\left[h^{\bfai}_{\oz}(\rn)\r]^{\ast}=\bmo_{\rz,\,\oz}(\rn);$$
see Corollaries \ref{c7.6} and \ref{c7.7} below.

In Section \ref{s8}, as applications of Theorem \ref{6.2}, we obtain
the boundedness of some operators from $h^{\bfai}_{\oz}(\rn)$ to
some $\beta$-quasi-Banach space $\mathcal{B}_{\beta}$ with
$\bz\in(0,1]$. First, we prove that the local Riesz transforms are
bounded on $h^{\bfai}_{\oz}(\rn)$ if $p_{\bfai}=p_{\bfai}^+$ and
\eqref{2.5} holds for $p_{\bfai}^+$ with $t\in[1,\fz)$ (see Section
\ref{s2} below for the definitions of $p_{\bfai}^+$), which
completely covers \cite[Lemma\,8.3]{Ta1} by taking $\bfai(t)\equiv
t$ for all $t\in(0,\fz)$; see Theorem \ref{t8.2} and Remark
\ref{r8.4} below. Then we introduce the local fractional integral
operator $I^{\loc}_{\az}$ and show that $I^{\loc}_{\az}$ is bounded
from $h^p_{\oz^p}(\rn)$ to $L^q_{\oz^q}(\rn)$ when $\az\in(0,n)$,
$p\in[\frac{n}{n+\az},1]$, $\frac{1}{q}=\frac{1}{p}-\frac{\az}{n}$,
and $\oz^{\frac{nr}{nr-n-r\az}}\in A^{\loc}_1 (\rn)$ for some
$r\in(\frac{n}{n-\az},\,\fz)$ and $\int_{\rn}[\oz(x)]^p\,dx=\fz$
(see Theorem \ref{t8.10} below); furthermore, when $\az\in(0,1)$,
$p\in(0,\,\frac{n}{n+\az}]$,
$\frac{1}{q}=\frac{1}{p}-\frac{\az}{n}$, and $\oz$ satisfies the
same conditions, we prove that $I^{\loc}_{\az}$ is bounded from
$h^p_{\oz^p}(\rn)$ to $h^q_{\oz^q}(\rn)$ (see Theorem \ref{t8.11}
below). To the best of our knowledge, Theorems \ref{t8.10} and
\ref{t8.11} are new even when $\oz\equiv1$. Finally, let $\oz\in
A_{\fz}(\phi)$ which was introduced by Tang \cite{Ta2} (see also
Definition \ref{d8.13} below) and $T$ be an $S^0_{1,\,0}(\rn)$
pseudo-differential operator. We prove that $T$ is bounded on
$h^{\bfai}_{\oz}(\rn)$ if $p_{\bfai}=p_{\bfai}^+$ and \eqref{2.5}
holds for $p_{\bfai}^+$ with $t\in[1,\fz)$; see Theorem \ref{t8.18}
below. We point out that $A_{\fz}(\phi)\subset A^{\loc}_{\fz}(\rn)$
but $A_{\fz}(\phi)\supset A_{\fz}(\rn)$. We also remark that Theorem
\ref{t8.18} below extends \cite[Theorem\,4]{Go79} to the setting of
weighted local Orlicz-Hardy spaces, and completely covers
\cite[Theorem\,7.3]{Ta1} by taking $\bfai(t)\equiv t^p$ for all
$t\in(0,\fz)$ and $p\in(0,1]$ and also \cite[Theorem\,2]{lll10} by
taking $\bfai(t)\equiv t$ for all $t\in(0,\fz)$ and $\oz\in
A_1(\rn)$; see Remark \ref{r8.19} below. As an application of
Theorems \ref{t7.5} and \ref{t8.18}, we also obtain that $T$ is
bounded on $\bmo_{\rz,\,\oz}(\rn)$; see Corollary \ref{c8.20} below.

A main motivation of this paper is to pave the way for the study of
weighted Orlicz-Hardy spaces associated with divergence operators on
strongly Lipschitz domains of $\rn$. The corresponding Hardy spaces
associated with divergence operators on strongly Lipschitz domains
of $\rn$ were originally studied by Auscher and Russ \cite{ar03},
where the atomic characterization of the classical Hardy spaces
plays a key role. Earlier works on Hardy spaces on domains can be
found, for example, in \cite{jsw84,M90,cks93, cds99,tw96}. It was
shown in these papers that the theory of Hardy spaces on domains
plays an important role in partial differential equations and
harmonic analysis.

Finally we make some conventions on notation. Throughout the whole
paper, we denote by $C$ {\it a positive constant} which is
independent of the main parameters, but it may vary from line to
line. We also use $C(\gz,\bz,\cdots)$ to denote {\it a positive
constant depending on the indicated parameters $\gz$, $\bz$,
$\cdots$}. The {\it symbol} $A\ls B$ means that $A\le CB$. If $A\ls
B$ and $B\ls A$, then we write $A\sim B$. The {\it symbol} $\lfz
s\rf$ for $s\in\rr$ denotes the maximal integer not more than $s$.
For any given normed spaces $\mathcal A$ and $\mathcal B$ with the
corresponding norms $\|\cdot\|_{\mathcal A}$ and
$\|\cdot\|_{\mathcal B}$, the {\it symbol} ${\mathcal
A}\subset{\mathcal B}$ means that for all $f\in \mathcal A$, then
$f\in\mathcal B$ and $\|f\|_{\mathcal B}\ls \|f\|_{\mathcal A}$. For
any subset $G$ of $\rn$, we denote by $G^\complement$ the {\it set}
$\rn\setminus G$; for a measurable set $E$, denote by $\chi_{E}$ the
{\it characteristic function of} $E$. We also set $\nn\equiv\{1,\,
2,\, \cdots\}$ and $\zz_+\equiv\nn\cup\{0\}$. For any
$\tz\equiv(\tz_{1},\ldots,\tz_{n})\in\zz_{+}^{n}$, let
$|\tz|\equiv\tz_{1}+\cdots+\tz_{n}$ and
$\partial^{\tz}_x\equiv\frac{\partial^{|\tz|}}{\partial
{x_{1}^{\tz_{1}}}\cdots\partial {x_{n}^{\tz_{n}}}}$. Given a
function $g$ on $\rn$, if $\int_\rn g(x)\,dx\neq 0$, let $L_g\equiv
-1$; otherwise, let $L_g\in\zz_+$ be the {\it maximal integer} such
that $g$ has vanishing moments up to order $L_g$, namely,
$\int_{\rn}g(x)x^{\az}\,dx=0$ for all multi-indices $\az$ with
$|\az|\le L_g$.

\chapter{Preliminaries\label{s2}}

In this section, we first recall some notions and notation
concerning local weights introduced in \cite{R01,Ta1}, then describe
some basic assumptions and present some properties of Orlicz
functions considered in this paper.

\section{$A^{\loc}_p (\rn)$ weights\label{s2.1}}

In this subsection, we recall some notions and properties of the
local weights. Let $Q$ be a cube in $\rn$ and we denote its Lebesgue
measure by $|Q|$. Throughout the whole paper, {\it all cubes are
assumed to be closed and their sides parallel to the coordinate
axes}.
\begin{definition}\label{d2.1}
Let $p\in(1,\fz)$. The {\it weight class $A^{\loc}_p (\rn)$} is
defined to be the set of all nonnegative locally integrable
functions $\oz$ on $\rn$ such that
\begin{equation}\label{2.1}
A^{\loc}_p (\oz)\equiv\sup_{|Q|\le1}\frac{1}{|Q|^p}\int_Q \oz(x)\,dx
\left(\int_Q [\oz(y)]^{-p'/p}\,dy\r)^{p/p'}<\fz,
\end{equation}
where the supremum is taken over all cubes $Q\subset\rn$ with
$|Q|\le1$ and $\frac{1}{p}+\frac{1}{p'}=1$.

When $p=1$, the {\it weight class $A^{\loc}_1 (\rn)$} is defined to
be the set of all nonnegative locally integrable functions $\oz$ on
$\rn$ such that
\begin{equation}\label{2.2}
A^{\loc}_1 (\oz)\equiv\sup_{|Q|\le1}\frac{1}{|Q|}\int_Q \oz(x)\,dx
\left(\esup_{y\in Q}[\oz(y)]^{-1}\r)<\fz,
\end{equation}
where the supremum is taken over all cubes $Q\subset\rn$ with
$|Q|\le1$.

When $p=\fz$, the {\it weight class $A^{\loc}_{\fz} (\rn)$} is
defined to be the set of all nonnegative locally integrable
functions $\oz$ on $\rn$ such that for any $\az\in(0,1)$,
\begin{equation}\label{2.3}
A^{\loc}_{\fz} (\oz;\,\az)\equiv\sup_{|Q|\le1} \left[\sup_{F\subset
Q,|F|\ge\az|Q|}\frac{\oz(Q)}{\oz(F)}\r]<\fz,
\end{equation}
where $F$ runs through all measurable sets in $\rn$ with the
indicated properties, the supremum is taken over all cubes
$Q\subset\rn$ with $|Q|\le1$ and $\oz(Q)\equiv\int_Q \oz(x)\,dx$.
\end{definition}

\begin{remark}\label{r2.2}
(i) We point out that the weight class $A^{\loc}_p (\rn)$ for
$p\in(1,\fz]$ was introduced by Rychkov \cite{R01} and $A^{\loc}_1
(\rn)$ by Tang \cite{Ta1}. By H\"older's inequality, we see that
$A^{\loc}_{p_1} (\rn)\subset A^{\loc}_{p_2} (\rn)\subset
A^{\loc}_{\fz} (\rn)$, if $1\le p_1<p_2<\fz$. Conversely, it was
proved in \cite[Lemma 1.3]{R01} that if $\oz\in A^{\loc}_{\fz}
(\rn)$, then $\oz\in A^{\loc}_{p} (\rn)$ for some $p\in(1,\fz)$.
Thus, we have $A^{\loc}_{\fz} (\rn)=\cup_{1\le p<\fz}A^{\loc}_p
(\rn)$.

(ii) For any constant $\wz{C}\in(0,\fz)$, the condition $|Q|\le1$
can be replaced by $|Q|\le\wz{C}$ in \eqref{2.1}, \eqref{2.2} and
\eqref{2.3}; see \cite[Remark\,1.5]{R01}. In this case, $A^{\loc}_p
(\oz)$ with $p\in[1,\fz)$ and $A^{\loc}_{\fz}(\oz,\,\az)$ depend on
$\wz{C}$.
\end{remark}

In what follows, $Q(x,t)$ denotes the \emph{closed cube centered at
$x$ and of the sidelength $t$}. Similarly, given $Q=Q(x,t)$ and
$\lz\in(0,\fz)$, we write $\lz Q$ for the {\it$\lz$-dilated cube},
which is the cube with the same center $x$ and with sidelength $\lz
t$. Given a Lebesgue measurable set $E$ and a weight $\oz\in
A^{\loc}_{\fz} (\rn)$, let $\oz(E)\equiv\int_{E}\oz(x)\,dx$. For any
$\oz\in A^{\loc}_{\fz} (\rn)$, the \emph{space} $L^p_{\oz}(\rn)$
with $p\in(0,\fz)$ denotes the set of all measurable functions $f$
such that
$$\|f\|_{L^p_{\oz}(\rn)}\equiv\left\{\int_{\rn}|f(x)|^p
\oz(x)\,dx\r\}^{1/p}<\fz,$$ and $L^{\fz}_{\oz} (\rn)\equiv
L^{\fz}(\rn)$. The {\it space $L^{1,\,\fz}_{\oz}(\rn)$} denotes the
set of all measurable functions $f$ such that
$$\|f\|_{L^{1,\,\fz}_{\oz}(\rn)}\equiv\sup_{\lz>0}\left\{\lz
\oz(\{x\in\rn:\,|f(x)|>\lz\})\r\}<\fz.$$ For a positive constant
$\wz{C}$, any locally integrable function $f$ and $x\in\rn$, the
{\it local Hardy-Littlewood maximal function $M^{\loc}_{\wz{C}}(f)$}
is defined by
$$M^{\loc}_{\wz{C}}(f)(x)\equiv\sup_{Q\ni x,\,|Q|\le \wz{C}}
\frac{1}{|Q|}\int_{Q}|f(y)|\,dy,$$ where the supremum is taken over
all cubes $Q\subset\rn$ such that $Q\ni x$ and $|Q|\le\wz{C}$. If
$\wz{C}=1$, we denote $M^{\loc}_{\wz{C}}(f)$ simply by
$M^{\loc}(f)$.

Next, we recall some properties for weights in $A^{\loc}_{\fz}(\rn)$
and $A_p (\rn)$, where and in what follows, $A_p (\rn)$ for
$p\in[1,\fz)$ denotes the classical {\it Muckenhoupt weights}; see
\cite{gr85,St93} for their definitions.

\begin{lemma}\label{l2.3}
$\mathrm{(i)}$ Let $p\in[1,\fz),\,\oz\in A^{\loc}_p(\rn)$, and $Q$
be a unit cube, namely, $l(Q)=1$. Then there exist an
$\overline{\oz}\in A_p (\rn)$ such that $\overline{\oz}=\oz$ on $Q$,
and a positive constant $C$ independent of $Q$ such that $A_p
(\overline{\oz})\le CA^{\loc}_p(\oz)$, where $A_p (\overline{\oz})$
denotes the weight constant of $\overline{\oz}$, which is as in
\eqref{2.1} and \eqref{2.2} by removing the restriction $l(Q)\le1$.

$\mathrm{(ii)}$ If $\oz\in A^{\loc}_p(\rn)$ with $p\in(1,\fz)$, then
there exist $\ez_1,\,\ez_2\in(0,\fz)$ such that $\oz\in
A^{\loc}_{p-\ez_1}(\rn)$ with $p-\ez_1\in(1,\fz)$, and
$\oz^{1+\ez_2}\in A^{\loc}_p(\rn)$.

$\mathrm{(iii)}$ If $1\le p_1<p_2<\fz$, then
$A^{\loc}_{p_1}(\rn)\subset A^{\loc}_{p_2}(\rn)$.

$\mathrm{(iv)}$ For $p\in(1,\fz)$, $\oz\in A^{\loc}_{p}(\rn)$ if and
only if $\oz^{-1/(p-1)}\in A^{\loc}_{p'}(\rn)$, where
$$1/p+1/p'=1.$$

$\mathrm{(v)}$ For $\oz\in A^{\loc}_{\fz}(\rn)$ and $Q=Q(x_0,l(Q))$,
there exists a positive constant $C$ such that $\oz(2Q)\le C\oz(Q)$
when $l(Q)<1$, and $\oz(Q(x_0,l(Q)+1))\le C\oz(Q)$ when $l(Q)\ge1$.

$\mathrm{(vi)}$ If $p\in(1,\fz)$ and $\oz\in A^{\loc}_{p}(\rn)$,
then the local Hardy-Littlewood maximal operator $M^{\loc}$ is
bounded on $L^p_{\oz}(\rn)$.

$\mathrm{(vii)}$ If $\oz\in A^{\loc}_{1}(\rn)$, then $M^{\loc}$ is
bounded from $L^1_{\oz}(\rn)$ to $L^{1,\,\fz}_{\oz}(\rn)$.

$\mathrm{(viii)}$ If $\oz\in A_{p}(\rn)$ with $p\in[1,\fz)$, then
there exists a positive constant $C$ such that for all cubes
$Q_1,\,Q_2\subset\rn$ with $Q_1\subset Q_2$,
$$\frac{\oz(Q_2)}{\oz(Q_1)}\le C\left(\frac{|Q_2|}{|Q_1|}\r)^p.$$
\end{lemma}

Lemma \ref{l2.3}(i) is just \cite[Lemma\,1.1]{R01}. The statements
(ii) through (vii) of Lemma \ref{l2.3} are just Lemma 2.1 and
Corollary 2.1 in \cite{Ta1}, which are deduced from Lemma
\ref{l2.3}(i) and the properties of $A_p (\rn)$; see the proofs of
\cite[Lemma\,2.1,\,Corollary\,2.1]{Ta1}. Lemma \ref{l2.3}(viii) is
included, for example, in \cite{ga79,gr85,St93}.

\begin{remark}\label{r2.4}
Let $\wz{C}$ be a positive constant. It was pointed out in
\cite[Remark\,1.5]{R01} and \cite{Ta1} that (i) through (vii) of
Lemma \ref{l2.3} are also true, if $l(Q)=1$, $l(Q)\ge1,$ $l(Q)<1,$
$Q(x_0,\,l(Q)+1)$ and $M^{\loc}$ are respectively replaced  by
$l(Q)=\wz{C}$, $l(Q)\ge\wz{C},$ $l(Q)<\wz{C},$ $Q(x_0,\,l(Q)+\wz{C})$
and $M^{\loc}_{\wz{C}}$. In this case, the constants appearing in
(i), (vi) and (vii) of Lemma \ref{l2.3} depend on $\wz{C}$.
\end{remark}

For any given $\oz\in A^{\loc}_{\fz}(\rn)$, define the {\it critical
index of $\oz$} by
\begin{equation}\label{2.4}
q_{\oz}\equiv\inf\left\{p\in[1,\fz):\,\oz\in A^{\loc}_p(\rn)\r\}.
\end{equation}
\begin{remark}\label{r2.5}
Obviously, $q_{\oz}\in[1,\fz)$. If $q_{\oz}\in(1,\fz)$, by Lemma
\ref{r2.2}(ii), it is easy to know that $\oz\not\in
A^{\loc}_{q_{\oz}}(\rn)$. Moreover, there exists an $\oz\not\in
A^{\loc}_1 (\rn)$ such that $q_{\oz}=1$. Indeed, for
$t\in\rr\setminus\{0\}$, let
$\oz(t)\equiv[\ln(\frac{1}{|t|})]^{-1}$. Johnson and Neugebauer
\cite[p.\,254,\,Remark]{jn87} showed that $\oz\in(\cap_{p>1}A_p
(\rn))\setminus A_1 (\rn)$. By the fact that $A_p (\rn)\subset
A^{\loc}_p (\rn)$ for all $p\in[1,\fz)$, which is obvious by their
definitions, we see that $\oz\in\cap_{p>1}A^{\loc}_p (\rn)$. We
claim that $\oz\not\in A^{\loc}_1 (\rn)$. In fact, taking
$x\in(0,1/2)$, then we have
\begin{eqnarray*}
M^{\loc}(\oz)(x)\ge\frac{1}{2}\int^{x+1}_{x-1}\oz(t)\,dt\ge
\int^{1/2}_{0}\left[\ln\left(\frac{1}{t}\r)\r]^{-1}\,dt\equiv\fz.
\end{eqnarray*}
Moreover, it is easy to see that $\oz(x)\to0$ as $x\to0$. Thus, by
\eqref{2.2}, we know that $\oz\not\in A^{\loc}_1 (\rn)$.
\end{remark}

For $\cd(\rn),\,\cd'(\rn)$ and $L^q_{\oz}(\rn)$, we have the
following conclusions.
\begin{lemma}\label{l2.6}
Let $\oz\in A^{\loc}_{\fz}(\rn)$, $q_{\oz}$ be as in \eqref{2.4} and
$p\in(q_{\oz},\fz]$.

$\mathrm{(i)}$ If $\frac{1}{p}+\frac{1}{p'}=1$, then
$\cd(\rn)\subset L^{p'}_{\oz^{-1/(p-1)}}(\rn)$.

$\mathrm{(ii)}$ $L^{p}_{\oz}(\rn)\subset\cd' (\rn)$ and the
inclusion is continuous.

$\mathrm{(iii)}$ Let $\phi\in\cd(\rn)$ and
$\int_{\rn}\phi(x)\,dx=1$. If $q\in(q_{\oz},\fz)$, then for any
$f\in L^q_{\oz}(\rn)$, $f\ast\phi_t\to f$ in $L^q_{\oz}(\rn)$ as
$t\to0$, where and in what follows, $\phi_t (x)\equiv\frac{1}{t^n}
\phi(\frac{x}{t})$ for all $t\in(0,\fz)$ and $x\in\rn$.
\end{lemma}

We remark that (i) and (ii) of Lemma \ref{l2.6}, and Lemma
\ref{l2.6}(iii) are, respectively, Lemma 2.2 and Proposition 2.1 in
\cite{Ta1}.

\section{Orlicz functions\label{s2.2}}

Let $\bfai$ be a positive function on $\rr_{+}\equiv(0,\fz)$. The
function $\bfai$ is said to be of {\it upper type $p$ (resp. lower
type $p$)} for some $p\in[0,\fz)$, if there exists a positive
constant $C$ such that for all $t\in[1,\fz)$ (resp. $t\in(0,1]$) and
$s\in(0,\fz)$,
\begin{equation}\label{2.5}
\bfai(st)\le Ct^p \bfai(s).
\end{equation}
Obviously, if $\bfai$ is of lower type $p$ for some $p\in(0,\fz)$,
then $\lim_{t\to0_{+}}\bfai(t)=0$. So for the sake of convenience,
if it is necessary, we may assume that $\bfai(0)=0$. If $\bfai$ is
of both upper type $p_1$ and lower type $p_0$, then $\bfai$ is said
to be of \emph{type $(p_0,\,p_1)$}. Let
\begin{eqnarray*}
&&p_{\bfai}^{+}\equiv\inf\{p\in(0,\fz):\,\text{there exists}\,
\,C\in(0,\fz)\,\,\\
&&\hspace{4 em}\text{such that}\,\eqref{2.5} \,\, \text{holds for
all} \,\,t\in[1,\fz)\,\, \text{and}\,\,s\in(0,\fz)\},
\end{eqnarray*}
and
\begin{eqnarray*}
&&p_{\bfai}^{-}\equiv\sup\{p\in(0,\fz):\,\text{there exists}\,
\,C\in(0,\fz)\,\,\\
&&\hspace{4 em}\,\text{such that}\,\eqref{2.5} \,\, \text{holds for
all} \,\,t\in(0,1]\,\, \text{and}\,\,s\in(0,\fz)\}.
\end{eqnarray*}
The function $\bfai$ is said to be of {\it strictly lower type $p$}
if for all $t\in(0,1)$ and $s\in(0,\fz)$, $\bfai(st)\le
t^p\bfai(s)$, and we define
\begin{equation}\label{2.6}
p_{\bfai}\equiv\sup\{p\in(0,\fz):\,\bfai(st)\le t^p\bfai(s)
\,\text{holds for all}\ t\in(0,1)\,\text{and}\ s\in(0,\fz)\}.
\end{equation}
It is easy to see that $p_{\bfai}\le p_{\bfai}^{-}\le p_{\bfai}^{+}$
for all $\bfai$. In what follows, $p_{\bfai},\,p_{\bfai}^{-}$ and
$p_{\bfai}^{+}$ are respectively called the {\it strictly critical
lower type index}, the  {\it critical lower type index} and the {\it
critical upper type index} of $\bfai$. We point out that if
$p_{\bfai}$ is defined as in \eqref{2.6}, then $\bfai$ is also of
strictly critical lower type $p_{\bfai}$; see \cite{jy10} for the
proof.

Throughout the whole paper, we always assume that $\bfai$ satisfies
the following assumption.

\medskip

\noindent{\sc Assumption (A).} Let $\bfai$ be a positive function
defined on $\rr_{+}$, which is of strictly lower type and its
strictly critical lower type index $p_{\bfai}\in(0,1]$. Also assume
that $\bfai$ is continuous, strictly increasing, subadditive and
concave.

\medskip

Notice that if $\bfai$ satisfies Assumption (A), then $\bfai(0)=0$
and $\bfai$ is obviously of upper type 1. For any concave and
positive function $\wz{\bfai}$ of strictly lower type $p$, if we set
$\bfai(t)\equiv\int_0^t \frac{\bfai(s)}{s}\,ds$ for $t\in[0,\fz)$,
then by \cite[Proposition 3.1]{Vi87}, $\bfai$ is equivalent to
$\wz{\bfai}$, namely, there exists a positive constant $C$ such that
$C^{-1}\wz{\bfai}(t)\le\bfai(t)\le C\wz{\bfai}(t)$ for all
$t\in[0,\fz)$; moreover, $\bfai$ is strictly increasing, concave,
subadditive and continuous function of strictly lower type $p$.
Notice that all our results are invariant on equivalent functions
satisfying Assumption (A). From this, we deduce that all results in
this paper with $\bfai$ as in Assumption (A) also holds for all
concave and positive functions $\wz{\bfai}$ of the same strictly
critical lower type $p_{\bfai}$ as $\bfai$.

Let $\bfai$ satisfy Assumption (A) and $\oz\in A_{\fz}^{\loc}(\rn)$.
A measurable function $f$ on $\rn$ is called to belong to the \emph{space}
$L^{\bfai}_{\oz}(\rn)$ if $\int_{\rn}\bfai(|f(x)|)\oz(x)\,dx<\fz$.
Moreover, for any $f\in L^{\bfai}_{\oz}(\rn)$, define
$$\|f\|_{L^{\bfai}_{\oz}(\rn)}\equiv\inf\left\{\lz\in(0,\fz):
\,\int_{\rn}\bfai\left(\frac{|f(x)|}{\lz}\r)\oz(x)\,dx\le1\right\}.$$
Since $\bfai$ is strictly increasing, we define the \emph{function
$\rz$} on $\rr_{+}$ by setting, for all $t\in(0,\fz)$,
\begin{equation}\label{2.7}
\rz(t)\equiv\frac{t^{-1}}{\bfai^{-1}(t^{-1})},
\end{equation}
 where $\bfai^{-1}$ is the inverse function of
$\bfai$. Then the types of $\bfai$  and $\rz$ have the following
relation. Let $0<p_0\le p_1\le1$ and $\bfai$ be an increasing
function, then $\bfai$ is of type $(p_0,p_1)$ if and only if $\rz$
is of type $(p_1^{-1}-1,p_0^{-1}-1);$ see \cite{Vi87} for its proof.
Moreover, it is easy to see that for all $t\in(0,\fz)$,
\begin{equation}\label{2.8}
t\Phi\left(\frac 1{t\rho(t)}\r)=1,
\end{equation}
which is used in what follows.

\chapter{Weighted local Orlicz-Hardy spaces and their maximal
function characterizations\label{s3}}

In this section, we introduce the weighted local Orlicz-Hardy space
$h^{\bfai}_{\oz,\,N}(\rn)$ via the local grand maximal function and
establish its local vertical and nontangential maximal function
characterizations via a local Calder\'on reproducing formula and
some useful estimates obtained by Rychkov \cite{R01}. We also
introduce the weighted atomic local Orlicz-Hardy space
$h^{\rz,\,q,\,s}_{\oz}(\rn)$ and give some basic properties of these
spaces.

First, we introduce some local maximal functions. For $N\in\zz_+$
and $R\in(0,\fz)$, let
\begin{eqnarray*}
\cd_{N,\,R}(\rn)\equiv&&\Bigg\{\pz\in\cd(\rn):\,\supp(\pz)
\subset B(0,R),\\
&&\hspace{0.5cm}\left.\|\pz\|_{\cd_{N}(\rn)}\equiv\sup_{x\in\rn}
\sup_{{\az\in\zz^n_+},\,{|\az|\le
N}}|\partial^{\az}\pz(x)|\le1\r\}.
\end{eqnarray*}

\begin{definition}\label{d3.1}
Let $N\in\zz_+$ and $R\in(0,\fz)$. For any $f\in\cd'(\rn)$, the {\it
local nontangential grand maximal function} $\wz{\cg}_{N,\,R} (f)$
of $f$ is defined by setting, for all $x\in\rn$,
\begin{eqnarray}\label{3.1}
\wz{\cg}_{N,\,R} (f)(x)\equiv\sup\left\{|\pz_t\ast
f(z)|:\,|x-z|<t<1,\,\pz\in\cd_{N,\,R}(\rn)\r\},
\end{eqnarray}
and the {\it local vertical grand maximal function} $\cg_{N,\,R}
(f)$ of $f$ is defined by setting, for all $x\in\rn$,
\begin{eqnarray}\label{3.2}
\cg_{N,\,R} (f)(x)\equiv\sup\left\{|\pz_t\ast
f(x)|:\,t\in(0,1),\,\pz\in\cd_{N,\,R}(\rn)\r\}.
\end{eqnarray}
\end{definition}

For convenience's sake, when $R=1$, we denote $\cd_{N,\,R}(\rn)$,
$\wz{\cg}_{N,\,R} (f)$ and $\cg_{N,\,R}(f)$ simply by
$\cd^0_{N}(\rn)$, $\wz{\cg}^0_{N}(f)$ and $\cg^0_{N}(f)$,
respectively; when $R=2^{3(10+n)}$, we denote $\cd_{N,\,R}(\rn)$,
$\wz{\cg}_{N,\,R} (f)$ and $\cg_{N,\,R}(f)$ simply by
$\cd_{N}(\rn)$, $\wz{\cg}_{N}(f)$ and $\cg_{N}(f)$, respectively.
For any $N\in\zz_+$ and $x\in\rn$, obviously,
$$\cg^{0}_N(f)(x)\le\cg_N (f)(x)\le\wz{\cg}_N (f)(x).$$

For the local grand maximal function $\cg^0_N (f)$, we have the
following Proposition \ref{p3.2}, which is just \cite[Proposition
2.2]{Ta1}.

\begin{proposition}\label{p3.2}
Let $N\ge2$.

$\mathrm{(i)}$ Then there exists a positive constant $C$ such that
for all $f\in(L^1_{\loc}(\rn)\cap\cd'(\rn))$ and almost every
$x\in\rn$,
$$|f(x)|\le \cg^0_N (f)(x)\le M^{\loc}(f)(x).$$

$\mathrm{(ii)}$ If $\oz\in A_p^{\loc}(\rn)$ with $p\in(1,\fz)$, then
$f\in L^p_{\oz}(\rn)$ if and only if $f\in \cd'(\rn)$ and $\cg^0_N
(f)\in L^p_{\oz}(\rn)$; moreover,
$$\|f\|_{L^p_{\oz}(\rn)}\sim\|\cg^0_N (f)\|_{L^p_{\oz}(\rn)}.$$

$\mathrm{(iii)}$ If $\oz\in A_1^{\loc}(\rn)$, then $\cg^0_N$ is
bounded from $L^1_{\oz}(\rn)$ to $L^{1,\fz}_{\oz}(\rn)$.
\end{proposition}

Now we introduce the weighted local Orlicz-Hardy space via the local
grand maximal function as follows.

\begin{definition}\label{d3.3}
Let $\bfai$ satisfy Assumption (A), $\oz\in A^{\loc}_{\fz}(\rn)$,
$q_{\oz}$ and $p_{\bfai}$ be respectively as in \eqref{2.4} and
\eqref{2.6}, and $\wz{N}_{\bfai,\,\oz}\equiv\lfz
n(\frac{q_\oz}{p_\bfai}-1)\rf+2$. For each $N\in\nn$ with $N\ge
\wz{N}_{\bfai,\,\oz}$, the {\it weighted local Orlicz-Hardy space}
is defined by
$$h^{\bfai}_{\oz,\,N}(\rn)\equiv\left\{f\in\cd'(\rn):\ \cg_N (f)\in
L^{\bfai}_{\oz}(\rn)\r\}.$$ Moreover, let
$\|f\|_{h^{\bfai}_{\oz,\,N}(\rn)}\equiv\|\cg_N
(f)\|_{L^{\bfai}_{\oz}(\rn)}$.
\end{definition}

We remark that when $\bfai(t)\equiv t^p$ for all $t\in(0,\fz)$ and
$p\in(0,1]$, $h^{\bfai}_{\oz,\,N}(\rn)$ above is the weighted local
Hardy space $h^p_{\oz,\,N}(\rn)$ introduced
 by Tang \cite{Ta1}. Obviously, for any integers $N_1$ and $N_2$
with $N_1\ge N_2\ge \wz{N}_{\bfai,\,\oz}$,
$$h^{\bfai}_{\oz,\,\wz{N}_{\bfai,\,\oz}}(\rn)\subset
h^{\bfai}_{\oz,\,N_2}(\rn) \subset h^{\bfai}_{\oz,\,N_1}(\rn),$$
and the inclusions are continuous. We also point out that  Theorem
\ref{t3.14} below further implies that
$$\|\cg^0_N
(f)\|_{L^{\bfai}_{\oz}(\rn)}\sim\|\wz{\cg}^0_N
(f)\|_{L^{\bfai}_{\oz}(\rn)}\sim\|\cg_N
(f)\|_{L^{\bfai}_{\oz}(\rn)}\sim\|\wz{\cg}_N
(f)\|_{L^{\bfai}_{\oz}(\rn)}$$
for all $N\in\nn$ with $N\ge
N_{\bfai,\,\oz}$ (see \eqref{3.25} for the definition of
$N_{\bfai,\,\oz}$).

Next, we introduce the weighted local atoms, via which, we introduce
the weighted atomic local Orlicz-Hardy space.
\begin{definition}\label{d3.4}
Let $\bfai$ satisfy Assumption $\mathrm{(A)}$, $\oz\in
A^{\loc}_{\fz}(\rn)$ and $q_{\oz}$, $\rz$ be respectively as in
\eqref{2.4} and \eqref{2.7}. A triplet $(\rz,\,q,\,s)_{\oz}$ is
called {\it admissible} if $q\in(q_{\oz},\fz]$, $s\in\zz_+$ and
$s\ge\lfz n(\frac{q_{\oz}}{p_{\bfai}}-1)\rf$. A function $a$ on $\rn$
is called a {\it $(\rz,\,q,\,s)_{\oz}$-atom} if there exists a cube
$Q\subset\rn$ such that

$\mathrm{(i)}$ $\supp(a)\subset Q$;

$\mathrm{(ii)}$ $\|a\|_{L^q_{\oz}(\rn)}\le[\oz(Q)]^{\frac{1}{q}
-1}[\rz(\oz(Q))]^{-1}$;

$\mathrm{(iii)}$ $\int_{\rn}a(x)x^{\az}\,dx=0$ for all
$\az\in\zz_+^n$ with $|\az|\le s$, when $l(Q)<1$.\\ Moreover, a
function $a$ on $\rn$ is called  a {\it
$(\rz,\,q)_{\oz}$-single-atom} with $q\in(q_{\oz},\fz]$, if
$$\|a\|_{L^q_{\oz}(\rn)}\le[\oz(\rn)]^{\frac{1}{q}-1}
[\rz(\oz(\rn))]^{-1}.$$
\end{definition}

We point out that when $\bfai(t)\equiv t^p$ for $t\in(0,\fz)$ and
$p\in(0,1]$, $(\rz,\,q,\,s)_{\oz}$-atoms and
$(\rz,\,q)_{\oz}$-single-atoms are respectively
$(p,\,q,\,s)_{\oz}$-atoms and $(p,\,q)_{\oz}$-single-atoms,
introduced by Tang \cite{Ta1}.

\begin{definition}\label{d3.5}
Let $\bfai$ satisfy Assumption (A), $\oz\in A^{\loc}_{\fz}(\rn)$,
$q_{\oz}$ and $\rz$ be respectively as in \eqref{2.4} and
\eqref{2.7}, and $(\rz,\,q,\,s)_{\oz}$ be admissible. The {\it
weighted atomic local Orlicz-Hardy space}
$h^{\rz,\,q,\,s}_{\oz}(\rn)$ is defined be the set of all
$f\in\cd'(\rn)$ satisfying that
$$f=\sum_{i=0}^{\fz}\lz_i a_i$$
in $\cd'(\rn)$, where $\{a_i\}_{i\in\nn}$ are
$(\rz,\,q,\,s)_{\oz}$-atoms with $\supp(a_i)\subset Q_i$, $a_0$ is a
$(\rz,\,q)_{\oz}$-single-atom, $\{\lz_i\}_{i\in\zz_+}\subset\cc$,
and
$$\sum_{i=1}^{\fz}\oz(Q_i)\bfai\left(\frac{|\lz_i|}{\oz(Q_i)\rz(\oz(Q_i))}\r)
+\oz(\rn)\bfai\left(\frac{|\lz_0|}{\oz(\rn)\rz(\oz(\rn))}\r)<\fz.$$
Moreover, letting
\begin{eqnarray*}
&&\blz(\{\lz_i a_i\}_i)\equiv \inf\left\{\lz>0:\
\sum_{i=1}^{\fz}\oz(Q_i)\bfai\left(\frac{|\lz_i|}
{\lz\oz(Q_i)\rz(\oz(Q_i))}\r)\r.\\
&&\hspace{8
em}+\oz(\rn)\bfai\left(\frac{|\lz_0|}{\lz\oz(\rn)\rz(\oz(\rn))}\r)\le1\Bigg\},
\end{eqnarray*}
the \emph{quasi-norm} of $f\in h^{\rz,\,q,\,s}_{\oz}(\rn)$ is defined by
$$\|f\|_{h^{\rz,\,q,\,s}_{\oz}(\rn)}\equiv\inf\left\{
\blz(\{\lz_i a_i\}_{i\in\zz_+})\r\},$$ where the infimum is taken
over all the decompositions of $f$ as above.
\end{definition}

\begin{remark}\label{r3.6}
(i) Notice that the definition of $\blz(\{\lz_i a_i\}_{i\in\zz_+})$
above is different from \cite{Vi87}. In fact, if $p\in(0,1]$ and
$\bfai(t)\equiv t^p$ for all $t\in(0,\fz)$, then $\blz(\{\lz_i
a_i\}_{i\in\zz_+})$ coincides with $(\sum_{i\in\zz_+}
|\lz_i|^p)^{1/p}$.

(ii) Let $\{\lz_i^k\}_{i,\,k}$ and $\{a_i^k\}_{i,\,k}$ satisfy
$\blz(\{\lz^k_i a^k_i\}_{i\in\zz_+})<\fz$, where $k=1,\,2$. By the
fact that $\bfai$ is subadditive and of strictly lower type
$p_{\bfai}$, we have
$$\left[\blz(\{\lz^1_i a^1_i,\ \lz^2_i a^2_i\}_{i\in\zz_+}
)\r]^{p_{\bfai}}\le\sum_{k=1}^2 \left[\blz(\{\lz^k_i
a^k_i\}_{i\in\zz_+})\r]^{p_{\bfai}}.$$

(iii) Since $\bfai$ is concave, it is of upper type 1. Thus, for any
$f\in h^{\rz,\,q,\,s}_{\oz}(\rn)$, there exist $\{a_i\}_{i\in\zz_+}$
and $\{\lz_i\}_{i\in\zz_+}$ as in Definition \ref{d3.5} such that
$$\sum_{i\in\zz_+} |\lz_i|\ls \blz(\{\lz_i a_i\}_{i\in\zz_+})\ls
\|f\|_{h^{\rz,\,q,\,s}_{\oz}(\rn)}.$$
\end{remark}

Next, we introduce some local vertical, tangential and nontangential
maximal functions, and then establish the characterizations of the
weighted local Orlicz-Hardy space $h^{\bfai}_{\oz,\,N}(\rn)$ on
these local maximal functions.
\begin{definition}\label{d3.7}
Let
\begin{eqnarray}\label{3.3}
\pz_0\in\cd(\rn)\,\, \text{with}\,\,\int_{\rn}\pz_0 (x)\,dx\neq0.
\end{eqnarray}
For $j\in\zz_+$, $A,\,B\in[0,\fz)$ and $y\in\rn$, let
$m_{j,\,A,\,B}(y)\equiv(1+2^j |y|)^A 2^{B|y|}$. The {\it local
vertical maximal function} $\pz_0^{+}(f)$ of $f$ associated to
$\pz_0$ is defined by setting, for all $x\in\rn$,
\begin{eqnarray}\label{3.4}
\pz_0^{+}(f)(x)\equiv \sup_{j\in\zz_+}|(\pz_0)_j \ast f(x)|,
\end{eqnarray}
the {\it local tangential Peetre-type  maximal function
$\pz^{\ast\ast}_{0,\,A,\,B}(f)$} of $f$ associated to $\pz_0$ is
defined by setting, for all $x\in\rn$,
\begin{eqnarray}\label{3.5}
\pz^{\ast\ast}_{0,\,A,\,B}(f)(x)\equiv\sup_{j\in\zz_+,\,y\in\rn}
\frac{|(\pz_0)_j \ast f(x-y)|}{m_{j,\,A,\,B}(y)}
\end{eqnarray}
and the {\it local nontangential maximal function
$(\pz_0)^{\ast}_{\triangledown}(f)$} of $f$ associated to $\pz_0$ is
defined by setting, for all $x\in\rn$,
\begin{eqnarray}\label{3.6}
(\pz_0)^{\ast}_{\triangledown}(f)(x)\equiv
\sup_{|x-y|<t<1}|(\pz_0)_t \ast f(y)|,
\end{eqnarray}
where and in what follows, for all $x\in\rn$,
$(\pz_0)_j(x)\equiv2^{jn}\pz_0 (2^j x)$ for all $j\in\zz_+$ and
$(\pz_0)_t (x)\equiv\frac{1}{t^n}\pz_0 (\frac{x}{t})$ for all
$t\in(0,\fz)$.
\end{definition}

Obviously, for any $x\in\rn$, we have
$$\pz_0^{+}(f)(x)\le(\pz_0)^{\ast}_{\triangledown}(f)(x)
\ls\pz^{\ast\ast}_{0,\,A,\,B}(f)(x).$$
We remark that the local
tangential Peetre-type maximal function
$\pz^{\ast\ast}_{0,\,A,\,B}(f)$ was introduced by Rychkov
\cite{R01}.

In order to establish the local vertical and the local nontangential
maximal function characterizations of $h^{\bfai}_{\oz,\,N}(\rn)$, we
first establish some relations in the norm of $L^{\bfai}_{\oz}(\rn)$
of the local maximal functions
$\pz^{\ast\ast}_{0,\,A,\,B}(f),\,\pz_0^{+}(f)$ and
$\wz{\cg}_{N,\,R}(f)$, which further imply the desired
characterizations. We begin with some technical lemmas.

\begin{lemma}\label{l3.8}
Let $\pz_0$ be as in \eqref{3.3} and $\pz(x)\equiv\pz_0
(x)-\frac{1}{2^n}\pz_0 (\frac{x}{2})$ for all $x\in\rn$. Then for
any given integer $L\in\zz_+$, there exist $\ez_0,\,\ez\in\cd(\rn)$
such that $L_{\ez}\ge L$ and
$$f=\ez_0 \ast\pz_0 \ast
f+\sum_{j=1}^{\fz}\ez_j\ast\pz_j\ast f$$
in $\cd'(\rn)$ for all $f\in\cd'(\rn)$.
\end{lemma}

Lemma \ref{l3.8} is just \cite[Theorem\,1.6]{R01}.

\begin{remark}\label{r3.9}
Let $\pz_0,\,\pz,\,\ez_0$ and $\ez$ be as in Lemma \ref{l3.8}. From
the proof of \cite[Theorem 1.6]{R01}, it is easy to deduce that for
any $j\in\zz_+$ and $f\in\cd'(\rn)$,
$$f=(\ez_0)_j\ast(\pz_0)_j \ast f+\sum_{k=j+1}^{\fz}
\ez_k\ast\pz_k \ast f$$ in $\cd'(\rn)$ (see also
\cite[(2.11)]{R01}). We omit the details.
\end{remark}

For $f\in L^1_{\loc}(\rn)$, $B\in[0,\fz)$ and $x\in\rn$, let
\begin{eqnarray}\label{3.7}
K_B f(x)\equiv\int_{\rn}|f(y)|2^{-B|x-y|}\,dy,
\end{eqnarray}
 where and in what follows,
the \emph{space} $L^1_{\loc}(\rn)$ denotes the set of all locally
integrable functions on $\rn$.
\begin{lemma}\label{l3.10}
Let $p\in(1,\fz)$, $q\in(1,\fz]$, and $\oz\in A^{\loc}_p (\rn)$.
Then there exists a positive constant $C$ such that for any sequence
$\{f^j\}_j$ of measurable functions,
\begin{eqnarray}\label{3.8}
\left\|\{M^{\loc}(f^j)\}_j\r\|_{L^p_{\oz}(l_q)}\le C
\left\|\{f^j\}_j\r\|_{L^p_{\oz}(l_q)},
\end{eqnarray}
where and in what follows,
$$\left\|\{f^j\}_j\r\|_{L^p_{\oz}(l_q)}\equiv\left\|\left\{\sum_{j}|f^j|^q\r\}^{1/q}
\r\|_{L^p_{\oz}(\rn)}.$$
Also, there exists positive constants $C$ and
$B_0\equiv B_0 (\oz,n)$ such that for all $B\ge B_0/p$,
\begin{eqnarray}\label{3.9}
\left\|\{K_B (f^j)\}_j\r\|_{L^p_{\oz}(l_q)}\le C
\left\|\{f^j\}_j\r\|_{L^p_{\oz}(l_q)}.
\end{eqnarray}
\end{lemma}

Lemma \ref{l3.10} is just \cite[Lemma\,2.11]{R01}. Moreover, from
the proof of \cite[Lemma\,2.11]{R01}, it is easy to deduce that
\eqref{3.8} also holds for $M^{\loc}_{\wz{C}}$ with any given
positive constant $\wz{C}$. In this case, the positive constant $C$
in Lemma \ref{l3.10} depends on $\wz{C}$.

\begin{lemma}\label{l3.11}
Let $\pz_0$ be as in \eqref{3.3} and $r\in(0,\fz)$. Then there
exists a positive constant $A_0$ depending only on the support of
$\pz_0$ such that for any $A\in(\max\{A_0,\frac{n}{r}\},\fz)$ and
$B\in[0,\fz)$, there exists a positive constant $C$, depending only
on $n,\,r,\,\pz_0,\,A$ and $B$, satisfying that for all
$f\in\cd'(\rn)$, $x\in\rn$ and $j\in\zz_+$,
\begin{eqnarray*}
\left[(\pz_0 )^{\ast}_{j,\,A,\,B}(f)(x)\r]^r&\le& C\sum_{k=j}^{\fz}
2^{(j-k)(Ar-n)}\left\{M^{\loc}(|(\pz_0 )_k \ast
f|^r)(x)\r.\\
&&+K_{Br}(|(\pz_0 )_k \ast f|^r)(x)\big\},
\end{eqnarray*}
where
$$(\pz_0)^{\ast}_{j,\,A,\,B}(f)(x)
\equiv\sup_{y\in\rn}\frac{|(\pz_0)_j \ast f(x-y)|}{m_{j,\,A,\,B}(y)}$$
for all $x\in\rn$.
\end{lemma}

\begin{proof}
Lemma \ref{l3.11} is a modified version of \cite[Lemma\,2.10]{R01},
and was essentially obtained by Rychkov in the proof of
\cite[Theorem\,2.24]{R01}. Let $\pz$ be as in Lemma \ref{l3.8}.
Indeed, Rychkov \cite{R01} showed Lemma \ref{l3.11} under the
assumption that $f\in\cs'_e$, namely, there exist positive constant
$A_f$ and nonnegative integer $N_f$ such that for all
$\gz\in\cd(\rn)$,
$$|\langle f,\gz\rangle|\le A_f \sup\left\{|\partial^{\az}\gz(x)|e^{N_f |x|}:
\ x\in\rn,\ \az\in\zz^n_+\ \text{and}\ |\az|\le N_f\r\},$$
which guarantees that for all $x\in\rn$ and
$j\in\zz_+$,
$$M_{A,\,B}(x,\,j)\equiv\sup_{k\ge j,\,y\in\rn}2^{(j-k)A}\frac{|\pz_k
\ast f(x-y)|}{m_{j,\,A,\,B}(y)}<\fz.$$
By \cite[Proposition\,2.3.4(a)]{Gr08}, we have that for
any $f\in\cd'(\rn)$, $M_{A,\,B}(x,\,j)<\fz$ for all $x\in\rn$ and
$j\in\zz_+$, provided $A>A_0$, where $A_0$ is a positive constant
depending only on the support of $\pz_0$. Thus, Lemma \ref{l3.11} is
true for all $f\in\cd'(\rn)$. This finishes the proof of Lemma
\ref{l3.11}.
\end{proof}

\begin{theorem}\label{t3.12}
Let $\bfai$ satisfy Assumption $\mathrm{(A)}$, $\oz\in
A^{\loc}_{\fz}(\rn)$, $R\in(0,\fz)$, $\pz_0,\,q_{\oz}$ and
$p_{\bfai}$ be respectively as in \eqref{3.3} , \eqref{2.4} and
$\eqref{2.6}$, $\pz_0^{+}(f),\, \pz^{\ast\ast}_{0,\,A,\,B}(f)$, and
$\wz{\cg}_{N,\,R} (f)$ be respectively as in \eqref{3.4},
\eqref{3.5} and \eqref{3.1}. Let
$$A_1\equiv\max\{A_0,\,nq_{\oz}/p_{\bfai}\},$$
$B_1\equiv B_0/p_{\bfai}$
and integer $N_0\equiv\lfz2A_1\rf+1$, where $A_0$ and $B_0$ are
respectively as in Lemmas \ref{3.3} and \ref{3.2}. Then for any
$A\in(A_1,\fz)$, $B\in(B_1,\fz)$ and integer $N\ge N_0$, there
exists a positive constant $C$, depending only on
$A,\,B,\,N,\,R,\,\pz_0,\,\bfai,\,\oz$ and $n$, such that for all
$f\in\cd'(\rn)$,
 \begin{eqnarray}\label{3.10}
 \left\|\pz^{\ast\ast}_{0,\,A,\,B}(f)\r\|_{L^{\bfai}_{\oz}(\rn)}\le C
 \left\|\pz_0^{+}(f)\r\|_{L^{\bfai}_{\oz}(\rn)},
 \end{eqnarray}
 and
\begin{eqnarray}\label{3.11}
\left\|\wz{\cg}_{N,\,R} (f)\r\|_{L^{\bfai}_{\oz}(\rn)}\le C
\left\|\pz_0^{+}(f)\r\|_{L^{\bfai}_{\oz}(\rn)}.
\end{eqnarray}
\end{theorem}

\begin{proof}
Let $f\in\cd'(\rn)$. First, we prove \eqref{3.10}. Let
$A\in(A_1,\fz)$ and $B\in(B_1,\fz)$. By $A_1
\equiv\max\{A_0,\,nq_{\oz}/p_{\bfai}\}$ and $B_1\equiv B_0
/p_{\bfai}$, we know that there exists
$r_0\in(0,\frac{p_{\bfai}}{q_{\oz}})$ such that $A>\frac{n}{r_0}$
and $Br_0>\frac{B_0}{q_{\oz}}$,  where $A_0$ and $B_0$ are
respectively as in Lemmas \ref{3.3} and \ref{l3.10}. Thus, by Lemma
\ref{l3.11}, for all $x\in\rn$, we have
\begin{eqnarray}\label{3.12}
\left[(\pz_0)^{\ast}_{j,\,A,\,B}(f)(x)\r]^{r_0}&\ls&\sum_{k=j}^{\fz}
2^{(j-k)(Ar_0 -n)}\Big\{M^{ \loc}\left(|(\pz_0)_k\ast f|^{r_0}\r)(x)
\nonumber\\
 &&+K_{Br_0}\left(|(\pz_0)_k\ast f|^{r_0}\r)(x)\Big\}.
\end{eqnarray}
 Let $\pz^{+}_0 (f)$ and $\pz^{\ast\ast}_{0,\,A,\,B}(f)$ be
respectively as in \eqref{3.4} and \eqref{3.5}. We notice that for
any $x\in\rn$ and $k\in\zz_+$,
$$|(\pz_0)_k \ast f (x)|\le\pz^{+}_0(f)(x),$$
which together with \eqref{3.12} implies that for all
$x\in\rn$,
\begin{eqnarray}\label{3.13}
\left[\pz^{\ast\ast}_{0,\,A,\,B}(f)(x)\r]^{r_0}\ls
M^{\loc}\left([\pz^{+}_0 (f)]^{r_0})(x)+K_{Br_0}([\pz^{+}_0
(f)]^{r_0}\r)(x).
\end{eqnarray}
By \eqref{3.12} and the subadditivity of $\bfai$, we have
\begin{eqnarray}\label{3.14}
&&\int_{\rn}\bfai\left(\pz^{\ast\ast}_{0,\,A,\,B}(f)(x)\r)\oz(x)\,dx\nonumber\\
&&\hs\ls\int_{\rn}\bfai\left(\left\{M^{\loc}\left([\pz^{+}_0 (f)
]^{r_0}\r)(x)\r\}^{1/{r_0}}\r)\oz(x)\,dx\nonumber\\
&&\hs\hs+\int_{\rn}\bfai\left(\left\{K_{Br_0}\left([\pz^{+}_0
(f)]^{r_0}\r)(x)\r\}^{1/{r_0}}\r)\oz(x)\,dx
\equiv\mathrm{I_1}+\mathrm{I_2}.
\end{eqnarray}

First, we estimate $\mathrm{I_1}$.  By $r_0
<\frac{p_{\bfai}}{q_{\oz}}$, we know that there exists
$q\in(q_{\oz},\fz)$ such that $r_0 q<p_{\bfai}$ and $\oz\in
A^{\loc}_q (\rn)$. For any $\az\in(0,\fz)$ and $g\in L^1_{\loc}
(\rn)$, let
$$g=g \chi_{\{x\in\rn:\,\,|g(x)|\le\az\}}+g
\chi_{\{x\in\rn:\,\,|g(x)|>\az\}}\equiv g_1 +g_2.$$
It is easy to see that
$$\left\{x\in\rn:\,M^{\loc}(g)(x)>2\az\r\}\subset
\left\{x\in\rn:\,M^{\loc}(g_2)(x)>\az\r\},$$
which together with Lemma \ref{l2.3}(vi) implies that
\begin{eqnarray}\label{3.15}
&&\oz\left(\left\{x\in\rn:\,M^{\loc}(g)(x)>2\az\r\}\r)\nonumber\\   &&\hs
\le\oz\left(\left\{x\in\rn:\,M^{\loc}(g_2)(x)>\az\r\}\r)
\le\frac{1}{\az^q}\int_{\rn} \left[M^{\loc}(g_2)(x)\r]^q\oz(x)\,dx\nonumber\\
 &&\hs\ls\frac{1}{\az^q}\int_{\rn}|g_2 (x)|^q\oz(x)\,dx
\sim\frac{1}{\az^q}\int_{\{x\in\rn:\,|g(x)|>\az\}}|g(x)|^q\oz(x)\,dx.
\end{eqnarray}
Thus, for any $\az\in(0,\fz)$, by \eqref{3.15}, we have
\begin{eqnarray}\label{3.16}
&&\oz\left(\left\{x\in\rn:\,\left[M^{\loc}\left([\pz^+_0
(f)]^{r_0}\r)(x)\r]^{1/{r_0}}>\az\r\}\r)\nonumber\\  &&
\hs\ls\frac{1}{\az^{r_0 q}}\int_{\{x\in\rn:\,[\pz^{+}_0
(f)(x)]^{r_0}>\frac{\az^{r_0}}{2}\}}[\pz^{+}_0 (f)(x)]^{r_0
q}\oz(x)\,dx\nonumber\\  &&\hs\sim\sz_{\pz^+_0
(f)}\left(\frac{\az}{2^{1/{r_0}}}\r)+\frac{1}{\az^{r_0
q}}\int^{\fz}_{\frac{\az}{2^{1/{r_0}}}}r_0 qs^{r_0 q-1}\sz_{\pz^+_0
(f)}(s)\,ds,
\end{eqnarray}
where and in what follows,
$$\sz_{\pz^+_0(f)}(s)\equiv\oz(\{x\in\rn:\,\pz^+_0 (f)(x)>s\}).$$
From the fact that $\bfai$ is concave and of lower type
$p_{\bfai}$, we infer that
$\bfai(t)\sim\int_0^t \frac{\bfai(s)}{s}\,ds$ for all $t\in(0,\fz)$.
By this, \eqref{3.16} and the lower type $p_{\bfai}$ property of
$\bfai$, the fact $r_0 q<p_{\bfai}$ and Fubini's theorem, we have
\begin{eqnarray}\label{3.17}
\hs\mathrm{I_1}&\sim&\int_{\rn}\left\{\int^{\left\{M^{\loc}([\pz^{+}_0
(f)]^{r_0})(x)\r\}^{1/{r_0}}}_0\frac{\bfai(t)}{t}\,dt\r\}\oz(x)\,dx
\nonumber\\
&\sim&\int^{\fz}_0\frac{\bfai(t)}{t}\sz_{\left\{M^{\loc}([\pz^{+}_0
(f)]^{r_0})\r\}^{1/{r_0}}}(t)\,dt\nonumber\\
&\ls&\int^{\fz}_0\frac{\bfai(t)}{t}\left\{\sz_{\pz^+_0
(f)}\left(\frac{t}{2^{1/{r_0}}}\r)+\frac{1}{t^{r_0 q}}
\int^{\fz}_{\frac{t}{2^{1/{r_0}}}}r_0 qs^{r_0 q-1}\sz_{\pz^{+}_0
(f)}(s)\,ds\r\}\,dt\nonumber\\  &\ls& \mathrm{J}_f+\int^{\fz}_0 r_0
q s^{r_0 q-1}\sz_{\pz^{+}_0 (f)}(s) \left\{\int_0^{2^{\frac{1}{r_0}}
s}\frac{\bfai(t)}{t}\frac{1}{t^{r_0 q}}\,dt\r\}\,ds\nonumber\\
&\sim&\mathrm{J}_f +\int^{\fz}_0 r_0 qs^{r_0 q-1}\sz_{\pz^{+}_0
(f)}(s)\bfai(2^{\frac{1}{r_0}} s)\left\{\int_0^{2^{\frac{1}{r_0}}
s}\left(\frac{t}{2^{\frac{1}{r_0}} s}\r) ^{p_{\bfai}}\frac{1}{t^{r_0
q+1}}\,dt\r\}\,ds\nonumber\\
&\sim&\mathrm{J}_f\sim\int_{\rn}\bfai\left(\pz^+_0 (f)(x)\r)\oz(x)\,dx,
\end{eqnarray}
where $\mathrm{J}_f\equiv\int^{\fz}_0
\frac{\bfai(t)}{t}\sz_{\pz^{+}_0 (f)}(t)\,dt$.

Next, we estimate $\mathrm{I_2}$. For any $\az\in(0,\fz)$ and $g\in
L^1_{\loc} (\rn)$, let $g_1$ and $g_2$ be as above. For
$H\in[\frac{B_0}{q},\fz)$, let $\int_{\rn}2^{-H|x-y|}\,dy\equiv
c_H$.  It is easy to see that for all $x\in\rn$, $K_H (g_1) (x)\le
c_H \az$, which implies that
$$\left\{x\in\rn:\,K_H (g)(x)>(c_H +1)\az\r\}\subset\left\{x\in\rn:\,K_H
(g_2) (x)>\az\r\},$$ where $K_H$ is as in \eqref{3.7}. Thus, by
Lemma \ref{l3.10}, we have
\begin{eqnarray*}
\oz\left(\{x\in\rn:\,K_H g
(x)>(c_H +1)\az\}\r)&\le&\oz\left(\{x\in\rn:\,K_H g_2 (x)>\az\}\r)\\
&\ls&\frac{1}{\az^q}\int_{\{x\in\rn:\,|g(x)|>\az\}}|g(x)|^q\oz(x)\,dx.
\end{eqnarray*}
Similarly to \eqref{3.16}, from the above estimate,
$Br_0>\frac{B_0}{q}$ and Lemma \ref{3.2}, we also deduce that
\begin{eqnarray*}
&&\oz\left(\left\{x\in\rn:\,\left[K_{Br_0}([\pz^+_0
(f)]^{r_0})(x)\r]^{1/{r_0}}>\az\r\}\r)\\
&&\hs\ls\sz_{\pz^+_0
(f)}\left(\frac{\az}{(c_{Br_0}+1)^{1/r_0}}\r)+\frac{1}{\az^{r_0
q}}\int^{\fz}_{\frac{\az}{(c_{Br_0}+1)^{1/r_0}}}r_0 qs^{r_0
q-1}\sz_{\pz^+_0 (f)}(s)\,ds.
\end{eqnarray*}
By this, similarly to the estimate of $\mathrm{I_1}$, we also have
\begin{eqnarray}\label{3.18}
\mathrm{I_2}\ls\int_{\rn}\bfai\left(\pz^+_0 (f)(x)\r)\oz(x)\,dx.
\end{eqnarray}

Thus, we deduce from \eqref{3.14}, \eqref{3.17} and \eqref{3.18} that
$$\int_{\rn}\bfai\left(\pz^{\ast\ast}_{0,\,A,\,B}(f)(x)\r)\oz(x)\,dx\ls
\int_{\rn}\bfai\left(\pz^+_0 (f)(x)\r)\oz(x)\,dx.$$ Replacing $f$ by
$f/\lz$ with $\lz\in(0,\fz)$ in the above inequality, and
noticing that
$$\bfai(\pz^{\ast\ast}_{0,\,A,\,B}(f/\lz))
=\bfai(\pz^{\ast\ast}_{0,\,A,\,B}(f)/\lz)$$ and $\bfai(\pz^+_0
(f/\lz))=\bfai(\pz^+_0 (f)/\lz)$, we have
\begin{eqnarray}\label{3.19}
\int_{\rn}\bfai\left(\pz^{\ast\ast}_{0,\,A,\,B}(f)(x)/\lz\r)\oz(x)\,dx\ls
\int_{\rn}\bfai\left(\pz^+_0 (f)(x)/\lz\r)\oz(x)\,dx,
\end{eqnarray}
which together with the arbitrariness of $\lz\in(0,\fz)$ implies
\eqref{3.10}.

Now, we prove \eqref{3.11}. By $N_0 \equiv\lfz2A_1\rf+1$, we know
that there exists $A\in(A_1,\fz)$ such that $2A<N_0$. In the rest of
this proof, we fix $A\in(A_1,\fz)$ satisfying $2A<N_0$ and
$B\in(B_1,\fz)$. Let integer $N\ge N_0$ and $R\in(0,\fz)$. For any
$\gz\in\cd_{N,\,R}(\rn)$, $t\in(0,1)$ and $j\in\zz_+$, from Lemma
\ref{l3.8} and Remark \ref{r3.9}, it follows that
\begin{eqnarray}\label{3.20}
\gz_t \ast f=\gz_t \ast(\ez_0)_j\ast(\pz_0)_j\ast
f+\sum_{k=j+1}^{\fz}\gz_t \ast\ez_k\ast\pz_k \ast f,
\end{eqnarray}
where $\ez_0,\,\ez\in\cd(\rn)$ with $L_{\ez}\ge N$ and $\pz$ is as
in Lemma \ref{l3.8}.

For any given $t\in(0,1)$ and $x\in\rn$, let $2^{-j_0-1}\le
t<2^{-j_0}$ for some $j_0\in\zz_+$, and $z\in\rn$ satisfy $|z-x|<t$.
Then, by \eqref{3.20}, we have
\begin{eqnarray}\label{3.21}
|\gz_t \ast f(z)|&\le&\left|\gz_t \ast(\ez_0)_{j_0}\ast(\pz_0)_{j_0}
\ast f(z)\r|+\sum^{\fz}_{k=j_0+1}\left|\gz_t \ast\ez_k\ast\pz_k \ast f(z)\r|
\nonumber\\
 &\le&\int_{\rn}\left|\gz_t \ast(\ez_0)_{j_0}
(y)\r|\left|(\pz_0)_{j_0}\ast f(z-y)\r|\,dy\nonumber\\
&&+\sum_{k=j_0+1}^{\fz}\int_{\rn}\left|\gz_t \ast\ez_k
(y)\r|\left|\pz_k\ast f(z-y)\r|\,dy \equiv\mathrm{I_3}+\mathrm{I_4}.
\end{eqnarray}

To estimate $\mathrm{I_3}$, from
\begin{eqnarray*}
\pz^{\ast\ast}_{0,\,A,\,B}(f)(x)&=&\sup_{j\in\zz_+,\,y\in\rn}
\frac{|(\pz_0)_j\ast
f(x-y)|}{m_{j,\,A,\,B}(y)}\\
&=&\sup_{j\in\zz_+,\,y\in\rn}\frac{|(\pz_0)_j\ast
f(x-(y+x-z))|}{m_{j,\,A,\,B}(y+x-z)}\\
&=& \sup_{j\in\zz_+,\,y\in\rn}\frac{|(\pz_0)_j\ast
f(z-y)|}{m_{j,\,A,\,B}(y+x-z)},
\end{eqnarray*}
we infer that
$$\left|(\pz_0)_{j_0}\ast f(z-y)\r|\le\pz^{\ast\ast}_{0,\,A,\,B}(f)(x)
m_{j_0,\,A,\,B}(y+x-z),$$
which, together with the facts that
$$m_{j_0,\,A,\,B}(y+x-z)\le m_{j_0,\,A,\,B}(x-z)m_{j_0,\,A,\,B}(y)$$
and $m_{j_0,\,A,\,B}(x-z)\ls2^{A}$, implies that
$$|(\pz_0)_{j_0}\ast f(z-y)|\ls2^{A}\pz^{\ast\ast}_{0,\,A,\,B}(f)(x)
m_{j_0,\,A,\,B}(y).$$
Thus, we have
\begin{eqnarray*}
\mathrm{I_3}\ls2^{A}\left\{\int_{\rn}|\gz_t \ast(\ez_0)_{j_0}
(y)|m_{j_0,\,A,\,B}(y)\,dy\r\}\pz^{\ast\ast}_{0,\,A,\,B}(f)(x).
\end{eqnarray*}

To estimate $\mathrm{I_4}$, by the definition of $\pz$, it is easy
to know that for any $k\in\nn$,
$$\left|\pz_k\ast
f(z-y)\r|\le\left|(\pz_0)_k\ast f(z-y)\r|+\left|(\pz_0)_{k-1}\ast
f(z-y)\r|.$$
By the definition of $\pz^{\ast\ast}_{0,\,A,\,B}(f)$
and the facts that
$$m_{k,\,A,\,B}(y+x-z)\le
m_{k,\,A,\,B}(x-z)m_{k,\,A,\,B}(y)$$
for any $k\in\nn$ and
$m_{k,\,A,\,B}(x-z)\ls2^{(k-j_0)A}$, we conclude that
\begin{eqnarray*}
|(\pz_0)_k\ast f(z-y)|&\le&
\pz^{\ast\ast}_{0,\,A,\,B}(f)(x)m_{k,\,A,\,B}(y+x-z)\\
&\le&\pz^{\ast\ast}_{0,\,A,\,B}(f)(x)
m_{k,\,A,\,B}(x-z)m_{k,\,A,\,B}(y)\\
&\ls&2^{(k-j_0)A} m_{k,\,A,\,B}(y)\pz^{\ast\ast}_{0,\,A,\,B}(f)(x).
\end{eqnarray*}
Similarly, we also have
$$|(\pz_0)_{k-1}\ast f(z-y)|\ls2^{(k-j_0)A}
m_{k,\,A,\,B}(y)\pz^{\ast\ast}_{0,\,A,\,B}(f)(x).$$
Thus,
\begin{eqnarray*}
\mathrm{I_4}\ls\sum_{k=j_0+1}^{\fz}2^{(k-j_0)A}\left\{\int_{\rn}|\gz_t
\ast\ez_k
(y)|m_{k,\,A,\,B}(y)\,dy\r\}\,\pz^{\ast\ast}_{0,\,A,\,B}(f)(x).
\end{eqnarray*}

From \eqref{3.21} and the above estimates of $\mathrm{I_3}$ and
$\mathrm{I_4}$, it follows that
\begin{eqnarray}\label{3.22}
|\gz_t \ast f(z)|&\ls&\left\{\int_{\rn}|\gz_t \ast(\ez_0)_{j_0}
(y)|m_{j_0,\,A,\,B}(y)\,dy\r.\nonumber\\
&&+\left.\sum_{k=j_0+1}^{\fz}2^{(k-j_0)A}\int_{\rn}|\gz_t \ast\ez_k
(y)|m_{k,\,A,\,B}(y)\,dy\r\}\pz^{\ast\ast}_{0,\,A,\,B}(f)(x).
\end{eqnarray}
Assume that $\supp(\ez_0)\subset B(0,R_0)$. Then
$\supp((\ez_0)_j)\subset B(0,2^{-j}R_0)$ for all $j\in\zz_+$.
Moreover, by $\supp(\gz)\subset B(0,R)$ and $2^{-j_0-1}\le
t<2^{-j_0}$, we see that
$$\supp(\gz_t)\subset B(0,2^{-j_0}R).$$
From this, we further deduce that
$\supp(\gz_t\ast(\ez_0)_{j_0})\subset B(0,2^{-j_0}(R_0
+R))$ and
\begin{eqnarray*}
|\gz_t\ast(\ez_0)_{j_0} (y)|\ls\int_{\rn}|\gz_t (s)||(\pz_0)_{j_0}
(y-s)|\,ds\ls2^{j_0 n}\int_{\rn}|\gz_t (s)|\,ds\sim2^{j_0 n},
\end{eqnarray*}
which implies that
\begin{eqnarray}\label{3.23}
&&\int_{\rn}|\gz_t\ast(\ez_0)_{j_0} (y)|m_{j_0,\,A,\,B}(y)\,dy\nonumber\\
 &&\hs\ls2^{j_0 n}\int_{B(0,2^{-j_0}(R_0 +R))}(1+2^{j_0}
|y|)^A 2^{B|y|}\,dy\ls1.
\end{eqnarray}
Moreover, since $\ez$ has vanishing moments up to order $N$, it was
proved in \cite[(2.13)]{R01} that
$$\|\gz_t\ast\ez_k\|_{L^{\fz}(\rn)}\ls2^{(j_0-k)N}2^{j_0 n}$$
for all $k\in\nn$ with $k\ge j_0 +1$, which, together with the facts that
$N>2A$ and
$$\supp(\gz_t\ast\ez_k) \subset B(0,2^{-j_0}R_0
+2^{-k}R),$$
implies that
\begin{eqnarray}\label{3.24}
&&\sum_{k=j_0+1}^{\fz}2^{(k-j_0)A}\int_{\rn}|\gz_t \ast\ez_k
(y)|m_{k,\,A,\,B}(y)\,dy\nonumber\\
 &&\hs\ls\sum_{k=j_0+1}^{\fz}2^{(k-j_0)A}2^{(j_0-k)N}2^{j_0
n} (2^{-j_0}R_0 +2^{-k}R)^n\nonumber\\  &&\hs\hspace{1
em}\times\left[1+2^k
(2^{-j_0}R_0+2^{-k}R)\r]^A 2^{(2^{-j_0}R_0+2^{-k}R)B}\nonumber\\
 &&\hs\ls\sum_{k=j_0 +1}^{\fz}2^{(j_0-k)(N-2A)}\ls1.
\end{eqnarray}
Thus, from \eqref{3.22}, \eqref{3.23} and \eqref{3.24}, we deduce
that $|\gz_t \ast f(z)|\ls\pz^{\ast\ast}_{0,\,A,\,B}(f)(x)$. Then,
by the arbitrariness of $t\in(0,1)$ and $z\in B(x,t)$, we know that
$$\wz{\cg}_{N,R}(f)(x)\ls\pz^{\ast\ast}_{0,\,A,\,B}(f)(x),$$
which together with \eqref{3.19} implies that for any $\lz\in(0,\fz)$,
$$\int_{\rn}\bfai\left(\wz{\cg}_{N,\,R}(f)(x)/\lz\r)\oz(x)\,dx\ls
\int_{\rn}\bfai\left(\pz^{+}_0 (f)(x)/\lz\r)\oz(x)\,dx.$$ From this, we
infer that \eqref{3.11} holds, which completes the proof of Theorem
\ref{t3.12}.
\end{proof}

\begin{remark}\label{r3.13}
Let $p\in (0,1]$. We point out that Theorem \ref{t3.12} when
$R\equiv1$ and $\bfai(t)\equiv t^p$ for all $t\in(0,\fz)$ was
obtained by Rychkov \cite[Theorem\,2.24]{R01}.
\end{remark}

As a corollary of Theorem \ref{3.1}, we immediately obtain that the
local vertical and the local nontangential maximal function
characterizations of $h^{\bfai}_{\oz,\,N}(\rn)$ with $N\ge
N_{\bfai,\,\oz}$ as follows. Here and in what follows,
\begin{eqnarray}\label{3.25}
N_{\bfai,\,\oz}\equiv\max\left\{\wz{N}_{\bfai,\,\oz},\,N_0\r\},
\end{eqnarray}
where $\wz{N}_{\bfai,\,\oz}$ and $N_0$ are  respectively as in
Definition \ref{d3.3} and Theorem \ref{t3.12}.

\begin{theorem}\label{t3.14}
Let $\bfai$ satisfy Assumption $\mathrm{(A)}$, $\oz\in
A^{\loc}_{\fz}(\rn)$, $\pz_0$ and $N_{\bfai,\,\oz}$ be respectively
as in \eqref{3.3} and \eqref{3.25}. Then for any integer $N\ge
N_{\bfai,\,\oz}$, the following are equivalent:

\begin{enumerate}

\item[\rm(i)] $f\in h^{\bfai}_{\oz,\,N}(\rn);$

\item[\rm(ii)] $f\in\cd'(\rn)$ and $\pz^{+}_0 (f)\in L^{\bfai}_{\oz}(\rn);$

\item[\rm(iii)] $f\in\cd'(\rn)$ and $(\pz_0)^{\ast}_{\triangledown}
(f)\in L^{\bfai}_{\oz}(\rn);$

\item[\rm(iv)] $f\in\cd'(\rn)$ and $\wz{\cg}_N (f)\in L^{\bfai}_{\oz}(\rn);$

\item[\rm(v)] $f\in\cd'(\rn)$ and $\wz{\cg}^0_N (f)\in L^{\bfai}_{\oz}(\rn);$

\item[\rm(vi)] $f\in\cd'(\rn)$ and $\cg^0_N (f)\in L^{\bfai}_{\oz}(\rn)$.

\end{enumerate}

Moreover,
\begin{eqnarray}\label{3.26}
\|f\|_{h^{\bfai}_{\oz,\,N}(\rn)}&\sim&\left\|\pz^{+}_0
(f)\r\|_{L^{\bfai}_{\oz}(\rn)}\sim
\left\|(\pz_0)^{\ast}_{\triangledown}(f)\r\|_{L^{\bfai}_{\oz}(\rn)}
\nonumber\\&\sim&\left\|\wz{\cg}_N
(f)\r\|_{L^{\bfai}_{\oz}(\rn)}\sim\left\|\wz{\cg}^0_N
(f)\r\|_{L^{\bfai}_{\oz}(\rn)}\sim\left\|\cg^0_N
(f)\r\|_{L^{\bfai}_{\oz}(\rn)},
\end{eqnarray}
where the implicit constants are independent of $f$.
\end{theorem}

\begin{proof}
$\mathrm{(i)\Rightarrow(ii)}$. Let integer $N\ge N_{\bfai,\,\oz}$
and $f\in h^{\bfai}_{\oz,\,N}(\rn)$. Let $\wz{\pz}_0$ satisfy
\eqref{3.3} and $\wz{\pz}_0\in\cd_N (\rn)$. Then from the definition
of $\cg_N (f)$, we infer that $\wz{\pz}^{+}_0 (f)\le\cg_N (f)$ and
hence $\wz{\pz}^{+}_0 (f)\in L^{\bfai}_{\oz}(\rn)$. For any $\pz_0$
satisfying \eqref{3.3}, assume that $\supp(\pz_0)\subset B(0,R)$.
Then, by \eqref{3.11} and the above argument, we have
$$\left\|\wz{\cg}_{N,\,R} (f)\r\|_{L^{\bfai}_{\oz}(\rn)}\ls
\left\|\wz{\pz}_0^{+}(f)\r\|_{L^{\bfai}_{\oz}(\rn)}\ls
\|f\|_{h^{\bfai}_{\oz,\,N}(\rn)},$$
which together with $\pz^{+}_0
(f)\ls\wz{\cg}_{N,\,R} (f)$ implies that $\pz^{+}_0 (f)\in
L^{\bfai}_{\oz}(\rn)$ and
$$\left\|\pz^{+}_0
(f)\r\|_{L^{\bfai}_{\oz}(\rn)}\ls\left\|f\r\|_{h^{\bfai}_{\oz,\,N}(\rn)}.$$

$\mathrm{(ii)\Rightarrow(iii)}$. Let $f\in\cd'(\rn)$ satisfy
$\pz_0^{+}(f) \in L^{\bfai}_{\oz}(\rn)$, where $\pz_0$ is as in
\eqref{3.3}. Then from the fact that
$$\pz_0^{+}(f)
\le(\pz_0)^{\ast}_{\triangledown}(f)\ls\pz^{\ast\ast}_
{0,\,A,\,B}(f)$$
and \eqref{3.10}, we deduce that $(\pz_0)^{\ast}_
{\triangledown}(f)\in L^{\bfai}_{\oz}(\rn)$ and
$$\|(\pz_0)^{\ast}_{\triangledown}(f)\|_{L^{\bfai}_{\oz}(\rn)}\ls\|\pz^{+}_0
(f)\|_{L^{\bfai}_{\oz}(\rn)}.$$

$\mathrm{(iii)\Rightarrow(iv)}$. Let $f\in\cd'(\rn)$ satisfy
$(\pz_0)^{\ast}_ {\triangledown}(f)\in L^{\bfai}_{\oz}(\rn)$, where
$\pz_0$ is as in \eqref{3.3}. By \eqref{3.11},
$$\|\wz{\cg}_{N}
(f)\|_ {L^{\bfai}_{\oz}(\rn)}\ls
\|\pz_0^{+}(f)\|_{L^{\bfai}_{\oz}(\rn)},$$
which together with the fact that
$$\pz_0^{+}(f) \le(\pz_0)^{\ast}_{\triangledown}(f)$$
and the assumption that $(\pz_0)^{\ast}_ {\triangledown}(f)\in
L^{\bfai}_{\oz}(\rn)$ implies $\wz{\cg}_{N} (f) \in
L^{\bfai}_{\oz}(\rn)$ and
$$\left\|\wz{\cg}_N
(f)\r\|_{L^{\bfai}_{\oz}(\rn)}\ls\left\|(\pz_0)^{\ast}_{\triangledown}(f)
\r\|_{L^{\bfai}_{\oz}(\rn)}.$$

$\mathrm{(iv)\Rightarrow(v)\Rightarrow(vi)}$. By the facts that
$\cg^0_N (f)\le\wz{\cg}^0_N (f)\le\wz{\cg}_N (f)$ for any
$f\in\cd'(\rn)$ and $\bfai$ is increasing, we see that all the
conclusions hold. Moreover, it is obvious that
$$\left\|\cg^0_N
(f)\r\|_{L^{\bfai}_{\oz}(\rn)}\le\left\|\wz{\cg}^0_N
(f)\r\|_{L^{\bfai}_{\oz}(\rn)}\le\left\|\wz{\cg}_N
(f)\r\|_{L^{\bfai}_{\oz}(\rn)}.$$

$\mathrm{(vi)\Rightarrow(i)}$. Let $f\in\cd'(\rn)$ satisfy
$$\cg^0_N(f)\in L^{\bfai}_{\oz}(\rn).$$
Let $\pz_1$ satisfy \eqref{3.3} and
$\pz_1\in\cd^0_N (\rn)$. Then by \eqref{3.10}, we have that
$$\left\|\wz{\cg}_{N} (f)\r\|_ {L^{\bfai}_{\oz}(\rn)}\ls
\|\pz_1^{+}(f)\|_{L^{\bfai}_{\oz}(\rn)},$$
which together with the
facts that $\pz_1^{+}(f)\le\cg^0_{N} (f)$
and $\cg_{N} (f)\le\wz{\cg}_{N} (f)$ implies that
$$\left\|\cg_{N} (f)\r\|_
{L^{\bfai}_{\oz}(\rn)}\ls\left\|\cg^0_{N} (f)\r\|_
{L^{\bfai}_{\oz}(\rn)}.$$ Thus, by the definition of
$h^{\bfai}_{\oz,\,N}(\rn)$, we know that $f\in h^{\bfai}_
{\oz,\,N}(\rn)$ and
$$\|f\|_ {h^{\bfai}_{\oz,\,N}(\rn)}\ls\|\cg^0_{N}
(f)\|_ {L^{\bfai}_{\oz}(\rn)},$$
which completes the proof of Theorem \ref{t3.14}.
\end{proof}

As a corollary of Theorems \ref{t3.12} and \ref{t3.14}, we have the
following local tangential maximal function characterization of
$h^{\bfai}_{\oz,\,N}(\rn)$. We omit the details.
\begin{corollary}\label{c3.15}
Let $\bfai$ satisfy Assumption $\mathrm{(A)}$, $\pz_0$ be as in
\eqref{3.3}, $\oz\in A^{\loc}_{\fz}(\rn)$, $N_{\bfai,\,\oz}$ be as
in \eqref{3.25}, $A$ and $B$ be as in Theorem \ref{t3.12}. Then for
integer $N\ge N_{\bfai,\,\oz}$,
$$f\in h^{\bfai}_{\oz,\,N}(\rn)$$
if and only if $f\in\cd'(\rn)$ and
$\pz^{\ast\ast}_{0,\,A,\,B}(f)\in{L^{\bfai}_{\oz}(\rn)}$; moreover,
$$\|f\|_{h^{\bfai}_{\oz,\,N}(\rn)}\sim
\|\pz^{\ast\ast}_{0,\,A,\,B}(f)\|_{L^{\bfai}_{\oz}(\rn)}.$$
\end{corollary}

Next, we give some basic properties concerning
$h^{\bfai}_{\oz,\,N}(\rn)$ and $h^{\rz,\,q,\,s}_{\oz}(\rn)$.

\begin{proposition}\label{p3.16}
Let $\bfai$ satisfy Assumption $\mathrm{(A)}$, $\oz\in
A^{\loc}_{\fz}(\rn)$ and $N_{\bfai,\,\oz}$ be as in \eqref{3.25}. If
integer $N\ge N_{\bfai,\,\oz}$, then the inclusion
$h^{\bfai}_{\oz,\,N}(\rn)\hookrightarrow\cd'(\rn)$ is continuous.
\end{proposition}

\begin{proof}
Let $f\in h^{\bfai}_{\oz,\,N}(\rn)$. For any given
$\phi\in\cd(\rn)$, assume that $\supp(\phi)\subset B(0,R)$ with
$R\in(0,\fz)$. Then we have
\begin{eqnarray}\label{3.27}
|\langle f,\phi\rangle|=\left|f\ast\wz{\phi}(0)\r|\le\left\|\wz{\phi}\r\|_{\cd_{N,\,R}
(\rn)}\inf_{x\in B(0,1)}\wz{\cg}_{N,\,R} (f)(x),
\end{eqnarray}
where $\wz{\cg}_{N,\,R} (f)$ is as in \eqref{3.2} and
$\wz{\phi}(x)\equiv\phi(-x)$ for all $x\in\rn$. Now, to prove
Proposition \ref{p3.16}, we consider the following two cases for
$\|f\|_{h^{\bfai}_{\oz,\,N}(\rn)}$.

{\it Case (i)} $\|f\|_{h^{\bfai}_{\oz,\,N}(\rn)}\ge1$. In this case,
by the upper type 1 property of $\bfai$ and Theorems \ref{t3.12} and
\ref{t3.14}, we obtain
\begin{eqnarray}\label{3.28}
&&\int_{\rn}\bfai\left(\wz{\cg}_{N,\,R} (f)(x)\r)\oz(x)\,dx\nonumber \\
&&\hs\ls \|f\|_{h^{\bfai}_{\oz,\,N}(\rn)}
\int_{\rn}\bfai\left(\frac{\wz{\cg}_{N,\,R} (f)(x)}
{\|f\|_{h^{\bfai}_{\oz,\,N}(\rn)}}\r)\oz(x)\,dx
\ls\|f\|_{h^{\bfai}_{\oz,\,N}(\rn)}.
\end{eqnarray}
Notice that the upper type 1 property of $\bfai$ implies that for
$t\in(0,1]$,
$$\bfai(1)=\bfai\left(t\frac{1}{t}\r)\ls\frac{1}{t}\bfai(t)$$
and hence $\bfai(t)\gs t$. Thus, when $\inf_{x\in
B(0,1)}\wz{\cg}_{N,\,R} (f)(x)\le1$, from \eqref{3.27} and
\eqref{3.28}, we deduce that
 \begin{eqnarray}\label{3.29}
|\langle f,\phi\rangle|&\ls&\|\phi\|_{\cd_{N,\,R} (\rn)}
\bfai\left(\inf_{x\in B(0,1)}\wz{\cg}_{N,\,R} (f)(x)\r)\nonumber\\
&\ls&\|\phi\|_{\cd_{N,\,R} (\rn)}\frac{1}{\oz(B(0,1))}
\int_{B(0,1)}\bfai\left(\wz{\cg}_{N,\,R} (f)(y)\r)\oz(y)\,dy\nonumber\\
 &\ls&\|\phi\|_{\cd_{N,\,R} (\rn)}\frac{1}{\oz(B(0,1))}
\|f\|_{h^{\bfai}_{\oz,\,N}(\rn)}.
\end{eqnarray}
Let $p_{\bfai}$ be as in \eqref{2.6}. Since $\bfai$ is lower type
$p_{\bfai}$, then for $t\in(1,\fz)$,
$$\bfai(1)=\bfai\left(t\frac{1}{t}\r)\ls\frac{1}{t^{p_{\bfai}}}\bfai(t)$$
and hence $t\ls[\bfai(t)]^{1/p_{\bfai}}$. Thus, when $\inf_{x\in
B(0,1)}\wz{\cg}_{N,\,R} (f)(x)>1$, by \eqref{3.26} and \eqref{3.27},
we conclude that
\begin{eqnarray}\label{3.30}
|\langle f,\phi\rangle|&\ls&\|\phi\|_{\cd_{N,\,R} (\rn)}
\left\{\bfai\left(\inf_{x\in B(0,1)}\wz{\cg}_{N,\,R} (f)(x)\r)\r\}
^{1/p_{\bfai}}\nonumber\\  &\ls&\|\phi\|_{\cd_{N,\,R}
(\rn)}[\oz(B(0,1))]^{-1/p_{\bfai}}\nonumber\\
&&\times\left\{\int_{B(0,1)}\bfai\left(\wz{\cg}_{N,\,R}
(f)(y)\r)\oz(y)\,dy\r\}^{1/p_{\bfai}}\nonumber\\
&\ls&\|\phi\|_{\cd_{N,\,R} (\rn)}[\oz(B(0,1))]^{-1/p_{\bfai}}
\|f\|_{h^{\bfai}_{\oz,\,N}(\rn)}^{1/p_{\bfai}}.
\end{eqnarray}

{\it Case (ii)} $\|f\|_{h^{\bfai}_{\oz,\,N}(\rn)}<1$. In this case,
by the lower type $p_{\bfai}$ property of $\bfai$ and Theorems
\ref{t3.12} and \ref{t3.14}, we see that
\begin{eqnarray*}
&&\int_{\rn}\bfai\left(\wz{\cg}_{N,\,R} (f)(x)\r)\oz(x)\,dx\\
&&\hs\ls \|f\|^{p_{\bfai}}_{h^{\bfai}_{\oz,\,N}(\rn)}
\int_{\rn}\bfai\left(\frac{\wz{\cg}_{N,\,R} (f)(x)}
{\|f\|_{h^{\bfai}_{\oz,\,N}(\rn)}}\r)\oz(x)\,dx
\ls\|f\|^{p_{\bfai}}_{h^{\bfai}_{\oz,\,N}(\rn)}.
\end{eqnarray*}
Thus, from  this fact and \eqref{3.27}, similarly to the proof of
\eqref{3.29} and \eqref{3.30}, we infer that if $\inf_{x\in
B(0,1)}\wz{\cg}_{N,\,R} (f)(x)\le1$, then
$$|\langle f,\phi\rangle|\ls\|\phi\|_{\cd_{N,\,R}
(\rn)}\frac{1}{\oz(B(0,1))} \|f\|^{p_{\bfai}}
_{h^{\bfai}_{\oz,\,N}(\rn)},$$ and if $\inf_{x\in
B(0,1)}\wz{\cg}_{N,\,R} (f)(x)>1$, then
$$|\langle f,\phi\rangle|\ls\|\phi\|_{\cd_{N,\,R}
(\rn)}[\oz(B(0,1))]^{-1/p_{\bfai}}
\|f\|_{h^{\bfai}_{\oz,\,N}(\rn)}.$$ Thus, $f\in\cd'(\rn)$ and the
inclusion is continuous, which completes the proof of Proposition
\ref{p3.16}.
\end{proof}

\begin{proposition}\label{p3.17}
Let $\bfai$ satisfy Assumption $\mathrm{(A)}$, $\oz\in
A^{\loc}_{\fz}(\rn)$ and $N_{\bfai,\,\oz}$ be as in \eqref{3.25}. If
integer $N\ge N_{\bfai,\,\oz}$, then the space
$h^{\bfai}_{\oz,\,N}(\rn)$ is complete.
\end{proposition}

\begin{proof}
For any $\pz\in\cd_N (\rn)$ and $\{f_i\}_{i\in\nn}\subset\cd'(\rn)$
such that $\{\sum_{i=1}^j f_i\}_{j\in\nn}$ converges in $\cd'(\rn)$
to a distribution $f$ as $j\to\fz$, the series $\{\sum_{i=1}^j
f_i\ast\pz\}_{j\in\nn}$ converges to $f\ast\pz$ also pointwise as
$j\to\fz$. By Assumption $\mathrm{(A)}$, we know that $\bfai$ is
strictly increasing and subadditive, which together with the
continuity of $\bfai$ implies that  for all $x\in\rn$,
$$\bfai\left(\cg_N (f)(x)\r)\le\bfai\left(\sum_{i=1}^{\fz}\cg_N
(f_i)(x)\r)\le\sum_{i=1}^{\fz}\bfai\left(\cg_N (f_i)(x)\r).$$ If
$\sum_{i=1}^{\fz}\|f_i\|_{h^{\bfai}_{\oz,\,N}(\rn)}^{p_{\bfai}}<\fz$,
letting $\lz_i =\|f_i\|_{h^{\bfai}_{\oz,\,N}(\rn)}^{p_{\bfai}}$,
then by the strictly lower type $p_{\bfai}$ property of $\bfai$ and
the levi lemma, we know that
\begin{eqnarray*}
&&\int_{\rn}\bfai\left(\frac{\cg_N
(f)(x)}{(\sum_{j=1}^{\fz}\lz_j)^{1/p_{\bfai}}}\r)\oz(x)\,dx\\
&&\hs\le \sum_{i=1}^{\fz}\int_{\rn}\bfai\left(\frac{\cg_N
(f_i)(x)}{(\sum_{j=1}^{\fz}\lz_j)^{1/p_{\bfai}}}\r)\oz(x)\,dx\\
&&\hs\le\sum_{i=1}^{\fz}\frac{\lz_i}{\sum_{j=1}^{\fz}\lz_j}
\int_{\rn}\bfai\left(\frac{\cg_N
(f_i)(x)}{\lz_i^{1/p_{\bfai}}}\r)\oz(x)\,dx\le
\sum_{i=1}^{\fz}\frac{\lz_i}{\sum_{j=1}^{\fz}\lz_j}=1,
\end{eqnarray*}
 which further implies that
\begin{eqnarray}\label{3.31}
\|f\|^{p_{\bfai}}_{h^{\bfai}_{\oz,\,N}(\rn)}\le\sum_{i=1}^{\fz}
\|f_i\|^{p_{\bfai}}_{h^{\bfai}_{\oz,\,N}(\rn)}.
\end{eqnarray}

To prove that $h^{\bfai}_{\oz,\,N}(\rn)$ is complete, it suffices to
show that for every sequence $\{f_j\}_{j\in\nn}$ with $\|f_j\|
_{h^{\bfai}_{\oz,\,N}(\rn)}<2^{-j}$ for any $j\in\nn$, the series
$\{f_j\}_{j\in\nn}$ converges in $h^{\bfai}_{\oz,\,N}(\rn)$. Since
$\{\sum_{i=1}^j f_i\}_{j\in\nn}$ is a Cauchy sequence in
$h^{\bfai}_{\oz,\,N}(\rn)$, by Proposition \ref{p3.16} and the
completeness of $\cd'(\rn)$, $\{\sum_{i=1}^j f_i\}_{j\in\nn}$ is
also a Cauchy sequence in $\cd'(\rn)$ and thus converges to some
$f\in\cd'(\rn)$. Therefore, by \eqref{3.31},
$$\left\|f-\sum_{i=1}^j
f_i\r\|^{p_{\bfai}}_{h^{\bfai}_{\oz,\,N}(\rn)}=\left\|\sum_{i=j+1}^{\fz}
f_i\r\|^{p_{\bfai}}_{h^{\bfai}_{\oz,\,N}(\rn)}\le
\sum_{i=j+1}^{\fz}2^{-ip_{\bfai}}\to0$$ as $j\to\fz$, which
completes the proof of Proposition \ref{p3.17}.
\end{proof}

\begin{theorem}\label{t3.18}
Let $\bfai$ satisfy Assumption $\mathrm{(A)}$, $\oz\in
A^{\loc}_{\fz}(\rn)$ and $N_{\bfai,\,\oz}$ be as in \eqref{3.25}. If
$(\rz,\,q,\,s)_{\oz}$ is admissible (see Definition \ref{d3.4}) and
integer $N\ge N_{\bfai,\,\oz}$, then
$$h^{\rz,\,q,\,s}_{\oz}(\rn)\subset
h^{\bfai}_{\oz,\,N_{\bfai,\,\oz}}(\rn)\subset
h^{\bfai}_{\oz,\,N}(\rn)$$
and, moreover, there exists a positive
constant $C$ such that for all $f\in h^{\rz,\,q,\,s}_{\oz}(\rn)$,
$$\|f\|_{h^{\bfai}_{\oz,\,N}(\rn)}\le\|f\|_{h^{\bfai}_{\oz,\,
N_{\bfai,\,\oz}}(\rn)} \le C\|f\|_{h^{\rz,\,q,\,s}_{\oz}(\rn)}.$$
\end{theorem}

\begin{proof}
Obviously, by Definition \ref{d3.3}, we only need prove that
$h^{\rz,\,q,\,s}_{\oz}(\rn)\subset
h^{\bfai}_{\oz,\,N_{\bfai,\,\oz}}(\rn)$, and for all $f\in
h^{\rz,\,q,\,s}_{\oz}(\rn)$,
$$\|f\|_{h^{\bfai}_{\oz,\,N_{\bfai,\,\oz}}(\rn)} \ls
\|f\|_{h^{\rz,\,q,\,s}_{\oz}(\rn)}.$$
To this end, by Theorem
\ref{t3.14} and Definition \ref{d3.5}, it suffices to prove that for
any $(\rz,\,q)_{\oz}$-single-atom $a$ and $\lz\in\cc$,
\begin{eqnarray}\label{3.32}
\int_{\rn}\bfai\left(\cg^0_{N_{\bfai,\,\oz}}(\lz a)(x)\r)\oz(x)\,dx\ls
\oz(\rn)\bfai\left(\frac{|\lz|}{\oz(\rn)\rz(\oz(\rn))}\r),
\end{eqnarray}
and for any $(\rz,\,q,\,s)_{\oz}$-atom $a$ supported in the cube $Q$
and $\lz\in\cc$,
\begin{eqnarray}\label{3.33}
&&\int_{\rn}\bfai\left(\cg^0_{N_{\bfai,\,\oz}}(\lz
a)(x)\r)\oz(x)\,dx\ls\oz(Q)\bfai\left(\frac{|\lz|}{\oz(Q)\rz(\oz(Q))}\r).
\end{eqnarray}
Indeed, for any $f\in h^{\rz,\,q,\,s}_{\oz}(\rn)$,
$$f=\sum_{i=0}^{\fz}\lz_i a_i$$
in $\cd'(\rn)$, where
$\{\lz_i\}_{i=0}^{\fz}\subset\cc$, $a_0$ is a
$(\rz,\,q)_{\oz}$-single-atom and for any $i\in\nn$, $a_i$ is a
$(\rz,\,q,\,s)_{\oz}$-atom supported in the cube $Q_i$. Then, for
any $\lz\in(0,\fz)$, by the facts that
$\cg^0_{N_{\bfai,\,\oz}}(f/\lz)=\cg^0_{N_{\bfai,\,\oz}}(f)/\lz$ and
$\bfai$ is strictly increasing, subadditive and continuous,
\eqref{3.32} and \eqref{3.33}, we have
\begin{eqnarray*}
&&\int_{\rn}\bfai\left(\frac{\cg^0_{N_{\bfai,\,\oz}}
(f)(x)}{\lz}\r)\oz(x)\,dx\\
&&\hs=\int_{\rn}\bfai\left(\cg^0_{N_{\bfai,\,\oz}}
\left(\frac{f}{\lz}\r)(x)\r)\oz(x)\,dx\le\sum_{i=0}^{\fz}\int_{\rn}\bfai
\left(\cg^0_{N_{\bfai,\,\oz}}
\left(\frac{\lz_i a_i}{\lz}\r)(x)\r)\oz(x)\,dx\\
&&\hs\ls\oz(\rn)\bfai
\left(\frac{|\lz_0|}{\lz\oz(\rn)\rz(\oz(\rn))}\r)+
\sum_{i=1}^{\fz}\oz(Q_i)\bfai\left(\frac{|\lz_i|}{\lz\oz(Q_i)\rz(\oz(Q_i))}\r),
\end{eqnarray*}
which together with Theorem \ref{t3.14} implies that
$\|f\|_{h^{\bfai}_{\oz,\,N_{\bfai,\,\oz}}(\rn)} \ls
\|f\|_{h^{\rz,\,q,\,s}_{\oz}(\rn)}$.

We now prove \eqref{3.32}. Since $q\in(q_{\oz},\fz]$, by the
definition of $q_{\oz}$, we have $\oz\in A^{\loc}_q (\rn)$. Let $a$
be a $(\rz,\,q)_{\oz}$-single-atom and $\lz\in\cc$. When
$\oz(\rn)=\fz$, by the definition of the single atom, we know that
$a=0$ for almost every $x\in\rn$. In this case, it is easy to see
that \eqref{3.32} holds. When $\oz(\rn)<\fz$, since $\bfai$ is
concave, from Jensen's inequality, H\"older's inequality and
Proposition \ref{p3.2}(ii), we deduce that
\begin{eqnarray*}
&&\int_{\rn}\bfai\left(\cg^0_{N_{\bfai,\,\oz}}(\lz
a)(x)\r)\oz(x)\,dx\\
&&\hs\le\oz(\rn)\bfai\left(\frac{1}{\oz(\rn)}
\int_{\rn}\cg^0_{N_{\bfai,\,\oz}}(\lz
a)(x)\oz(x)\,dx\r)\\
&&\hs\le\oz(\rn)\bfai\left(\frac{1}{[\oz(\rn)]^{1/q}}
\left\{\int_{\rn}\left[\cg^0_{N_{\bfai,\,\oz}}(\lz
a)(x)\r]^{q}\oz(x)\,dx\r\}^{1/q}\r)\\
&&\hs\ls\oz(\rn)\bfai\left(\frac{1}{[\oz(\rn)]^{1/q}}
|\lz|\|a\|_{L^q_{\oz}(\rn)}\r)\ls\oz(\rn)
\bfai\left(\frac{|\lz|}{\oz(\rn)\rz(\oz(\rn))}\r).
\end{eqnarray*}
That is, \eqref{3.32} holds.

 Next, we prove \eqref{3.33}. Let $a$ be a
$(\rz,\,q,\,s)_{\oz}$-atom supported in the  cube $Q\equiv
Q(x_0,r)$, and $\lz\in\cc$. We consider the following two cases for
$Q$.

{\it Case 1)} $|Q|<1$. In this case, letting
$\wz{Q}\equiv2\sqrt{n}Q$, then we have
\begin{eqnarray}\label{3.34}
&&\int_{\rn}\bfai\left(\cg^0_{N_{\bfai,\,\oz}}(\lz a)(x)\r)\oz(x)\,dx\nonumber\\
 &&\hs=\int_{\wz{Q}}\bfai\left(\cg^0_{N_{\bfai,\,\oz}}(\lz
a)(x)\r)\oz(x)\,dx+\int_{\wz{Q}^{\complement}}
\cdots\equiv\mathrm{I_1}+\mathrm{I_2}.
\end{eqnarray}

For $\mathrm{I_1}$, by Jensen's inequality, H\"older's inequality,
Lemma \ref{l2.3}(v) and Proposition \ref{p3.2}(ii), we have
\begin{eqnarray}\label{3.35}
\mathrm{I_1}&\le&\oz(\wz{Q})\bfai\left(\frac{1}{\oz(\wz{Q})}
\int_{\wz{Q}}\cg^0_{N_{\bfai,\,\oz}}(\lz a)(x)\oz(x)\,dx\r)\nonumber\\
 &\le&\oz(\wz{Q})\bfai\left(\frac{1}{[\oz(\wz{Q})]^{1/q}}
\left\{\int_{\wz{Q}}\left[\cg^0_{N_{\bfai,\,\oz}}(\lz
a)(x)\r]^{q}\oz(x)\,dx\r\}^{1/q}\r)\nonumber\\
&\ls&\oz(\wz{Q})\bfai\left(\frac{1}{[\oz(\wz{Q})]^{1/q}}
|\lz|\|a\|_{L^q_{\oz}(\rn)}\r)\nonumber\\
&\ls&\oz(\wz{Q})\bfai\left(\frac{|\lz|}{\oz(Q)\rz(\oz(Q))
}\r)\ls\oz(Q)\bfai\left(\frac{|\lz|}{\oz(Q)\rz(\oz(Q))}\r),
\end{eqnarray}
which is the desired estimate for $\mathrm{I_1}$.

To estimate $\mathrm{I_2}$, we claim that for all
$x\in\wz{Q}^{\complement}$,
\begin{eqnarray}\label{3.36}
\cg^0_{N_{\bfai,\,\oz}}(\lz a)(x)&\ls&|\lz||Q|^{\frac{s_0
+n+1}{n}}[\oz(Q)\rz(\oz(Q))]^{-1}\nonumber\\ &&\times|x-x_0|^{-(s_0
+n+1)} \chi_{B(x_0,2\sqrt{n})}(x),
\end{eqnarray}
where $s_0\equiv\lfz n(\frac{q_{\oz}}{p_{\bfai}}-1)\rf$. Indeed, for
any $\pz\in\cd^0_{N}(\rn)$ and $t\in(0,1)$, let $P$ be the Taylor
expansion of $\pz$ about $(x-x_0)/t$ with degree $s_0$. By Taylor's
remainder theorem, for any $y\in\rn$, we have
\begin{eqnarray*}
&&\left|\pz\left(\frac{x-y}{t}\r)-P\left(\frac{x-y}{t}\r)\r|\\
&&\hs\ls\sum_{\gfz{\az\in\zz^n_+}{|\az|=s_0
+1}}\left|\left(\partial^{\az}\pz\r)
\left(\frac{\theta(x-y)+(1-\theta)(x-x_0)}{t}\r)\r|\left|\frac{x_0
-y}{t}\r|^{s_0 +1},
\end{eqnarray*}
where $\theta\in(0,1)$. By $t\in(0,1)$ and
$x\in\wz{Q}^{\complement}$, we see that $\supp (a\ast\pz_t)\subset
B(x_0,2\sqrt{n})$ and that $a\ast\pz_t (x)\neq0$ implies that
$t>\frac{|x-x_0|}{2}$. Thus, from the above facts, Definition
\ref{d3.4} and \eqref{2.1}, it follows that for all
$x\in\wz{Q}^{\complement}$,
\begin{eqnarray*}
|a\ast\pz_t (x)|&\le&\frac{1}{t^n}\left\{\int_{Q}|a(y)|
\left|\pz\left(\frac{x-y}{t}\r)-P\left(\frac{x-y}{t}\r)\r|\,dy\r\}\,
\chi_{B(x_0,2\sqrt{n})}(x)\\
&\ls&|x-x_0|^{-(s_0 +n+1)}\left\{\int_{Q}|a(y)||x_0 -y|^{s_0
+1}\,dy\r\} \chi_{B(x_0,2\sqrt{n})}(x)\\
&\ls&|Q|^{\frac{s_0 +1}{n}}\|a\|_{L^q_{\oz}
(\rn)}\left(\int_{Q}[\oz(y)]^{-q'/q}\,dy\r)^{1/q'}
|x-x_0|^{-(s_0 +n+1)}\chi_{B(x_0,2\sqrt{n})}(x)\\
&\ls&|Q|^{\frac{s_0 +n+1}{n}}[\oz(Q)\rz(\oz(Q))]^{-1}|x-x_0|^{-(s_0
+n+1)}\chi_{B(x_0,2\sqrt{n})}(x),
\end{eqnarray*}
which together with the arbitrariness of $\pz\in\cd^0_{N}(\rn)$
implies \eqref{3.36}. Thus, the claim holds.

Let $Q_k\equiv2^k\sqrt{n}Q$ for all $k\in\nn$ and $k_0 \in\nn$
satisfy $2^{k_0}r\le4<2^{k_0 +1}r$. By
$$s_0=\left\lfz n\left(\frac{q_{\oz}}{p_{\bfai}}-1\r)\r\rf,$$
we know that there exists $q_0\in(q_{\oz},\fz)$ such that
$p_{\bfai}(s_0 +n+1)>nq_0$. From Lemma \ref{l2.3}, it follows that
there exists an $\overline{\oz}\in A_{p_0}(\rn)$ such that
$\oz=\overline{\oz}$ on $Q(x_0,8\sqrt{n})$. By this fact,
\eqref{3.36}, the lower type $p_{\bfai}$ property of $\bfai$ and
Lemma \ref{l2.3}(viii), we conclude that
\begin{eqnarray*}
\mathrm{I_2}&\le&\int_{\sqrt{n}r\le|x-x_0|<2\sqrt{n}}
\bfai\left(\cg^0_{N_{\bfai,\,\oz}}(\lz a)(x)\r)\oz(x)\,dx\\
&\ls&\int_{\sqrt{n}r\le|x-x_0|<2\sqrt{n}}\bfai\left(|\lz||Q|^{\frac{s_0
+n+1}{n}}\left[\oz(Q)\rz(\oz(Q))\r]^{-1}|x-x_0|^{-(s_0 +n+1)}\r)
\overline{\oz}(x)\,dx\\
&\ls&\sum_{k=1}^{k_0}\int_{Q_{k+1}\setminus Q_{k}}\bfai\left(|\lz|2
^{-k(s_0 +n+1)}\left[\oz(Q)\rz(\oz(Q))\r]^{-1}\r)
\overline{\oz}(x)\,dx\\
&\ls&\sum_{k=1}^{k_0} 2^{-k(s_0
+n+1)p_{\bfai}}\overline{\oz}(Q_{k+1})\bfai\left(
\frac{|\lz|}{\oz(Q)\rz(\oz(Q))}\r)\\
&\ls&\sum_{k=1}^{k_0} 2^{-k[(s_0 +n+1)p_{\bfai}-n
q_0]}\oz(Q)\bfai\left(\frac{|\lz|}{\oz(Q)\rz(\oz(Q))}\r)\\
&\ls&\oz(Q)\bfai\left(\frac{|\lz|}{\oz(Q)\rz(\oz(Q))}\r),
\end{eqnarray*}
which together with \eqref{3.34} and \eqref{3.35} implies
\eqref{3.33} in Case 1).

{\it Case 2)} $|Q|\ge1$. In this case, let $Q^{\ast}\equiv Q(x_0,
r+2)$. Thus, from
$$\supp\left(\cg^0_{N_{\bfai,\,\oz}}(\lz a)\r)\subset
Q^{\ast},$$
Jensen's inequality, H\"older's inequality, Lemma
\ref{l2.3}(v), and Proposition \ref{p3.2}(ii), we deduce that
\begin{eqnarray*}
&&\int_{\rn}\bfai\left(\cg^0_{N_{\bfai,\,\oz}}(\lz
a)(x)\r)\oz(x)\,dx\\
&&\hs=\int_{Q^{\ast}}\bfai\left(\cg^0_{N_{\bfai,\,\oz}} (\lz
a)(x)\r)\oz(x)\,dx
\le\oz(Q^{\ast})\bfai\left(\frac{1}{\oz(Q^{\ast})}\int_{Q^{\ast}}
\cg^0_{N_{\bfai,\,\oz}}(\lz
a)(x)\oz(x)\,dx\r)\\
&&\hs\le\oz(Q^{\ast})\bfai\left(\frac{1}
{[\oz(Q^{\ast})]^{1/q}}\left\{\int_{Q^{\ast}}
\left[\cg^0_{N_{\bfai,\,\oz}}(\lz a)(x)\r]^q\oz(x)\,dx\r\}^{1/q}\r)\\
&&\hs\ls\oz(Q^{\ast})\bfai\left(\frac{|\lz|}{[\oz(Q^{\ast})]^{1/q}}
\|a\|_{L^q_{\oz}(\rn)}\r)
\ls\oz(Q^{\ast})\bfai\left(\frac{|\lz|}{\oz(Q)\rz(\oz(Q))}\r)\\
&&\hs\ls\oz(Q)\bfai\left(\frac{|\lz|}{\oz(Q)\rz(\oz(Q))}\r),
\end{eqnarray*}
which proves \eqref{3.33} in Case 2). This finishes the proof of
Theorem \ref{t3.18}.
\end{proof}

\chapter{Calder\'on-Zygmund decompositions\label{s4}}

In this section, we establish some subtle estimates for the
Calder\'on-Zygmund decomposition associated with local grand maximal
functions on the weighted Euclidean space $\rn$ given in \cite{Ta1}.
Notice that the construction of the Calder\'on-Zygmund decomposition
in \cite{Ta1} is similar to those in \cite{St93,B03,blyz08}.

Let $\Phi$ be a positive function on $\rr_+$ satisfying Assumption
$\mathrm{(A)}$, $\oz\in A^{\loc}_{\fz}(\rn)$ and $q_{\oz}$ be as in
\eqref{2.4}. Let integer $N\ge 2$, $\cg_N (f)$ and $\cg_N^0 (f)$ be
as in \eqref{3.2}.

Throughout this section, let $f\in\cd'(\rn)$ satisfy that for all
$\lz\in(0,\fz)$,
$$\oz\left(\{x\in\rn:\,\cg_N
(f)(x)>\lz\}\r)<\fz.$$ For a given $\lz>\inf_{x\in\rn}\cg_{N}
(f)(x)$, we set
\begin{eqnarray}\label{4.1}
\boz_{\lz}\equiv\{x\in\rn:\,\cg_{N} (f)(x)>\lz\}.
\end{eqnarray}
It is obvious that $\boz_{\lz}$ is a proper open subset of $\rn$.
First, we recall the usual Whitney decomposition of $\boz_{\lz}$
given in \cite{Ta1} (see also \cite{St93,B03,blyz08}). We can find
closed cubes $\{Q_i\}_i$ such that
\begin{eqnarray}\label{4.2}
\boz_{\lz}=\bigcup_{i}Q_i,
\end{eqnarray}
their interiors are away from $\boz_{\lz}^\complement$ and
$$\diam(Q_i)\le2^{-(6+n)}\dist(Q_i,\boz_{\lz}^\complement)\le4\diam(Q_i).$$
In what follows, fix $a\equiv 1+2^{-(11+n)}$ and denote $aQ_i$ by
$Q^{\ast}_i$ for all $i$. Then we have $Q_i\subset Q^{\ast}_i$.
Moveover, $\boz_{\lz}=\cup_{i}Q^{\ast}_i$, and $\{Q^{\ast}_i\}_i$
have the bounded interior property, namely, every point in
$\boz_{\lz}$ is contained in at most a fixed number of
$\{Q^{\ast}_i\}_i$.

Now we take a function $\xz\in \cd(\rn)$ such that $0\le\xz\le1$,
$\supp(\xz)\subset aQ(0,1)$ and $\xz\equiv1$ on $Q(0,1)$. For
$x\in\rn$, set $\xz_i (x)\equiv\xz((x-x_k)/l_i)$, where and in what
follows, $x_i$ is the {\it center} of the cube $Q_i$ and $l_i$ its
{\it sidelength}. Obviously, by the construction of
$\{Q_i^{\ast}\}_i$ and $\{\xz_i\}_i$, for any $x\in\rn$, we have
$1\le\sum_{i}\xz_i (x)\le L$, where $L$ is a fixed positive integer
independent of $x$. Let
\begin{eqnarray}\label{4.3}
\zez_i\equiv\frac{\xz_i}{\sum_j\xz_j}.
\end{eqnarray}
 Then $\{\zez_i\}_i$ form a smooth
partition of unity for $\boz_{\lz}$ subordinate to the locally
finite cover $\{Q_i^{\ast}\}_i$ of $\boz_{\lz}$, namely,
$\chi_{\boz_{\lz}}=\sum_k \zez_k$ with each $\zez_i\in \cd(\rn)$
supported in $Q_i^{\ast}$.

Let $s\in\zz_{+}$ be some fixed integer and $\cp_s (\rn)$ denote the
linear space of polynomials in $n$ variables of degrees no more than
$s$. For each $i\in\nn$ and $P\in\cp_s (\rn)$, set
\begin{eqnarray}\label{4.4}
\|P\|_i\equiv\left[\frac{1}{\int_{\rn}\zez_i
(y)\,dy}\int_{\rn}|P(x)|^2\zez_i (x)\,dx\r]^{1/2}.
\end{eqnarray}
Then it is easy to know that $(\cp_s (\rn),\,\|\cdot\|_i)$ is a
finite dimensional Hilbert space. Let $f\in\cd'(\rn)$. Since $f$
induces a linear functional on $\cp_s (\rn)$ via
$$P\mapsto\frac{1}{\int_{\rn}\zez_i (y)\,dy}\langle f,
P\zez_i\rangle,$$
by the Riesz represent theorem, there exists a
unique polynomial
\begin{eqnarray}\label{4.5}
P_i \in\cp_s (\rn)
\end{eqnarray}
 for each $i$ such that for
all $P\in\cp_s (\rn)$, $\langle f, P\zez_i\rangle=\langle P_i,
P\zez_i\rangle$. For each $i$, define the distribution
\begin{eqnarray}\label{4.6}
b_i\equiv(f-P_i)\zez_i\,\, \text{when}\,\,l_i\in(0,1),\ \text{and} \
b_i\equiv f\zez_i \,\,\text{when}\,\,l_i\in[1,\fz).
\end{eqnarray}
We show that for suitable choices of $s$ and $N$, the series $\sum_i
b_i$ converge in $\cd'(\rn)$, and in this case, we let $g\equiv
f-\sum_i b_i$ in $\cd'(\rn)$. We point out that the represent
\begin{eqnarray}\label{4.7}
f=g+\sum_i b_i,
\end{eqnarray}
where $g$ and $b_i$ are as above, is called a {\it
Calder\'on-Zygmund decomposition} of $f$ of degree $s$ and height
$\lz$ associated with $\cg_N (f)$.

The rest of this section consists of a series of lemmas. Lemma
\ref{l4.1} gives a property of the smooth partition of unity
$\{\zez_i\}_i$, Lemmas \ref{l4.2} through \ref{l4.5} are devoted to
some estimates for the bad parts $\{b_i\}_i$, and Lemmas \ref{l4.6}
and \ref{l4.7} give some controls over the good part $g$. Finally,
Corollary \ref{c4.8} shows that the density of $L^q_{\oz}(\rn)\cap
h^{\bfai}_{\oz,\,N}(\rn)$ in $h^{\bfai}_{\oz,\,N}(\rn)$, where
$q\in(q_{\oz},\fz)$. The following Lemmas \ref{l4.1} through
\ref{l4.3}, and Lemmas \ref{l4.5} and \ref{l4.6} are respectively
Lemmas 4.2 through 4.5, and Lemmas 4.7 and 4.8 in \cite{Ta1}.
\begin{lemma}\label{l4.1}
There exists a positive constant $C_1$ such that for all
$f\in\cd'(\rn)$,
$$\lz>\inf_{x\in\rn}\cg_{N} (f)(x)$$
and $l_i\in(0,1)$,
$$\sup_{y\in\rn}|P_i (y)\zez_i (y)|\le C_1\lz.$$
\end{lemma}

\begin{lemma}\label{l4.2}
There exists a positive constant $C_2$ such that for all $i\in\nn$
and $x\in Q_i^{\ast}$,
\begin{eqnarray}\label{4.8}
\cg^0_N (b_i)(x)\le C_2\cg_N (f)(x).
\end{eqnarray}
\end{lemma}

\begin{lemma}\label{l4.3}
Assume that integers $s$ and $N$ satisfy $0\le s<N$ and $N\ge2$.
Then there exist positive constants $C$, $C_3$ and $C_4$ such that
for all $i\in\nn$ and $x\in(Q_i^{\ast})^\complement$,
\begin{eqnarray}\label{4.9}
\cg^0_N (b_i)(x)\le C\frac{\lz l_i^{n+s+1}}{(l_i
+|x-x_i|)^{n+s+1}}\chi_{B(x_i, C_3)}(x),
\end{eqnarray}
where $x_i$ is the center of the cube $Q_i$. Moreover, if
$x\in(Q_i^{\ast})^\complement$ and $l_i\in[C_4,\fz)$, then $\cg^0_N
(b_i)(x)=0$.
\end{lemma}

\begin{lemma}\label{l4.4}
Let $\bfai$ satisfy Assumption $\mathrm{(A)}$, $\oz\in
A^{\loc}_{\fz}(\rn)$, $q_{\oz}$ and $p_{\bfai}$ be respectively as
in \eqref{2.4} and \eqref{2.6}. If integers $s\ge\lfz
n(q_{\oz}/p_{\bfai}-1)\rf$, $N>s$ and $N\ge N_{\bfai,\,\oz}$, where
$N_{\bfai,\,\oz}$ is as in \eqref{3.25}. Then there exists a
positive constant $C_5$ such that for all $f\in
h^{\bfai}_{\oz,\,N}(\rn)$, $\lz>\inf_{x\in\rn}\cg_{N} (f)(x)$ and
$i\in\nn$,
\begin{eqnarray}\label{4.10}
\int_{\rn}\bfai\left(\cg^0_N (b_i)(x)\r)\oz(x)\,dx\le C_5
\int_{Q_i^{\ast}}\bfai\left(\cg_N (f)(x)\r)\oz(x)\,dx.
\end{eqnarray}
Moreover, the series $\sum_i b_i$ converges in
$h^{\bfai}_{\oz,\,N}(\rn)$ and
\begin{eqnarray}\label{4.11}
\int_{\rn}\bfai\left(\cg^0_N \left(\sum_i b_i\r)(x)\r)\oz(x)\,dx\le C_5
\int_{\boz_{\lz}}\bfai\left(\cg_N (f)(x)\r)\oz(x)\,dx.
\end{eqnarray}
\end{lemma}

\begin{proof}
By Lemmas \ref{l4.2} and \ref{l4.3}, we have
\begin{eqnarray}\label{4.12}
\int_{\rn}\bfai\left(\cg^0_N (b_i)(x)\r)\oz(x)\,dx&\ls&
\int_{Q_i^{\ast}}\bfai\left(\cg_N (f)(x)\r)\oz(x)\,dx\nonumber\\
 &&+\int_{(2C_3 Q_i^0)\setminus Q_i^{\ast}}\bfai\left(\cg^0_N
(b_i)(x)\r)\oz(x)\,dx,
\end{eqnarray} where $Q_i^0\equiv Q(x_i,1)$. Notice that
$s\ge\lfz n(q_{\oz}/p_{\bfai}-1)\rf$ implies
$(s+n+1)p_{\bfai}>nq_{\oz}$. Thus, we take $q_0\in(q_{\oz},\fz)$
such that $(s+n+1)p_{\bfai}>nq_0$ and $\oz\in A^{\loc}_{q_0}(\rn)$.
By Lemma \ref{l2.3}(i), we know that there exists an $\wz{\oz}\in
A_{q_0}(\rn)$ such that $\wz{\oz}=\oz$ on $2C_3 Q^0_i$ and $A_{q_0}
(\wz{\oz})\ls A^{\loc}_{q_0}(\oz)$. Using Lemma \ref{4.3}, the lower
$p_{\bfai}$ property of $\bfai$, Lemma \ref{l2.3}(viii) and the fact
that $\cg_N (f)>\lz$ for all $x\in Q^{\ast}_i$, we conclude that
\begin{eqnarray}\label{4.13}
\hs&&\int_{(2C_3 Q_i^0)\setminus Q_i^{\ast}}\bfai\left(\cg^0_N
(b_i)(x)\r)\oz(x)\,dx\nonumber\\  &&\hs\le\sum_{k=1}^{k_0}\int_{2^k
Q_i^{\ast}\setminus 2^{k-1}Q_i^{\ast}}\bfai\left(\cg^0_N
(b_i)(x)\r)\wz{\oz}(x)\,dx
\ls\sum_{k=1}^{k_0}\bfai\left(\frac{\lz}{2^{k(n+s+1)}}\r)\int_{2^k
Q_i^{\ast}}\wz{\oz}(x)\,dx\nonumber\\
&&\hs\ls\sum_{k=1}^{k_0}\bfai(\lz)\frac{1}{2^{k(n+s+1)p_{\bfai}}}\int_{2^k
Q_i^{\ast}}\wz{\oz}(x)\,dx
\ls\sum_{k=1}^{k_0}\bfai(\lz)2^{-k[(n+s+1)p_{\bfai}-nq_0]}
\wz{\oz}(Q_i^{\ast})\nonumber\\
&&\hs\ls\int_{Q_i^{\ast}}\bfai\left(\cg_N
(f)(x)\r)\wz{\oz}(x)\,dx\sim\int_{Q_i^{\ast}}\bfai\left(\cg_N
(f)(x)\r)\oz(x)\,dx,
\end{eqnarray}
where $k_0\in\nn$ satisfies $2^{k_0 -2}\le C_3<2^{k_0-1}$. From
\eqref{4.12} and \eqref{4.13}, we deduce that \eqref{4.10} holds.
Then, by \eqref{4.10}, we see that
\begin{eqnarray*}
\sum_i\int_{\rn}\bfai\left(\cg^0_N
(b_i)(x)\r)\oz(x)\,dx&\ls&\sum_i\int_{Q_i^{\ast}}\bfai\left(\cg_N
(f)(x)\r)\oz(x)\,dx\\
&\ls&\int_{\boz_{\lz}}\bfai\left(\cg_N (f)(x)\r)\oz(x)\,dx.
\end{eqnarray*}
Combining the above inequality with the completeness of
$h^{\bfai}_{\oz,\,N}(\rn)$, we infer that $\sum_i b_i$ converges in
$h^{\bfai}_{\oz,\,N}(\rn)$. So by Proposition \ref{p3.16}, the
series $\sum_i b_i$ converges in $\cd'(\rn)$ and hence
$$\cg^0_N\left(\sum_i b_i\r)(x)\le \sum_i \cg^0_N (b_i)(x)$$
for all $x\in\rn$, which
gives \eqref{4.11}. This finishes the proof of Lemma \ref{4.4}.
\end{proof}

\begin{lemma}\label{l4.5}
Let $\oz\in A^{\loc}_{\fz}(\rn)$ and $q_{\oz}$ be as in \eqref{2.4},
$s\in\zz_{+}$ and integer $N\ge2$. If $q\in(q_{\oz},\fz]$ and $f\in
L^q_{\oz}(\rn)$, then the series $\sum_i b_i$ converges in
$L^q_{\oz}(\rn)$ and there exists a positive constant $C_6$,
independent of $f$, such that
$$\left\|\sum_i |b_i|\r\|_{L^q_{\oz}(\rn)}\le
C_6 \|f\|_{L^q_{\oz}(\rn)}.$$
\end{lemma}

\begin{lemma}\label{l4.6}
Let integers $s$ and $N$ satisfy $0\le s<N$ and $N\ge2$,
$f\in\cd'(\rn)$ and $\lz>\inf_{x\in\rn}\cg_{N} (f)(x)$. If $\sum_i
b_i$ converges in $\cd'(\rn)$, then there exists a positive constant
$C_7$, independent of $f$ and $\lz$, such that for all $x\in\rn$,
$$\cg^0_N (g)(x)\le\cg^0_N (f)(x)\chi_{\boz_{\lz}^\complement}(x)+C_7\lz\sum_{i}
\frac{l_i^{n+s+1}}{(l_i
+|x-x_i|)^{n+s+1}}\chi_{B(x_i,C_3)}(x),$$where $x_i$ is the center
of $Q_i$ and $C_3$ is as in Lemma \ref{l4.3}.
\end{lemma}

\begin{lemma}\label{l4.7}
Let $\bfai$ satisfy Assumption $\mathrm{(A)}$, $\oz\in
A^{\loc}_{\fz}(\rn)$, $q_{\oz}$ and $p_{\bfai}$ be respectively as
in \eqref{2.4} and \eqref{2.6}, integer $N\ge N_{\bfai,\,\oz}$,
where $N_{\bfai,\,\oz}$ is as in \eqref{3.25}, and
$q\in(q_{\oz},\fz)$.

$\mathrm{(i)}$ If integers $s$ and $N$ satisfy $N>s\ge\lfz
n(q_{\oz}/p_{\bfai}-1)\rf$ and $f\in h^{\bfai}_{\oz,\,N}(\rn)$, then
$\cg^0_N (g)\in L^q_{\oz}(\rn)$ and there exists a positive constant
$C_8$, independent of $f$ and $\lz$, such that
\begin{eqnarray}\label{4.14}
\int_{\rn}\left[\cg^0_N (g)(x)\r]^q\oz(x)\,dx\le C_8
\begin{cases}\lz^{q-1}\int_{\rn}\bfai
\left(\cg_N (f)(x)\r)\oz(x)\,dx, &\lz\in(0,1),\\
\lz^{q-p_{\bfai}}\int_{\rn}\bfai\left(\cg_N (f)(x)\r)\oz(x)\,dx,
&\lz\in[1,\fz) .
\end{cases}
\end{eqnarray}

$\mathrm{(ii)}$  If $f\in L^q_{\oz}(\rn)$, then $g\in
L^{\fz}_{\oz}(\rn)$ and there exists a positive constant $C_9$,
independent of $f$ and $\lz$, such that
$\|g\|_{L^{\fz}_{\oz}(\rn)}\le C_9 \lz$.
\end{lemma}

\begin{proof}
We first prove (i). Let $f\in h^{\bfai}_{\oz,\,N}(\rn)$. By Lemma
\ref{l4.4} and Proposition \ref{p3.16}, $\sum_i b_i$ converges in
both $h^{\bfai}_{\oz,\,N}(\rn)$ and $\cd'(\rn)$. By $s\ge\lfz
n(q_{\oz}/p_{\bfai}-1)\rf$, we know that there exists
$q_0\in(q_{\oz},\fz)$ such that $(s+n+1)p_{\bfai}>nq_0$ and $\oz\in
A^{\loc}_{q_0}(\rn)$. Let
$$\mathrm{J}\equiv\int_{\boz_{\lz}^\complement}[\cg_N (f)(x)]^q
\oz(x)\,dx.$$
From Lemmas \ref{l4.6} and \ref{l3.10}, we infer that
\begin{eqnarray*}
\int_{\rn}\left[\cg^0_N
(g)(x)\r]^q\oz(x)\,dx&&\ls\lz^q\int_{\rn}\left[\sum_i
\frac{l_i^{n+s+1}}{(l_i +|x-x_i|)^{n+s+1}}\chi_{B(x_i,C_3)}(x)\r]^q
\oz(x)\,dx
+\mathrm{J}\\
&&\hs\ls\lz^q\int_{\rn}\left(\sum_i
\left[M^{\loc}_{2C_3}(\chi_{Q_i})(x)\r]^{(n+s+1)/n}\r)^q
\oz(x)\,dx+\mathrm{J}\\
&&\hs\ls\lz^q\int_{\rn}\left(\sum_i [\chi_{Q_i}(x)]^{(n+s+1)/n}\r)^q
\oz(x)\,dx+\mathrm{J}\\
&&\hs\ls\lz^q\int_{\boz_{\lz}}\oz(x)\,dx
+\mathrm{J}\sim\lz^q\oz(\boz_{\lz})+\mathrm{J}.
\end{eqnarray*}
Now, we consider the following two cases for $\lz$.

{\it Case 1)} $\lz\ge1$. In this case, since $\bfai$ has lower type
$p_{\bfai}$, we have
\begin{eqnarray*}
\lz^q\oz(\boz_{\lz})&\le&\lz^{q-p_{\bfai}}\oz(\boz_{\lz})
\left[\inf_{x\in\boz_{\lz}}\cg_N
(f)(x)\r]^{p_{\bfai}}\le\lz^{q-p_{\bfai}}\oz(\boz_{\lz})\bfai
\left(\inf_{x\in\boz_{\lz}}\cg_N
(f)(x)\r)\\
&\le&\lz^{q-p_{\bfai}}\int_{\boz_{\lz}}\bfai\left(\cg_N
(f)(x)\r)\oz(x)\,dx.
\end{eqnarray*}
Recall that
$$\boz_1\equiv\{x\in\rn:\,\cg_N (f)(x)>1\}.$$
From the fact that $\bfai$ has lower type $p_{\bfai}$ and upper type 1, it
follows that
\begin{eqnarray*}
\mathrm{J}&=&\int_{\boz_{\lz}^\complement\cap\boz_1}\left[\cg_N
(f)(x)\r]^q
\oz(x)\,dx+\int_{\boz_{\lz}^\complement\cap\boz_1^\complement}\cdots\\
&\ls&\int_{\boz_{\lz}^\complement}[\cg_N
(f)(x)]^{q-p_{\bfai}}\bfai\left(\cg_N (f)(x)\r)\oz(x)\,dx
+\int_{\boz_{\lz}^\complement}[\cg_N (f)(x)]^{q-1}\bfai\left(\cg_N
(f)(x)\r)\oz(x)\,dx\\
&\ls&(\lz^{q-p_{\bfai}}+\lz^{q-1})\int_{\boz_{\lz}^\complement}\bfai\left(\cg_N
(f)(x)\r)\oz(x)\,dx
\ls\lz^{q-p_{\bfai}}\int_{\boz_{\lz}^\complement}\bfai\left(\cg_N
(f)(x)\r)\oz(x)\,dx,
\end{eqnarray*}
which together with the estimate of $\lz^q\oz(\boz_{\lz})$ implies
\eqref{4.10} in Case 1).

{\it Case 2)} $\lz\in(0,1)$. In this case, for any $x\in\boz_{\lz}$,
if $\cg_N (f)(x)\ge1>\lz$, using the fact that $\bfai$ has lower
type $p_{\bfai}$, we conclude that
$$\lz^q\le\lz^{q-p_{\bfai}}[\cg_N (f)(x)]^{p_{\bfai}}
\ls\lz^{{q-p_{\bfai}}} \bfai\left(\cg_N
(f)(x)\r)\ls\lz^{q-1}\bfai\left(\cg_N (f)(x)\r).$$

If $\cg_N (f)(x)<1$ and $\cg_N (f)(x)>\lz$, by the fact that
$\bfai$ has upper type 1, we see that
$$\lz^q\le\lz^{q-1}\cg_N (f)(x)\ls\lz^{q-1}
\bfai\left(\cg_N (f)(x)\r).$$
From these estimates, we deduce that
$$\lz^q\oz(\boz_{\lz})\ls\lz^{q-1}\int_{\boz_{\lz}}\bfai\left(\cg_N
(f)(x)\r)\oz(x)\,dx.$$
For $\mathrm{J}$, by $\lz\in(0,1)$, $\cg_N
(f)(x)\le\lz$ for all $x\in\boz_{\lz}^\complement$ and the fact that
$\bfai$ has upper type $1$, we know that
$$\mathrm{J}\le\lz^{q-1}\int_{\boz_{\lz}^\complement}\cg_N
(f)(x)\oz(x)\,dx\ls\lz^{q-1}\int_{\boz_{\lz}^\complement}\bfai\left(\cg_N
(f)(x)\r)\oz(x)\,dx,$$
which together with the estimate of
$\lz^q\oz(\boz_{\lz})$ implies \eqref{4.14} in Case 2). Thus, (i)
holds.

Now we prove (ii). If $f\in L^q_{\oz}(\rn)$, then $g$ and
$\{b_i\}_i$ are functions. By Lemma \ref{l4.5}, we know that $\sum_i
b_i$ converges in $L^q_{\oz}(\rn)$ and hence in $\cd'(\rn)$ by Lemma
\ref{l2.6}(ii). Write
$$g=f-\sum_i b_i=f\left(1-\sum_i \zez_i\r)+\sum_{i\in F}
P_i\zez_i=f\chi_{\boz_{\lz}^\complement}+\sum_{i\in F} P_i \zez_i,$$
where $F\equiv\{i\in\nn:\,l_i\in(0,1)\}$. By Lemma \ref{l4.1}, we
have that $|g(x)|\ls\lz$ for all $x\in\boz_{\lz}$, which combined
with Proposition \ref{p3.2}(i) yields that
$$|g(x)|=|f(x)|\le\cg_N(f)(x)\le\lz$$
for almost every $x\in\boz_{\lz}^\complement$, Thus,
$\|g\|_{L^{\fz}_{\oz}(\rn)}\ls\lz$. This shows (ii) and hence
finishes the proof of Lemma \ref{l4.7}.
\end{proof}

\begin{corollary}\label{c4.8}
Let $\bfai$ satisfy Assumption $\mathrm{(A)}$, $\oz\in
A^{\loc}_{\fz}(\rn)$, $q_{\oz}$ be as in \eqref{2.4},
$$q\in(q_{\oz},\fz)$$
and integer $N\ge N_{\bfai,\,\oz}$, where
$N_{\bfai,\,\oz}$ is as in \eqref{3.25}. Then
$h^{\bfai}_{\oz,\,N}(\rn)\cap L^q_{\oz} (\rn)$ is dense in
$h^{\bfai}_{\oz,\,N}(\rn)$.
\end{corollary}

\begin{proof}
Let $f\in h^{\bfai}_{\oz,\,N}(\rn)$. For any
$\lz>\inf_{x\in\rn}\cg_N (f)(x)$, let
$$f=g^{\lz}+\sum_i b_i^{\lz}$$
be the Calder\'on-Zygmund decomposition of $f$ of degree $s$ with
$\lfz n(q_{\oz}/p_{\bfai}-1)\rf\le s<N$ and height $\lz$ associated
to $\cg_N (f)$. By Lemma \ref{l4.4},
$$\int_{\rn}\bfai\left(\cg^0_N \left(\sum_i b^{\lz}_i\r)(x)\r)\oz(x)\,dx
\ls\int_{\{x\in\rn:\,\cg_N (f)(x)>\lz\}}\bfai\left(\cg_N
(f)(x)\r)\oz(x)\,dx.$$ Therefore, $g^{\lz}\to f$ in
$h^{\bfai}_{\oz,\,N}(\rn)$ as $\lz\to\fz$. Moreover, by Lemma
\ref{l4.7}(i), we have
$$\cg^0_N (g^{\lz})\in L^q_{\oz}(\rn),$$
which together with Proposition \ref{p3.2}(ii) implies $g^{\lz}\in
L^q_{\oz}(\rn)$. This finishes the proof of Corollary \ref{c4.8}.
\end{proof}

\chapter{Weighted atomic decompositions of
$h^{\bfai}_{\oz,\,N}(\rn)$\label{s5}}

In this section, we establish the
equivalence between $h^{\bfai}_{\oz,\,N}(\rn)$ and
$h^{\rz,\,q,\,s}_{\oz}(\rn)$ by using the Calder\'on-Zygmund
decomposition associated to the local grand maximal function stated
in Section \ref{s4}.

Let $\bfai$ satisfy Assumption $\mathrm{(A)}$, $\oz\in
A^{\loc}_{\fz}(\rn)$, $q_{\oz}$, $p_{\bfai}$ and $N_{\bfai,\,\oz}$
be respectively as in \eqref{2.4}, \eqref{2.6} and \eqref{3.25},
integer $N\ge N_{\bfai,\,\oz}$ and $s_0\equiv\lfz
n(q_{\oz}/p_{\bfai}-1)\rf$. \emph{Throughout this section, let}
$$f\in h^{\bfai}_{\oz,\,N}(\rn).$$
We take $k_0\in\zz$ such that $2^{k_0
-1}\le\inf_{x\in\rn}\cg_N (f)(x)<2^{k_0}$ when
$$\inf_{x\in\rn}\cg_N(f)(x)>0,$$
and when $\inf_{x\in\rn}\cg_N (f)(x)=0$, let $k_0
\equiv-\fz$. {\it Throughout the whole section, we always assume
that $k\ge k_0$.} For each integer $k\ge k_0$, consider the
Calder\'on-Zygmund decomposition of $f$ of degree $s$ and height
$\lz=2^k$ associated to $\cg_N (f)$. Namely, for any $k\ge k_0$, by
taking $\lz\equiv2^k$ in \eqref{4.1}, we now write the
Calder\'on-Zygmund decomposition in \eqref{4.7} by
\begin{eqnarray}\label{5.1}
f=g^k +\sum_{i}b_i^k,
\end{eqnarray}
where and in what follows of this section, we write $\{Q_i\}_i$ in
\eqref{4.2}, $\{\zez_i\}_i$ in \eqref{4.3}, $\{P_i\}_i$ in
\eqref{4.5} and $\{b_i\}_i$ in \eqref{4.6}, respectively, as
$\{Q^k_i\}_i$, $\{\zez^k_i\}_i$, $\{P^k_i\}_i$ and $\{b^k_i\}_i$.
Now, the {\it center} and the {\it sidelength} of $Q^k_i$ is
respectively denoted by $x^k_i$ and $l^k_i$. Recall that for all $i$
and $k$,
\begin{eqnarray}\label{5.2}
\sum_i \zez^k_i =\chi_{\boz_{2^k}},\ \supp(b^k_i)\subset\supp
(\zez^k_i)\subset Q^{k\ast}_i,
\end{eqnarray}
$\{Q^{k\ast}_i\}_i$ has the bounded interior property, and for all
$P\in\cp_s (\rn)$,
\begin{eqnarray}\label{5.3}
\langle f, P\zez^k_i\rangle=\langle P^k_i,P\zez^k_i\rangle.
\end{eqnarray}

For each integer $k\ge k_0$ and $i,\,j\in\nn$, let $P^{k+1}_{i,\,j}$
be the orthogonal projection of $(f-P^{k+1}_j)\zez^k_i$ on $\cp_s
(\rn)$ with respect to the norm
$$\|P\|_j^2\equiv\frac{1}{\int_{\rn}\zez^{k+1}_j (y)\,dy}
\int_{\rn}|P(x)|^2 \zez^{k+1}_j (x)\,dx,$$ namely, $P^{k+1}_{i,\,j}$
is the unique polynomial of $\cp_s (\rn)$ such that for any
$P\in\cp_s (\rn)$,
\begin{eqnarray}\label{5.4}
\langle(f-P^{k+1}_j)\zez^k_i,P\zez^{k+1}_j\rangle=
\int_{\rn}P^{k+1}_{i,\,j} (x)P(x)\zez^{k+1}_j (x)\,dx.
\end{eqnarray}
Recall that $a\equiv 1+2^{-(11+n)}$. In what follows, let
$Q_i^{k\ast}\equiv aQ^k_i$,
$$E^k_1\equiv\left\{i\in\nn:\ |Q^k_i|\ge1/(2^4
n)\r\},$$
$$E^k_2\equiv\left\{i\in\nn:\ |Q^k_i|<1/(2^4 n)\r\},$$
$$F^k_1\equiv\left\{i\in\nn:\ |Q^k_i|\ge1\r\}$$
and
$$F^k_2\equiv\left\{i\in\nn:\ |Q^k_i|<1\r\}.$$

Observe that
\begin{eqnarray}\label{5.5}
P^{k+1}_{i,\,j}\neq0 \,\,\text{if and only if}\,\,Q_i^{k\ast}\cap
Q_j^{(k+1)\ast}\neq\emptyset.
\end{eqnarray} Indeed, this follows directly from
the definition of $P^{k+1}_{i,\,j}$. The following Lemmas \ref{l5.1}
through \ref{l5.3} are just Lemmas 5.1 through 5.3 in \cite{Ta1}.

\begin{lemma}\label{l5.1}
Let $\boz_{2^k}$ be as in \eqref{4.1} with $\lz=2^k$, $Q_i^{k\ast}$
and $l^k_i$ be as above.

$\mathrm{(i)}$ If $Q_i^{k\ast}\cap Q_j^{(k+1)\ast}\neq\emptyset$,
then $l^{k+1}_j\le2^4\sqrt{n}l^k_i$ and $ Q_j^{(k+1)\ast}\subset2^6
nQ_i^{k\ast}\subset\boz_{2^k}$.

$\mathrm{(ii)}$ There exists a positive integer $L$ such that for
each $i\in\nn$, the cardinality of $\{j\in\nn:\,Q_i^{k\ast}\cap
Q_j^{(k+1)\ast}\neq\emptyset\}$ is bounded by $L$.
\end{lemma}

\begin{lemma}\label{l5.2}
 There exists a positive constant $C$
 such that for all $i,\,j\in\nn$ and integer $k\ge k_0$ with
$l^{k+1}_j \in(0,1)$,
\begin{eqnarray}\label{5.6}
\sup_{y\in\rn}\left|P^{k+1}_{i,\,j}(y)\zez^{k+1}_j (y)\r|\le C2^{k+1}.
\end{eqnarray}
\end{lemma}

\begin{lemma}\label{l5.3}
For any $k\in\zz$ with $k\ge k_0$,
$$\sum_{i\in\nn}\left(\sum_{j\in
F^{k+1}_2}P^{k+1}_{i,\,j}\zez^{k+1}_j\r)=0,$$
where the series converges both in $\cd'(\rn)$ and pointwise.
\end{lemma}

The following lemma gives the weighted atomic decomposition for a
dense subspace of $h^{\bfai}_{\oz,\,N}(\rn)$.

\begin{lemma}\label{l5.4}
Let $\bfai$ satisfy Assumption $\mathrm{(A)}$, $\oz\in
A^{\loc}_{\fz}(\rn)$, $q_{\oz}$, $p_{\bfai}$ and $N_{\bfai,\,\oz}$
be respectively as in \eqref{2.4}, \eqref{2.6} and \eqref{3.25}. If
$q\in(q_{\oz},\fz)$, integers $N\ge N_{\bfai,\,\oz}$, $s\ge\lfz
n(q_{\oz}/p_{\bfai}-1)\rf$ and $N>s$, then for any
$f\in(L^q_{\oz}(\rn)\cap h^{\bfai}_{\oz,\,N}(\rn))$, there exist
$\lz_0\in\cc$, $\{\lz^k_i\}_{k\ge k_0,\,i}\subset\cc$, a
$(\rz,\fz)_{\oz}$-single-atom $a_0$ and $(\rz,\fz,s)_{\oz}$-atoms
$\{a^k_i\}_{k\ge k_0,\,i}$ such that
\begin{eqnarray}\label{5.7}
f=\sum_{k\ge k_0}\sum_{i}\lz^k_i a^k_i+\lz_0 a_0,
\end{eqnarray}
where the series converges both  in $\cd'(\rn)$ and almost
everywhere. Moreover, there exists a positive constant $C$,
independent of $f$, such that
\begin{equation}\label{5.8}
\blz(\{\lz^k_i a^k_i\}_{k\ge k_0,\,i}\cup\{\lz_0 a_0\}) \le
C\|f\|_{h^{\bfai}_{\oz,\,N}(\rn)}.
\end{equation}
\end{lemma}

\begin{proof}
Let $f\in (L^q_{\oz}(\rn)\cap h^{\bfai}_{\oz,\,N}(\rn))$. We first
consider the case that $k_0 =-\fz$. As above, for each $k\in\zz$,
$f$ has a Calder\'on-Zygmund decomposition of degree $s$ and height
$\lz=2^k$ associated to $\cg_N (f)$ as in \eqref{5.1}, namely,
$$f=g^k +\sum_{i}b_i^k.$$
By Corollary \ref{c4.8} and
Proposition \ref{3.2}, we have that $g^k \to f$ in both
$h^{\bfai}_{\oz,\,N}(\rn)$ and $\cd'(\rn)$ as $k\to\fz$. By
Lemma \ref{l4.7}(i), $\|g^k\|_{L^q_{\oz}(\rn)}\to0$ as
$k\to-\fz$, and furthermore, by Lemma \ref{l2.6}(ii),
$g^k\to0$ in $\cd'(\rn)$ as $k\to-\fz$. Therefore,
\begin{eqnarray}\label{5.9}
f=\sum_{k=-\fz}^{\fz}\left(g^{k+1}-g^k\r)
\end{eqnarray}
in $\cd'(\rn)$. Moreover,
since $\supp(\sum_i b^k_i)\subset\boz_{2^k}$ and $\oz(\boz_{2^k})\to 0$ as
$k\to\fz$, then $g^k\to f$ almost everywhere as
$k\to\fz$. Thus, \eqref{5.9} also holds almost everywhere.
By Lemma \ref{l5.3} and \eqref{5.2} with $\boz_{2^{k+1}}\subset
\boz_{2^k}$,
\begin{eqnarray}\label{5.10}
g^{k+1}-g^k&=&\left(f-\sum_j b^{k+1}_j\r)-\left(f-\sum_i b_i^k\r)\nonumber\\
 &=&\sum_i b_i^k-\sum_j b^{k+1}_j+\sum_i\left(\sum_{j\in
F^{k+1}_2}P^{k+1}_{i,\,j}\zez^{k+1}_j \r)\nonumber\\  &=&\sum_i
\left[b_i^k-\sum_j b^{k+1}_j \zez^k_i +\sum_{j\in
F^{k+1}_2}P^{k+1}_{i,\,j}\zez^{k+1}_j\r]\equiv\sum_i h^k_i,
\end{eqnarray}
where all the series converge in both $\cd'(\rn)$ and almost
everywhere. Furthermore, from the definitions of $b_j^k$ and
$b_j^{k+1}$ as in \eqref{4.3}, we infer that when $l^k_i\in(0,1)$,
\begin{eqnarray}\label{5.11}
h^k_i=f\chi_{\boz_{2^{k+1}}^\complement}\zez^k_i-P^k_i \zez^k_i
+\sum_{j\in F^{k+1}_2} P^{k+1}_j \zez^k_i \zez^{k+1}_j+\sum_{j\in
F^{k+1}_2}P^{k+1}_{i,\,j}\zez^{k+1}_j,
\end{eqnarray}
and when $l^k_i\in[1,\fz)$,
\begin{eqnarray}\label{5.12}
h^k_i=f\chi_{\boz_{2^{k+1}}^\complement}\zez^k_i +\sum_{j\in
F^{k+1}_2} P^{k+1}_j \zez^k_i \zez^{k+1}_j+\sum_{j\in
F^{k+1}_2}P^{k+1}_{i,\,j}\zez^{k+1}_j.
\end{eqnarray}
By Proposition \ref{p3.2}(i), we know that  for almost every
$x\in\boz_{2^{k+1}}^\complement$,
$$|f(x)|\le\cg_N (f)(x)\le2^{k+1},$$
which, together with Lemma \ref{l4.1}, Lemma \ref{l5.1}(ii),
\eqref{5.5}, Lemma \ref{5.2}, \eqref{5.11} and \eqref{5.12}, implies
that there exists a positive constant $C_{10}$ such that for all
$i\in\nn$,
\begin{eqnarray}\label{5.13}
\|h^k_i\|_{L^{\fz}_{\oz}(\rn)}\le C_{10} 2^k.
\end{eqnarray}
Next, we show that for each $i$ and $k$, $h^k_i$ is a multiple of a
$(\rz,\,\fz,\,s)_{\oz}$-atom by considering the following two cases
for $i$.

{\it Case 1)} $i\in E^k_1$.  In this case, from the fact that
$l^{k+1}_j <1$ for $j\in F^{k+1}_2$, we deduce that $Q^{(k+1)\ast}_j
\subset Q(x^k_i,a(l^k_i +2))$ for $j$ satisfying $Q^{k\ast}_i\cap
Q^{(k+1)\ast}_j\neq\emptyset$.  Let $\gz\equiv1+2^{-12-n}$. Thus,
when $l^k_i \ge 2/(\gz-1)$, if letting $\wz{Q}^k_i\equiv
Q(x^k_i,a(l^k_i +2))$, then,
$$\supp(h^k_i)\subset \wz{Q}^k_i\subset\gz Q^{k\ast}_i
\subset\boz_{2^k}.$$ When $l^k_i <2/(\gz-1)$, if letting
$\wz{Q}^k_i\equiv2^6 n Q^{k\ast}_i$, then by Lemma \ref{5.1}(i),
we have
$$\supp (h^k_i)\subset\wz{Q}^k_i\subset\boz_{2^k}.$$
From the definition of $\wz{Q}^k_i$, Lemma \ref{l2.3}(v) and Remark
\ref{r2.4} with $\wz{C}\equiv 2/(\gz-1)$, we infer that there exists
a positive constant $C_{11}$ such that
\begin{eqnarray}\label{5.14}
\oz(\wz{Q}^k_i)\le C_{11}\oz(Q^{k\ast}_i).
\end{eqnarray}
Let $\wz{A}_1\equiv\max\{C_{10},\,C_{11}\}$,
\begin{eqnarray}\label{5.15}
\lz^k_i\equiv \wz{A}_1 2^k \oz(\wz{Q}^k_i)\rz(\oz(\wz{Q}^k_i))
\end{eqnarray}
and $a^k_i\equiv(\lz^k_i)^{-1}h^k_i$. From \eqref{5.13} and $\supp
(h^k_i)\subset \wz{Q}^k_i$ with $l(\wz{Q}^k_i)\ge2a>1$, it follows
that $a^k_i$ is a $(\rz,\,\fz,\,s)_{\oz}$-atom.

{\it Case 2)} $i\in E^k_2$. In this case, if $j\in F^{k+1}_1$, then
$l^k_i<l^{k+1}_j /(2^4 n)$. By Lemma \ref{l5.1}(i), we know that
$Q^{k\ast}_i \cap Q^{(k+1)\ast}_j=\emptyset$ for $j\in F^{k+1}_1$.
From this, \eqref{5.2} and \eqref{5.10}, we conclude that
\begin{eqnarray}\label{5.16}
h^k_i&=&\left(f-P^k_i\r)\zez^k_i-\sum_{j\in F^{k+1}_1}f\zez^{k+1}_j
\zez^k_i -\sum_{j\in F^{k+1}_2}\left(f-P^{k+1}_{j}\r)\zez^{k+1}_j
\zez^k_i\nonumber \\ &&+\sum_{j\in F^{k+1}_2}P^{k+1}_{i,\,j}\zez^{k+1}_j
\nonumber \\
&=&\left(f-P^k_i\r)\zez^k_i -\sum_{j\in
F^{k+1}_2}\left\{\left(f-P^{k+1}_{j}\r)\zez^{k+1}_j \zez^k_i
-P^{k+1}_{i,\,j} \zez^{k+1}_j\r\}.
\end{eqnarray}
Let $\wz{Q}^k_i\equiv2^6 n Q^{k\ast}_i$. Then $\supp
(h^k_i)\subset\wz{Q}^k_i$. By $l^k_i<1/(2^4 n)$,  Lemma \ref{l2.3}(v)
and Remark \ref{r2.4} with $\wz{C}\equiv4a$, we know that there
exists a positive constant $C_{12}$ such that
\begin{eqnarray}\label{5.17}
\oz(\wz{Q}^k_i)\le C_{12}\oz(Q^{k\ast}_i).
\end{eqnarray}
Moveover, $h^k_i$ satisfies the desired moment conditions, which are
deduced from the moment conditions of $(f-P^k_i)\zez^k_i$ (see
\eqref{5.3}) and $(f-P^{k+1}_{j})\zez^{k+1}_j
\zez^k_i-P^{k+1}_{i,\,j}\zez^{k+1}_j$ (see \eqref{5.4}). Let
$\wz{A}_2\equiv\max\{C_{10},\,C_{12}\}$,
\begin{eqnarray}\label{5.18}
\lz^k_i\equiv \wz{A}_2 2^k \oz(\wz{Q}^k_i)\rz(\oz(\wz{Q}^k_i))
\end{eqnarray}
and $a^k_i\equiv(\lz^k_i)^{-1}h^k_i$. By this, \eqref{5.13}, $\supp
(h^k_i)\subset \wz{Q}^k_i$ and the moment conditions of $h^k_i$, we
know that $a^k_i$ is a $(\rz,\,\fz,\,s)_{\oz}$-atom.

Thus, from \eqref{5.9}, \eqref{5.10}, Case 1) and Case 2), we
infer that
$$f=\sum_{k\in\zz}\sum_{i\in\nn}\lz^k_i a^k_i$$
holds in both $\cd'(\rn)$ and almost everywhere, where for every $k$ and $i$,
$\lz^k_i\in\cc$ and $a^k_i$ is a $(\rz,\,\fz,\,s)_{\oz}$-atom, which
shows \eqref{5.7} in the case that $k_0=-\fz$ by letting $\lz_0=0$.
Furthermore, by the fact that $\bfai(t)\sim\int_0^t
\frac{\bfai(s)}{s}\,ds$ for all $t\in(0,\fz)$, \eqref{5.15},
\eqref{5.18}, \eqref{5.14}, \eqref{5.17}, the upper type 1 property
of $\bfai$, Fubini's theorem and the bounded interior property of
$\{Q_i^{k\ast}\}$, we know that for any $\lz\in(0,\fz)$,
\begin{eqnarray*}
&&\sum_{k,\,i}\oz(\wz{Q}^k_i)\bfai\left(\frac{|\lz^k_i|}
{\lz\rz(\oz(\wz{Q}^k_i))\oz(\wz{Q}^k_i)}\r)\\
&&\hs \ls\sum_{k,\,i}\oz(\wz{Q}^k_i)\bfai\left(\frac{2^k}{\lz}\r)
\ls\sum_{k,i}\oz(Q^{k\ast}_i)\bfai\left(\frac{2^k}{\lz}\r)\\
&&\hs\ls\sum_{k}\oz(\boz_{2^k})\bfai\left(\frac{2^k}{\lz}\r)
\sim\sum_{k}\int_{\boz_{2^k}}\bfai\left(\frac{2^k}{\lz}\r)\oz(x)\,dx\\
&&\hs\ls\int_{\rn}\sum_{k<\log[\cg_N
(f)(x)]}\bfai\left(\frac{2^k}{\lz}\r)\oz(x)\,dx
\ls\int_{\rn}\sum_{k<\log[\cg_N
(f)(x)]}\int^{2^{k+1}}_{2^k}\bfai\left(\frac{t}{\lz}\r)\frac{dt}{t}\oz(x)\,dx\\
&&\hs\ls\int_{\rn}\int^{2\cg_N
(f)(x)/\lz}_{0}\bfai(t)\frac{dt}{t}\oz(x)\,dx
\ls\int_{\rn}\bfai\left(\frac{\cg_N (f)(x)}{\lz}\r)\oz(x)\,dx,
\end{eqnarray*}
which implies \eqref{5.8} in the case that $k_0 =-\fz$.

Finally, we consider the case that $k_0>-\fz$. In this case, by
$f\in h^{\bfai}_{\oz,\,N}(\rn)$, we see that $\oz(\rn)<\fz$.
Adapting the previous arguments, we conclude that
\begin{eqnarray}\label{5.19}
f=\sum_{k=k_0}^{\fz}\left(g^{k+1}-g^k\r)+g^{k_0} \equiv\wz{f}+g^{k_0},
\end{eqnarray}
and for the function $\wz{f}$, we have the same
$(\rz,\,\fz,\,s)_{\oz}$-atomic decomposition as above that
\begin{eqnarray}\label{5.20}
\wz{f}= \sum_{k\ge k_0,\,i} \lz^k_i a^k_i
\end{eqnarray} and
\begin{eqnarray}\label{5.21}
 &&\blz\left(\{\lz^k_i
a^k_i\}_{k\ge k_0,\,i}\r)\ls\|f\|_{h^{\bfai}_{\oz,\,N}(\rn)}.
\end{eqnarray}
From Lemma \ref{l4.7}(ii), it follows that
\begin{eqnarray}\label{5.22}
\|g^{k_0}\|_{L^{\fz}_{\oz}(\rn)}\le C_9 2^{k_0}\le2C_9
\inf_{x\in\rn}\cg_N (f)(x),
\end{eqnarray}
where $C_9$ is the same as in Lemma \ref{4.7}(ii). Let $\lz_0 \equiv
2C_92^{k_0}\oz(\rn)\rz(\oz(\rn))$ and
$$a_0\equiv\lz_0^{-1}g^{k_0}.$$
Then we have
$$\|a_0\|_{L^{\fz}_{\oz}(\rn)}\le[\oz(\rn)\rz(\oz(\rn))]^{-1}.$$
Thus, we know that $a_0$ is a $(\rz,\,\fz)_{\oz}$-single-atom and
$g^{k_0}=\lz_0 a_0$, which together with \eqref{5.19} and
\eqref{5.20} implies \eqref{5.7} in the case that $k_0>-\fz$.
Moreover, from \eqref{5.22}, we deduce that for any $\lz\in(0,\fz)$,
\begin{eqnarray*}
\oz(\rn)\bfai\left(\frac{|\lz_0|}{\lz\oz(\rn)\rz(\oz(\rn))}\r)
=\oz(\rn)\bfai\left(\frac{C_9 2^{k_0}}{\lz}\r)\ls
\int_{\rn}\bfai\left(\frac{\cg_N (f)(x)}{\lz}\r)\oz(x)\,dx,
\end{eqnarray*}
which together with \eqref{5.21} implies \eqref{5.8} in the case
that $k_0>-\fz$. This finishes the proof of Lemma \ref{l5.4}.
\end{proof}

\begin{remark}\label{r5.5}
By its proof, all $(\rz,\,\fz,\,s)_{\oz}$-atoms in Lemma \ref{l5.4}
can be taken to have supports $Q$ satisfying $l(Q)\in(0,2]$. Indeed,
for any $(\rz,\,\fz,\,s)_{\oz}$-atom $a$ supported in a cube $Q_0$
with $l(Q_0)>2$, we know that there exist $N_0\in\nn$, depending on
$l(Q_0)$ and $n$, and cubes $\{Q_i\}_{i=1}^{N_0}$ satisfying
$l(Q_i)\in[1,2]$ with $i\in\{1,\,\cdots,\,N_0\}$ such that
$\cup_{i=1}^{N_0}Q_i=Q_0$, for any $x\in Q_0$,
$1\le\sum_{i=1}^{N_0}\chi_{Q_i}(x)\le C(n)$, and
$$a=\frac{1}{\sum_{j=1}^{N_0}\chi_{Q_j}}\sum_{i=1}^{N_0}a\chi_{Q_i},$$
where $C(n)$ is a positive integer, only depending on $n$. For any
given $\lz_0\in\cc$ and $i\in\{1,\,\cdots,\,N_0\}$, let
$$\gz_i\equiv\frac{\lz_0
\oz(Q_i)\rz(\oz(Q_i))}{\oz(Q_0)\rz(\oz(Q_0))}$$
and
$$b_i\equiv\frac{\oz(Q_0)\rz(\oz(Q_0))a\chi_{Q_i}}
{\oz(Q_i)\rz(\oz(Q_i))\sum_{i=1}^{N_0}\chi_{Q_i}}.$$
Then for any $i\in\{1,\,2,\,\cdots,\,N_0\}$, $b_i$ is a
$(\rz,\,\fz,\,s)_{\oz}$-atom supported in the cube $Q_i$ and
\begin{eqnarray}\label{5.23}
\lz_0 a=\sum_{i=1}^{N_0}\gz_i b_i.
\end{eqnarray}
From the definitions of $\gz_i$ and $b_i$,
$\cup_{i=1}^{N_0}Q_i=Q_0$, and for any $x\in Q_0$,
$$1\le\sum_{i=1}^{N_0}\chi_{Q_i}(x)\le C(n),$$
we also conclude that for all $\lz\in(0,\fz)$,
\begin{eqnarray}\label{5.24}
\sum_{i=1}^{N_0}\oz(Q_i)\bfai\left(\frac{|\gz_i|}{\lz
\oz(Q_i)\rz(\oz(Q_i))}\r)\le C(n) \oz(Q_0)\bfai\left(\frac{|\lz_0|}{
\lz\oz(Q_0)\rz(\oz(Q_0))}\r).
\end{eqnarray}
Thus, by the proof of Lemma \ref{l5.4}, \eqref{5.23} and
\eqref{5.24}, we see that the claim holds.
\end{remark}

Now we state the weighted atomic decompositions of
$h^{\bfai}_{\oz,\,N}(\rn)$ as follows.

\begin{theorem}\label{t5.6}
Let $\bfai$ satisfy Assumption $\mathrm{(A)}$, $\oz\in
A^{\loc}_{\fz}(\rn)$, and $q_{\oz}$ and $N_{\bfai,\,\oz}$ be
respectively as in \eqref{2.4} and \eqref{3.25}. If
$q\in(q_{\oz},\fz]$, integers $s$ and $N$ satisfy $N\ge
N_{\bfai,\,\oz}$ and $N>s\ge\lfz n(\frac{q_{\oz}}{p_{\bfai}}-1)\rf$,
then
$$h^{\rz,\,q,\,s}_{\oz}(\rn)=h^{\bfai}_{\oz,\,N}(\rn)=
h^{\bfai}_{\oz,\,N_{\bfai,\,\oz}}(\rn)$$
with equivalent norms.
\end{theorem}

\begin{proof}
It is easy to see that $$ h^{\rz,\,\fz,\,s_1}_{\oz}(\rn)\subset
h^{\rz,\,q,\,s}_{\oz}(\rn)\subset
h^{\bfai}_{\oz,\,N_{\bfai,\,\oz}}(\rn)\subset
h^{\bfai}_{\oz,\,N}(\rn)\subset h^{\bfai}_{\oz,\,N_1}(\rn),$$ where
the integers $s_1$ and $N_1$ are respectively no less than $s$ and
$N$, and the inclusions are continuous. Thus, to prove Theorem
\ref{t5.6}, it suffices to prove that for any integers $N,\,s$
satisfying $N>s\ge\lfz n(\frac{q_{\oz}}{p_{\bfai}}-1)\rf$,
$h^{\bfai}_{\oz,\,N}(\rn)\subset h^{\rz,\,\fz,\,s}_{\oz}(\rn)$ and
for all $f\in h^{\bfai}_{\oz,\,N}(\rn)$,
$$\|f\|_{h^{\rz,\,\fz,\,s}_{\oz}(\rn)}\ls
\|f\|_{h^{\bfai}_{\oz,\,N}(\rn)}.$$

Let $f\in h^{\bfai}_{\oz,\,N}(\rn)$. By Corollary \ref{c4.8}, there
exists a sequence of functions,
$$\{f_m\}_{m\in\nn}\subset(h^{\bfai}_{\oz,\,N}(\rn)\cap
L^q_{\oz}(\rn)),$$
such that for all $m\in\nn$,
\begin{eqnarray}\label{5.25}
\|f_m\|_{h^{\bfai}_{\oz,\,N}(\rn)}\le2^{-m}
\|f\|_{h^{\bfai}_{\oz,\,N}(\rn)}
\end{eqnarray}
and $f=\sum_{m\in\nn} f_m$ in $h^{\bfai}_{\oz,\,N}(\rn)$. By Lemma
\ref{l5.4}, we know that for each $m\in\nn$, $f_m$ has an atomic
decomposition
$$f=\sum_{i\in\zz_+}\lz^m_i a^m_i$$
in $\cd'(\rn)$ with
$$\blz(\{\lz^m_i a^m_i\}_{i})\ls\|f_m\|_{h^{\bfai}_{\oz,\,N}(\rn)},$$
where $\{\lz^m_i\}_{i\in\zz_+}\subset\cc$, $\{a^m_i\}_{i\in\nn}$ are
$(\rz,\,\fz,\,s)_{\oz}$-atoms and $a^m_0$ is a
$(\rz,\,\fz)_{\oz}$-single-atom.

Let
$$\wz{\lz}_0\equiv\oz(\rn)\rz(\oz(\rn))\sum^{\fz}_{m=1}|\lz^m_0|
\|a^m_0\|_{L^{\fz}_{\oz}(\rn)}$$
and
$$\wz{a}_0\equiv\left(\wz{\lz}_0\r)^{-1}\sum^{\fz}_{m=1}\lz^m_0 a^m_0.$$
Then
$$\wz{\lz}_0\wz{a}_0=\sum^{\fz}_{m=1}\lz^m_0 a^m_0.$$
It is easy to see that
$$\|\wz{a}_0\|_{L^{\fz}_{\oz}(\rn)}\le
[\oz(\rn)\rz(\oz(\rn))]^{-1},$$
which implies that $\wz{a}_0$ is a
$(\rz,\fz)_{\oz}$-single-atom. Since $\bfai$ is increasing, by
\eqref{5.8}, we know that for any $m\in\nn$,
\begin{eqnarray}\label{5.26}
\oz(\rn)\bfai\left(\frac{|\lz^m_0|}{C\|f_m\|_
{h^{\bfai}_{\oz,\,N}(\rn)}\oz(\rn)\rz(\oz(\rn))}\r)\le1,
\end{eqnarray}
where $C$ is as in \eqref{5.8}. Let
$$\wz{\gz}\equiv
C\left(\sum_{i=m}^{\fz}\|f_m\|_{h^{\bfai}_{\oz,\,N}
(\rn)}^{p_{\bfai}}\r)^{1/p_{\bfai}},$$
where $C$ is as in \eqref{5.8}. Then, from the
continuity and subadditivity of $\bfai$, the strictly lower type
$p_{\bfai}$ property of $\bfai$ and \eqref{5.26}, it follows that
\begin{eqnarray*}
&&\oz(\rn)\bfai\left(\frac{|\wz{\lz}_0|}{\wz{\gz}\oz(\rn)\rz(\oz(\rn))}\r)\\
&&\hs=\oz(\rn)\bfai\left(\frac{\sum^{\fz}_{m=1}|\lz^m_0|
\|a^m_0\|_{L^{\fz}_{\oz}(\rn)}}{\wz{\gz}}\r)\\
&&\hs\le\oz(\rn)\sum^{\fz}_{m=1}\frac{\|f_m\|
_{h^{\bfai}_{\oz,\,N}(\rn)}^{p_{\bfai}}}{\wz{\gz}^{p_{\bfai}}}
\bfai\left(\frac{|\lz^m_0|}{C\|f_m\|
_{h^{\bfai}_{\oz,\,N}(\rn)}\oz(\rn)\rz(\oz(\rn))}\r)\le1,
\end{eqnarray*}
which together with \eqref{5.25} implies that
$$\blz(\{\wz{\lz}_0\wz{a}_0\})\le\wz{\gz}\ls\|f\|
_{h^{\bfai}_{\oz,\,N}(\rn)}.$$
Thus, we see that
$$f=\sum_{m\in\nn}\sum_{i\in\nn}\lz^m_i a^m_i+\wz{\lz}_0\wz{a}_0\in
h^{\rz,\,\fz,\,s}_{\oz}(\rn)$$
and
$$\|f\|_{h^{\rz,\,\fz,\,s}_{\oz}(\rn)}\ls
\|f\|_{h^{\bfai}_{\oz,\,N}(\rn)}.$$
This finishes the proof of
Theorem \ref{t5.6}.
\end{proof}

\begin{remark}\label{r5.7}
Let $p\in(0,1]$. Theorem \ref{5.1} when $\bfai(t)\equiv t^p$ for all
$t\in(0,\fz)$  was obtained by Tang \cite[Theorem\,5.1]{Ta1}.
\end{remark}

For simplicity, from now on, we denote by $h^{\bfai}_{\oz}(\rn)$ the
{\it weighted local Orlicz-Hardy space $h^{\bfai}_{\oz,\,N}(\rn)$}
when $N\ge N_{\bfai,\,\oz}$.

\chapter{Finite atomic decompositions\label{s6}}

\hskip\parindent In this section, we prove that for any given finite
linear combination of weighted atoms when $q <\fz$ (or continuous
$(\rz,\,q,\,s)_{\oz}$-atoms when $q=\fz$), its norm in
$h^{\bfai}_{\oz,\,N}(\rn)$ can be achieved via all its finite
weighted atomic decompositions. This extends the main results in
\cite{msv08,yz09} to the setting of weighted local Orlicz-Hardy
spaces. As applications, we see that for a given admissible
triplet $(\rho,\,q,\,s)_{\omega}$ and a $\beta$-quasi-Banach space
$\mathcal{B}_{\beta}$ with $\beta\in(0,1]$, if $T$ is a
$\mathcal{B}_{\beta}$-sublinear operator, and maps all
$(\rho,\,q,\,s)_{\omega}$-atoms and
$(\rho,\,q)_{\omega}$-single-atoms with $q<\infty$ (or all
continuous $(\rho,\,q,\,s)_{\omega}$-atoms with $q=\infty$) into
uniformly bounded elements of $\mathcal{B}_{\beta}$, then $T$
uniquely extends to a bounded $\mathcal{B}_{\beta}$-sublinear
operator from $h^{\Phi}_{\omega}(\mathbb{R}^n)$ to
$\mathcal{B}_{\beta}$.

\begin{definition}\label{d6.1}
Let $\bfai$ satisfy Assumption (A), $\oz\in A^{\loc}_{\fz}(\rn)$ and
$(\rz,\,q,\,s)_{\oz}$ be admissible as in Definition \ref{d3.4}.
The \emph{space} $h^{\rz,\,q,\,s}_{\oz,\,\fin}(\rn)$ is defined
to be the vector space of all finite linear combinations
of $(\rz,\,q,\,s)_{\oz}$-atoms and
a $(\rz,\,q)_{\oz}$-single-atom, and the \emph{norm} of $f$
in $h^{\rz,\,q,\,s}_{\oz,\,\fin}(\rn)$ is defined by
\begin{eqnarray*}
&&\|f\|_{h^{\rz,\,q,\,s}_{\oz,\,\fin}(\rn)}\equiv\inf\left\{\blz(\{\lz_i
a_i\}_i):\,f=\sum_{i=0}^{k}\lz_i a_i,\,k\in\zz_+,\ \{\lz_i\}_{i=0}^k
\subset\cc,\,\ \{a_i\}_{i=1}^{k}\,\,\text{are}\r. \\
&&\hspace{9 em}(\rz,\,q,\,s)_{\oz}\text{-atoms}\ \text{and}\,\,a_0
\,\,\text{is a}\,\, (\rz,\,q)_{\oz}\text{-single-atom}\Bigg\}.
\end{eqnarray*}
\end{definition}

Obviously, for any admissible triplet $(\rz,\,q,\,s)_{\oz}$,
$h^{\rz,\,q,\,s}_{\oz,\,\fin}(\rn)$ is dense in
$h^{\rz,\,q,\,s}_{\oz}(\rn)$ with respect to the quasi-norm
$\|\cdot\|_{h^{\rz,\,q,\,s}_{\oz}(\rn)}$.

\begin{theorem}\label{t6.2}
Let $\bfai$ satisfy Assumption $\mathrm{(A)}$, $\oz\in
A^{\loc}_{\fz}(\rn)$, $q_{\oz}$ be as in \eqref{2.4} and
$(\rz,\,q,\,s)_{\oz}$ be admissible as in Definition \ref{d3.4}.

$\mathrm{(i)}$ If $q\in(q_{\oz},\fz)$, then
$\|\cdot\|_{h^{\rz,\,q,\,s}_{\oz,\,\fin}(\rn)}$ and
$\|\cdot\|_{h^{\bfai}_{\oz}(\rn)}$ are equivalent quasi-norms on
$h^{\rz,\,q,\,s}_{\oz,\,\fin}(\rn)$.

$\mathrm{(ii)}$ Let $h^{\rz,\,\fz,\,s}_{\oz,\,\fin,\,c}(\rn)$ denote
the set of all $f\in h^{\rz,\,\fz,\,s}_{\oz,\,\fin}(\rn)$ with
compact support. Then
$\|\cdot\|_{h^{\rz,\,\fz,\,s}_{\oz,\,\fin}(\rn)}$ and
$\|\cdot\|_{h^{\bfai}_{\oz}(\rn)}$ are equivalent quasi-norms on
$h^{\rz,\,\fz,\,s}_{\oz,\,\fin,\,c}(\rn)\cap C(\rn)$.
\end{theorem}

\begin{proof}
We first show (i). Let $q\in(q_{\oz},\fz)$ and $(\rz,\,q,\,s)_{\oz}$
be admissible. Obviously, from Theorem \ref{t5.6}, we infer that
$h^{\rz,\,q,\,s}_{\oz,\,\fin}(\rn)\subset
h^{\rz,\,q,\,s}_{\oz}(\rn)=h^{\bfai}_{\oz}(\rn)$ and for all $f\in
h^{\rz,\,q,\,s}_{\oz,\,\fin}(\rn)$,
$$\|f\|_{h^{\bfai}_{\oz}(\rn)}\ls
\|f\|_{h^{\rz,\,q,\,s}_{\oz,\,\fin}(\rn)}.$$ Thus, we only need show
that for all $f\in h^{\rz,\,q,\,s}_{\oz,\,\fin}(\rn)$,
\begin{eqnarray}\label{6.1}
\|f\|_{h^{\rz,\,q,\,s}_{\oz,\,\fin}(\rn)}\ls
\|f\|_{h^{\bfai}_{\oz}(\rn)}.
\end{eqnarray}
By homogeneity, without loss of generality, we may assume that $f\in
h^{\rz,\,q,\,s}_{\oz,\,\fin}(\rn)$ with
$\|f\|_{h^{\bfai}_{\oz}(\rn)}=1$. In the rest of this section, for
any $f\in h^{\rz,\,q,\,s}_{\oz,\,\fin}(\rn)$, let $k_0$ be as in
Section \ref{s5} and $\boz_{2^k}$ with $k\ge k_0$ as in \eqref{4.1}
with $\lz=2^k$. Since $f\in(h^{\bfai}_{\oz,\,N}(\rn)\cap
L^{q}_{\oz}(\rn))$, by Lemma \ref{l5.4}, there exist $\lz_0\in\cc$,
$\{\lz^k_i\}_{k\ge k_0,\,i}\subset\cc$, a
$(\rz,\,\fz)_{\oz}$-single-atom $a_0$ and
$(\rz,\,\fz,\,s)_{\oz}$-atoms $\{a^k_i\}_{k\ge k_0,\,i}$, such that
\begin{eqnarray}\label{6.2}
f=\sum_{k\ge k_0}\sum_{i}\lz^k_i a^k_i+\lz_0 a_0
\end{eqnarray}
holds both in $\cd'(\rn)$ and almost everywhere.  First, we claim
that \eqref{6.2} also holds in $L^q_{\oz}(\rn)$. For any $x\in\rn$,
by $\rn=\cup_{k\ge k_0}(\boz_{2^k}\setminus \boz_{2^{k+1}})$, we see
that there exists $j\in\zz$ such that $x\in(\boz_{2^j}\setminus
\boz_{2^{j+1}})$. By the proof of Lemma \ref{l5.4}, we know that for
all $k>j$, $\supp(a^k_i)\subset
\wz{Q}^k_i\subset\boz_{2^k}\subset\boz_{2^{j+1}}$; then from
\eqref{5.13} and \eqref{5.22}, we conclude that
$$\left|\sum_{k\ge k_0}\sum_{i}\lz^k_i a^k_i
(x)\r|+|\lz_0 a_0 (x)|\ls\sum_{k_0\le k\le j}2^k +2^{k_0}\ls
2^j\ls\cg_N (f)(x).$$ Since $f\in L^q_{\oz}(\rn)$, from Proposition
\ref{p3.2}(ii), we infer that $\cg_N (f)(x)\in L^q_{\oz}(\rn)$. This
combined with the Lebesgue dominated convergence theorem implies that
$$\sum_{k\ge k_0}\sum_{i}\lz^k_i a^k_i+\lz_0 a_0$$
converges to $f$ in $L^q_{\oz}(\rn)$, which completes
the proof of the claim.

Next, we show \eqref{6.1} by considering the following two cases for
$\oz$.

{\it Case 1)} $\oz(\rn)=\fz$. In this case, by $f\in
L^{q}_{\oz}(\rn)$, we know that $k_0=-\fz$ and $a_0 (x)=0$ for
almost every $x\in\rn$ in \eqref{6.2}. Thus, in this case,
\eqref{6.2} has the version
$$f=\sum_{k\in\zz} \sum_{i}\lz^k_ia^k_i.$$
Since, when $\oz(\rn)=\fz$, all
$(\rz,\,q)_{\oz}$-single-atoms are 0, if $f\in
h^{\rz,\,q,\,s}_{\oz,\,\fin}(\rn)$, then $f$ has compact support.
Assume that $\supp(f)\subset Q_0 \equiv Q(x_0,r_0)$ and
$$\wz{Q}_0\equiv Q(x_0,\sqrt{n}r_0 +2^{3(10+n)+1}).$$
Then for any $\pz\in\cd_N (\rn)$, $x\in\rn\setminus\wz{Q}_0$ and $t\in(0,1)$,
we have
\begin{eqnarray*}
\pz_t \ast f(x)=\int_{Q(x_0,r_0)}\pz_t
(x-y)f(y)\,dy=\int_{B(x,2^{3(10+n)})\cap Q(x_0,r_0)}\pz_t
(x-y)f(y)\,dy=0.
\end{eqnarray*}
Thus, for any $k\in\zz$, $\boz_{2^k}\subset\wz{Q}_0$, which implies
that  $\supp(\sum_{k\in\zz} \sum_{i}\lz^k_i a^k_i)\subset \wz{Q}_0$.
For each positive integer $K$, let
$$F_K\equiv\{(i,k):\,k\in\zz,\,k\ge k_0, i\in\nn,\,|k|+i\le K\}$$
and
$$f_K\equiv\sum_{(k,i)\in F_K}\lz^k_i a^k_i.$$
Then, by the above claim, we know that $f_K$ converges
to $f$ in $L^q_{\oz}(\rn)$.
Thus, for any given $\epz\in(0,1)$, there exists $K_0 \in\nn$ large
enough such that
$$\|(f-f_{K_0})/\epz\|_{L^q_{\oz}(\rn)}\le[\rz(\oz(\wz{Q}_0))]^{-1}
[\oz(\wz{Q}_0)]^{1/q-1},$$
which together with
$\supp(f-f_{K_0})/\epz\subset\wz{Q}_0$ implies that
$(f-f_{K_0})/\epz$ is a $(\rz,\,q,\,s)_{\oz}$-atom. Moreover, we
equivalently divide $\wz{Q}_0$ into the union of some cubes
$\{Q_i\}_{i=1}^{N_0}$ with disjoint interior and their sidelengths
satisfying $l_i\in(1,2]$, where $N_0$ depends only on $r_0$ and $n$.
It is clear that
$$\|(f-f_{K_0})\chi_{Q_i}/\epz\|_{L^q_{\oz}(\rn)}\le
[\rz(\oz(\wz{Q}_0))]^{-1}
[\oz(\wz{Q}_0)]^{1/q-1}\le[\rz(\oz(Q_i))]^{-1} [\oz(Q_i)]^{1/q-1},$$
which together with $\supp((f-f_{K_0})\chi_{Q_i}/\epz)\subset Q_i$
implies that $(f-f_{K_0})\chi_{Q_i}/\epz$ is a
$(\rz,\,q,\,s)_{\oz}$-atom for $i=1,\,2,\cdots,\,N_0$. Thus,
$$f=f_{K_0}+\sum_{i=1}^{N_0}(f-f_{K_0})\chi_{Q_i}$$
almost everywhere is a finite linear weighted atom combination of $f$. Let
$$b_i\equiv(f-f_{K_0})\chi_{Q_i}/\epz$$
and take $\epz\equiv N_0^{-1/p_{\bfai}}$. Then,
by \eqref{2.8} with $t\equiv\oz(Q_i)$,
Remark \ref{r3.6}(ii) and the lower type $p_{\bfai}$ property of
$\bfai$,
\begin{eqnarray*}
\|f\|_{h^{\rz,\,q,\,s}_{\oz,\,\fin}(\rn)}&\ls&\blz\left(\{\lz^k_i
a^k_i\}_{(i,k)\in F_{K_0}}\r)+\blz\left(\{\epz b_i\}_{i=1}^{N_0}\r)\\
&\ls&\|f\|_{h^{\rz,\,q,\,s}_{\oz}(\rn)}+\inf\left\{\lz>0:\,
\sum_{i=1}^{N_0}\oz(Q_i)\bfai\left(\frac{\epz}{\lz\oz(Q_i)
\rz(\oz(Q_i))}\r)\le1\r\}\ls1,
\end{eqnarray*}
which implies \eqref{6.1} in Case 1).

{\it Case 2)} $\oz(\rn)<\fz$. In this case, $f$ may not have compact
support. Similarly to Case 1), for any positive integer $K$, let
$$f_K\equiv\sum_{(k,i)\in F_K}\lz^k_i a^k_i+\lz_0 a_0$$
and $b_K\equiv f-f_K$, where $F_K$ is as in Case 1).
From the above claim, we
deduce that $f_K$ converges to $f$ in $L^q_{\oz}(\rn)$. Thus, there
exists a positive integer $K_1\in\nn$ large enough such that
$$\|b_{K_1}\|_{L^q_{\oz}(\rn)}\le [\rz(\oz(\rn))]^{-1}
[\oz(\rn)]^{1/q-1}.$$
Thus, $b_{K_1}$ is a
$(\rz,\,q)_{\oz}$-single-atom and $f=f_{K_1}+b_{K_1}$ is a finite
linear weighted atom combination of $f$. Moreover, by Remark
\ref{r3.6}(ii) and \eqref{2.8} with $t\equiv\oz(\rn)$,
\begin{eqnarray*}
\|f\|_{h^{\rz,\,q,\,s}_{\oz,\,\fin}(\rn)}&\ls&\blz\left(\{\lz^k_i
a^k_i\}_{(i,k)\in F_{K_1}}\r)+\blz\left(\{b_{K_1}\}\r)\\
&\ls&\|f\|_{h^{\rz,\,q,\,s}_{\oz}(\rn)}+\inf\left\{\lz>0:\,
\oz(\rn)\bfai\left(\frac{1}{\lz\oz(\rn)\rz(\oz(\rn))}\r)\le1\r\}\ls1,
\end{eqnarray*}
which implies \eqref{6.1} in Case 2). This finishes the proof of
(i).

We now prove (ii). In this case, similarly to the proof of (i), we
only need prove that for all $f\in
h^{\rz,\,\fz,\,s}_{\oz,\,\fin,\,c}(\rn)$,
$$\|f\|_{h^{\rz,\,\fz,\,s}_{\oz}(\rn)}\ls
\|f\|_{h^{\bfai}_{\oz}(\rn)}.$$
Again, by homogeneity, without loss
of generality, we may assume that $\|f\|_{h^{\bfai}_{\oz}(\rn)}=1$.
Since $f$ has compact support, by the definition of $\cg_N (f)$, it
is easy to know that $\cg_N (f)$ also has compact support. Assume
that $\supp(\cg_N (f))\subset B(0,R_0)$ for some $R_0 \in(0,\fz)$. By
$f\in L^{\fz}_{\oz}(\rn)$, we have that $\cg_N f\in
L^{\fz}_{\oz}(\rn)$. Thus, there exists $k_1\in\zz$ such that
$\boz_{2^k}=\emptyset$ for any $k\in\zz$ with $k\ge k_1 +1$. By
Lemma \ref{l5.4}, there exist $\lz_0\in\cc$, $\{\lz^k_i\}_{k_1 \ge
k\ge k_0,\,i}\subset\cc$, a $(\rz,\,\fz)_{\oz}$-single-atom $a_0$
and $(\rz,\,\fz,\,s)_{\oz}$-atoms $\{a^k_i\}_{k_1 \ge k\ge k_0,\,i}$
such that
$$f=\sum^{k_1}_{k=k_0}\sum_{i}\lz^k_i
a^k_i+\lz_0 a_0$$
holds both in $\cd'(\rn)$ and almost everywhere.
By the fact that $f$ is uniformly continuous, we know that for any
given $\varepsilon\in(0,\fz)$, there exists a $\delta\in(0,\fz)$
such that if
$$|x-y|<\sqrt{n}\delta/2,$$
then $|f(x)-f(y)|<\varepsilon$. Without loss of generality, we may assume
that $\delta<1$. Write $f=f^{\varepsilon}_1 +f^{\varepsilon}_2$ with
$$f^{\varepsilon}_1\equiv\sum_{(i,k)\in G_1}\lz^k_i a^k_i+\lz_0 a_0$$
and
$$f^{\varepsilon}_2\equiv\sum_{(i,k)\in G_2}\lz^k_i a^k_i,$$
where
$$G_1\equiv\left\{(i,k):\,l(\wz{Q}^k_i)\ge\delta,\,k_0 \le k\le k_1\r\},$$
$$G_2\equiv\left\{(i,k):\,l(\wz{Q}^k_i)<\delta,\,k_0 \le k\le k_1\r\},$$
and $\wz{Q}^k_i$ is the support of $a^k_i$ (see the proof of Lemma
\ref{l5.4}). For any fixed integer $k\in[k_0,k_1]$, by Lemma
\ref{l5.1}(ii) and $\boz_{2^k}\subset B(0,R_0)$, we see that $G_1$
is a finite set.

For any $(i,k)\in G_2$ and
$x\in\wz{Q}^k_i,\,|f(x)-f(x^k_i)|<\varepsilon$. For all $x\in\rn$,
let
$$\wz{f}(x)\equiv[f(x)-f(x^k_i)]\chi_{\wz{Q}^k_i}(x)$$
and $\wz{P}^k_i (x)\equiv P^k_i (x)-f(x^k_i)$. By the definition of
$P^k_i$,  for all $P\in\cp_s (\rn)$,
$$\int_{\rn}\left[\wz{f}(x)-\wz{P}^k_i(x)\r]P(x)\zez^k_i
(x)\,dx=0.$$ Since $|\wz{f}(x)|<\varepsilon$ for all $x\in\rn$
implies that $\cg_N (\wz{f})(x)\ls\varepsilon$ for all $x\in\rn$,
then by Lemma \ref{l4.1}, we see that
\begin{eqnarray}\label{6.3}
\sup_{y\in\rn}\left|\wz{P}^k_i(y)\zez^k_i (y)\r|\ls
\sup_{y\in\rn}\left|\cg_N (\wz{f})(y)\r|\ls\varepsilon.
\end{eqnarray}
Let $\wz{P}^k_{i,\,j}\in\cp_s (\rn)$ be such that for any $P\in\cp_s
(\rn)$,
$$\int_{\rn}\left[\wz{f}(x)-\wz{P}^k_i(x)\r]\zez^k_i
(x)P(x)\zez^{k+1}_i (x)\,dx= \int_{\rn}\wz{P}^{k+1}_{i,\,j}
(x)P(x)\zez^{k+1}_j (x)\,dx.$$ Since
$(\wz{f}-\wz{P}^k_i)\zez^k_i=(f-P^k_i)\zez^k_i$, by $\supp
(\zez^k_i)\subset \wz{Q}^k_i$, we have $\wz{P}^k_{i,\,j}=P^k_{i,\,j}$.
Then from Lemma \ref{l5.2}, we deduce that
\begin{eqnarray}\label{6.4}
\sup_{y\in\rn}\left|\wz{P}^k_{i,\,j}(y)\zez^{k+1}_i (y)\r|\ls
\sup_{y\in\rn}\left|\cg_N (\wz{f})(y)\r|\ls\varepsilon.
\end{eqnarray} Thus, by the definition
of $\lz^k_i a^k_i$, $\sum_j \zez^{k+1}_j =\chi_{\boz_{2^{k+1}}}$ and
\eqref{5.11}, we know that
\begin{eqnarray*}
\lz^k_i a^k_i&=&f\chi_{\boz_{2^{k+1}}^\complement}\zez^k_i-P^k_i
\zez^k_i+\sum_{j\in F^{k+1}_2} P^{k+1}_j \zez^k_i
\zez^{k+1}_j+\sum_{j\in
F^{k+1}_2}P^{k+1}_{i,\,j}\zez^{k+1}_j\\
&=&\wz{f}\chi_{\boz_{2^{k+1}}^\complement}\zez^k_i-\wz{P}^k_i
\zez^k_i+\sum_{j\in F^{k+1}_2} \wz{P}^{k+1}_j \zez^k_i
\zez^{k+1}_j+\sum_{j\in F^{k+1}_2}\wz{P}^{k+1}_{i,\,j}\zez^{k+1}_j.
\end{eqnarray*}
From this together with \eqref{6.3}, \eqref{6.4} and Lemma
\ref{l5.1}(ii), it follows that $|\lz^k_i a^k_i|\ls\varepsilon$ for
all $x\in\wz{Q}^k_i$ with $(i,k)\in G_2$. Moreover, using Lemma
\ref{l5.1}(ii) again, we conclude that
$$|f^{\varepsilon}_2|\ls\sum_{k=k_0}^{k_1}\varepsilon\ls(k_1-k_0)
\varepsilon.$$ By the arbitrariness of $\varepsilon$, $\supp
(f^{\varepsilon}_2)\subset B(0,R_0)$ and
$|f^{\varepsilon}_2|\ls(k_1-k_0) \varepsilon$, we choose
$\varepsilon$ small enough such that $f^{\varepsilon}_2$ is an
arbitrarily small multiple of a $(\rz,\,\fz,\,s)_{\oz}$-atom. In
particular, we choose $\varepsilon_0\in(0,\fz)$ such that
$f^{\varepsilon_0}_2=\wz{\lz}\wz{a}$ with $|\wz{\lz}|\le1$ and
$\wz{a}$ is a $(\rz,\,\fz,\,s)_{\oz}$-atom, then
$$f=\sum_{(i,k)\in
G_1}\lz^k_i a^k_i+\lz_0 a_0+\wz{\lz}\wz{a}$$
is a finite weighted
atomic decomposition of $f$, and
$$\|f\|_{h^{\rz,\,\fz,\,s}_{\oz}(\rn)}\ls\|f\|_{h^{\bfai}_{\oz}(\rn)}+1\ls1,$$
which completes the proof of Theorem \ref{t6.2}.
\end{proof}

\begin{remark}\label{r6.3}
(i) From the proof of Theorem \ref{t6.2}, it follows that for any
$f\in h^{\rz,\,q,\,s}_{\oz,\,\fin}(\rn)$ with $q\in(q_{\oz},\fz)$,
there exist $\{\lz_j\}_{j=0}^k \subset\cc$, a
$(\rz,\,q)_{\oz}$-single-atom $a_0$ and $(\rz,\,q,\,s)_{\oz}$-atoms
$\{a_j\}_{j=1}^k$ satisfying $\supp(a_j)\subset Q_j$ with
$l(Q_j)\in(0,2]$ such that $f=\sum_{j=0}^k \lz_j a_j$ in both
$L^q_{\oz}(\rn)$ and $\cd' (\rn)$. Moreover, for all $f\in
h^{\rz,\,q,\,s}_{\oz,\,\fin}(\rn)$,
\begin{eqnarray*}
\|f\|_{h^{\rz,\,q,\,s}_{\oz,\,\fin}(\rn)}&\sim&
\|f\|_{h^{\bfai}_{\oz}(\rn)}\\
&\sim&\inf\left\{\blz\{\lz_i a_i\}_i:\,f=\sum_{i=0}^{k}\lz_i
a_i,\,k\in\zz_+,\,\{a_i\}_{i=1}^{k}\,\text{are}\,
(\rz,\,q,\,s)_{\oz}\text{-atoms}\r. \\
&&\hs\hs\hs\text{satisfying}\ \supp(a_j)\subset Q_j,\ l(Q_j)\in(0,2]\\
&&\hs\hs\hs \text{and}\ a_0 \,\text{is a}\,
(\rz,\,q)_{\oz}\text{-single-atom}\Bigg\}.
\end{eqnarray*}

(ii) Obviously, when $\oz(\rn)=\fz$,
$$h^{\rz,\,\fz,\,s}_{\oz,\,\fin,\,c}(\rn)\cap
C(\rn)=h^{\rz,\,\fz,\,s}_{\oz,\,\fin}(\rn)\cap C(\rn).$$
\end{remark}

As an application of Theorem \ref{t6.2}, we establish the
boundedness on $h^{\bfai}_{\oz}(\rn)$ of quasi-Banach-valued
sublinear operators.

Recall that a {\it  quasi-Banach space $\cb$} is a vector space
endowed with a quasi-norm $\|\cdot\|_{\cb}$ which is nonnegative,
non-degenerate (i.\,e., $\|f\|_{\cb}=0$ if and only if $f= 0$),
homogeneous, and obeys the quasi-triangle inequality, i.\,e., there
exists a positive constant $K$ no less than 1 such that for all
$f,\,g\in\cb$,
$$\|f + g\|_{\cb}\le K(\|f\|_{\cb} + \|g\|_{\cb}).$$

Let $\bz\in(0, 1]$. As in \cite{yz08,yz09}, a quasi-Banach space
$\cb_{\bz}$ with the quasi-norm $\|\cdot\|_{\cb_{\bz}}$ is called a
{\it $\bz$-quasi-Banach space} if
$$\|f+g\|^{\bz}_{\cb_{\bz}}\le\|f\|^{\bz}_{\cb_{\bz}}+
\|g\|^{\bz}_{\cb_{\bz}}$$
for all $f,\, g\in\cb_{\bz}$.

Notice that any Banach space is a 1-quasi-Banach space, and the
quasi-Banach spaces $l^{\bz}$ and $L^{\bz}_{\oz}(\rn)$ are typical
$\bz$-quasi-Banach spaces. Let $\bfai$ satisfy Assumption (A). By
the subadditivity of $\bfai$  and \eqref{2.6}, we know that
$h^{\bfai}_{\oz}(\rn)$ is a $p_{\bfai}$-quasi-Banach space.

For any given $\bz$-quasi-Banach space $\cb_{\bz}$ with $\bz\in(0,
1]$ and a linear space $\cy$, an operator $T$ from $\cy$ to
$\cb_{\bz}$ is called {\it $\cb_{\bz}$-sublinear} if for any $f,\,g
\in\cb_{\bz}$ and $\lz,\,\nu\in\cc$,
$$\|T(\lz f+\nu g)\|_{\cb_{\bz}}\le\left(|\lz|^{\bz}\|T(f)\|^{\bz}
_{\cb_{\bz}}+|\nu|^{\bz}\|T(g)\|^{\bz} _{\cb_{\bz}}\r)^{1/\bz},$$
and
$$\|T(f)-T(g)\|_{\cb_{\bz}}\le \|T(f-g)\|_{\cb_{\bz}}$$
(see \cite{yz08,yz09}).

We remark that if $T$ is linear, then $T$ is $\cb_{\bz}$-sublinear.
Moreover, if $\cb_{\bz}$ is a space of functions, and $T$ is
nonnegative and sublinear in the classical sense, then $T$ is also
$\cb_{\bz}$-sublinear.

\begin{theorem}\label{t6.4}
Let $\bfai$ satisfy Assumption $\mathrm{(A)}$, $\oz\in
A^{\loc}_{\fz}(\rn)$, $q_{\oz}$ be as in \eqref{2.4} and
$(\rz,\,q,\,s)_{\oz}$ be admissible. Let $\cb_{\bz}$ be a
$\bz$-quasi-Banach space with $\bz\in(0,1]$ and $\wz{p}$ be a upper
type of $\bfai$ satisfying $\wz{p}\in(0,\bz]$. Suppose that one of
the following holds:

$\mathrm{(i)}$ $q\in(q_{\oz},\fz)$ and
$T:\,h^{\rz,\,q,\,s}_{\oz,\,\fin}(\rn)\to \cb_{\bz}$ is a
$\cb_{\bz}$-sublinear operator such that
\begin{eqnarray*}
&&S\equiv\sup\left\{\|T(a)\|_{\cb_{\bz}}:\ a \ \text{is
any}\,(\rz,\,q,\,s)_{\oz}\text{-atom with}\,\supp(a)\subset
Q\ \text{and}\r.\\
&& \hspace{4 em}l(Q)\in(0,2]\ \text{or}\ (\rz,\,q)_{\oz}\
\text{-single -atom}\big\}<\fz.
\end{eqnarray*}

$\mathrm{(ii)}$ $T$ is a $\cb_{\bz}$-sublinear operator defined on
continuous $(\rz,\,\fz,\,s)_{\oz}$-atoms such that
\begin{eqnarray*}
&&S\equiv\sup\{\|T(a)\|_{\cb_{\bz}}:\,a \,\text{is any
continuous}\,\,(\rz,\,\fz,\,s)_{\oz}\text{-atom}\}<\fz.
\end{eqnarray*}

Then there exists a unique bounded $\cb_{\bz}$-sublinear operator
$\wz{T}$ from $h^{\bfai}_{\oz}(\rn)$ to $\cb_{\bz}$ which extends
$T$.
\end{theorem}

\begin{proof}
We first show Theorem \ref{t6.4} under its assumption (i). For any
$f\in h^{\rz,\,q,\,s}_{\oz,\,\fin}(\rn)$, by Theorem \ref{t6.2}(i)
and Remark \ref{r6.3}(i), there exist a sequence
$\{\lz_j\}^l_{j=0}\subset\cc$ with some $l\in\nn$, a
$(\rz,\,q)_{\oz}$-single-atom $a_0$ and $(\rz,\,q,\, s)_{\oz}$-atoms
$\{a_j\}^l_{j=1}$ satisfying $\supp(a_j)\subset Q_j$ and
$l(Q_j)\in(0,2]$ for $j\in\{1,\,2,\,\cdots,\,l\}$ such that $f
=\sum^l_{j=0}\lz_j a_j$ pointwise and
\begin{eqnarray}\label{6.5}
\blz\left(\{\lz_j a_j\}_{j=0}^l\r)\ls\|f\|_{h^{\bfai}_{\oz}(\rn)}.
\end{eqnarray}
Then by the assumptions, we see that
\begin{eqnarray}\label{6.6}
\|T(f)\|_{\cb_{\bz}}&\le&\left\{\sum_{i=0}^l
|\lz_i|^{\bz}\|T(a)\|_{\cb_{\bz}}^{\bz}\r\}^{1/\bz}\le
\left\{\sum_{i=0}^l |\lz_i|^{\wz{p}}\|T(a)\|_{\cb_{\bz}}^{\wz{p}}\r\}
^{1/\wz{p}}\nonumber\\  &\ls&\left\{\sum_{i=0}^l
|\lz_i|^{\wz{p}}\r\}^{1/\wz{p}}.
\end{eqnarray}
Since $\bfai$ is of upper type $\wz{p}$, then for any $t\in(0,1]$
and $s\in(0,\fz)$, we have $\bfai(st)\gs t^{\wz{p}}\bfai(s)$. Let
$\wz{\lz}_0\equiv\{\sum_{i=0}^l |\lz_i|^{\wz{p}}\}^{1/\wz{p}}$.
Then,
\begin{eqnarray*}
\sum_{i=0}^l\oz(Q_i)\bfai\left(\frac{|\lz_i|}{\wz{\lz}_0
\oz(Q_i)\rz(\oz(Q_i))}\r)
\gs\sum_{i=0}^l\oz(Q_i)\left(\frac{|\lz_i|}{\wz{\lz}_0}\r)^{\wz{p}}
\frac{1}{\oz(Q_i)}\sim1.
\end{eqnarray*}
Thus, from this we deduce that $\wz{\lz}_0\ls\blz(\{\lz_i
a_i\}_{i=0}^l)$, which together with \eqref{6.5} and \eqref{6.6}
implies that
$$\|T(f)\|_{\cb_{\bz}}\ls\wz{\lz}_0\ls\blz\left(\{\lz_i
a_i\}_{i=0}^l\r)\ls\|f\|_{h^{\bfai}_{\oz}(\rn)}.$$ Since
$h^{\rz,\,q,\,s}_{\oz,\,\fin}(\rn)$ is dense in
$h^{\bfai}_{\oz}(\rn)$, a density argument gives the desired
conclusion in this case.

Now we prove Theorem \ref{t6.4} under its assumption (ii) by
considering the following two cases for $\oz$.

{\it Case 1)} $\oz(\rn)=\fz$. In this case, similarly to the proof
of (i), using Theorem \ref{t6.2}(ii) and Remark \ref{r6.3}(ii), we
see that for all $f\in h^{\rz,\,\fz,\,s}_{\oz,\,\fin}(\rn)\cap
C(\rn)$,
$$\|T(f)\|_{\cb_{\bz}}\ls\|f\|_{h^{\bfai}_{\oz}(\rn)}.$$
To extend $T$ to the whole $h^{\bfai}_{\oz}(\rn)$, we only need  prove
that $h^{\rz,\,\fz,\,s}_{\oz,\,\fin}(\rn)\cap C(\rn)$ is dense in
$h^{\bfai}_{\oz}(\rn)$. Since $h^{\rz,\,\fz,\,s}_{\oz,\,\fin}(\rn)$
is dense in $h^{\bfai}_{\oz}(\rn)$, it suffices to prove that
$h^{\rz,\,\fz,\,s}_{\oz,\,\fin}(\rn)\cap C(\rn)$ is dense in
$h^{\rz,\,\fz,\,s}_{\oz,\,\fin}(\rn)$ with respect to the quasi-norm
$\|\cdot\|_{h^{\bfai}_{\oz}(\rn)}$.

To see this, let $f\in h^{\rz,\,\fz,\,s}_{\oz,\,\fin}(\rn)$. In this
case, for any $(\rz,\,\fz)_{\oz}$-single-atom $b$, $b(x)=0$ for
almost every $x\in\rn$. Thus, $f$ is a finite linear combination of
$(\rz,\,\fz,\,s)_{\oz}$-atoms. Then there exists a cube $Q_0\equiv
Q(x_0, r_0)$ such that $\supp(f)\subset Q_0$. Take $\phi\in\cd(\rn)$
such that $\supp(\phi)\subset Q(0,1)$ and $\int_{\rn}\phi(x)\,dx=1$.
Then it is easy to see that for any $k\in\nn$, $\supp(\phi_k \ast f)
\subset Q(x_0, r_0 +1)$ and $\phi_k \ast f\in\cd(\rn)$. Assume that
$f=\sum_{i=1}^{N}\lz_i a_i$ with some $N\in\nn$,
 $\{\lz_i\}_{i=1}^N \subset\cc$ and $\{a_i\}_{i=1}^N $ being
$(\rz,\,\fz,\,s)_{\oz}$-atoms. Then for any $k\in\nn$,
$$\phi_k \ast f=\sum_{i=1}^{N}\lz_i \phi_k \ast a_i.$$
For any $k\in\nn$ and
$i\in\{1,\,2\,\cdots,\,N\}$, we now prove that $\phi_k \ast a_i$
 is a  multiple of some continuous $(\rz,\,\fz,\,s)_{\oz}$-atom, which
implies that for any $k\in\nn$,
\begin{eqnarray}\label{6.7}
\phi_k \ast f\in h^{\rz,\,\fz,\,s}_{\oz,\,\fin}(\rn)\cap C(\rn).
\end{eqnarray}
For $i\in\{1,\,2,\,\cdots,\,N\}$, assume that $\supp(a_i)\subset Q_i
\equiv Q(x_i,r_i)$. Then
$$\supp(\phi_k \ast a_i)\subset
\wz{Q}_{i,\,k}\equiv Q(x_i,r_i+1/2^k).$$
Moreover,
$$\|\phi_k \ast
a_i\|_{L^{\fz}_{\oz}(\rn)}\le\|a_i\|_{L^{\fz}_{\oz}(\rn)}\le
\frac{1}{\oz(Q_i)\rz(\oz(Q_i))}.$$ Furthermore, for any
$\az\in\zz^n_+$, $\int_{\rn}a_i (x)x^{\az}\,dx=0$ implies that
$$\int_{\rn}\phi_k \ast a_i (x) x^{\az}\,dx=0.$$
Thus, $\frac{\oz(Q_i)\rz
(\oz(Q_i))}{\oz(\wz{Q}_{i,\,k})\rz(\oz(\wz{Q}_{i,\,k}))}\phi_k \ast
a_i$ is a $(\rz,\,\fz,\,s)_{\oz}$-atom.

Likewise, $\supp(f-\phi_k \ast f)\subset Q(x_0, r_0 +1)$ and
$f-\phi_k \ast f$ has the same vanishing moments as $f$. Take
$q\in(q_{\oz},\fz)$. By Lemma \ref{l2.6}(iii),
\begin{eqnarray}\label{6.8}
\|f-\phi_k \ast f\|_{L^q_{\oz}(\rn)}\to0\,\,\text{as}\,\,k\to\fz.
\end{eqnarray}
Without loss of generality, we may assume that when $k$ is large
enough,
$$\|f-\phi_k \ast f\|_{L^q_{\oz}(\rn)}>0.$$
Let
$$c_k\equiv\|f-\phi_k \ast f\|_{L^q_{\oz}(\rn)}[\oz(Q(x_0, r_0
+1))]^{1/q-1}\rz(\oz(Q(x_0, r_0 +1)))$$ and $a_k\equiv\frac{f-\phi_k
\ast f}{c_k}$. Then,  $a_k$ is a $(\rz,\,q,\,s)_{\oz}$-atom,
$f-\phi_k \ast f=c_k a_k$,  and $c_k\to0$ as $k\to\fz$. Thus, from
\eqref{2.8} with $t\equiv\oz(Q(x_0,r_0+1))$, and Theorem \ref{t5.6},
we infer that
\begin{eqnarray}\label{6.9}
\|f-\phi_k \ast f\|_{h^{\bfai}_{\oz}(\rn)}\ls\blz(\{c_k
a_k\})\ls|c_k|\to0
\end{eqnarray}
as $k\to\fz$, which together with \eqref{6.7} shows the desired
conclusion in this case.

{\it Case 2)} $\oz(\rn)<\fz$. In this case, similarly to the proof
of Case 1), by Theorem \ref{t6.2}(ii), to finish the proof of (ii)
in this case, it suffices to prove that
$h^{\rz,\,\fz,\,s}_{\oz,\,\fin,\,c}(\rn)\cap C(\rn)$ is dense in
$h^{\rz,\,\fz,\,s}_{\oz,\,\fin}(\rn)$ in the quasi-norm
$\|\cdot\|_{h^{\bfai}_{\oz}(\rn)}$.

For any $f\in h^{\rz,\,\fz,\,s}_{\oz,\,\fin}(\rn)$, assume that
$$f\equiv\sum_{i=1}^{N_1} \lz_i a_i+\lz_0 a_0,$$
where $N_1\in\nn$,
$\{\lz_i\}_{i=1}^{N_1}\subset\cc$ and $a_0$ is a
$(\rz,\,\fz)_{\oz}$-single-atom and $\{a_i\}_{i=1}^{N_1}$
 are $(\rz,\,\fz,\,s)_{\oz}$-atoms. Let
$\{\psi_k\}_{k\in\nn}\subset\cd(\rn)$ satisfy $0\le\psi_k\le1$,
$\psi_k\equiv1$ on the cube $Q(0,2^k)$ and
$$\supp(\psi_k)\subset Q(0,2^{k+1}).$$
We assume that $\supp(\sum_{i=1}^{N_1} \lz_i
a_i)\subset Q(0,R_0)$ for some $R_0\in(0,\fz)$ and $k_0$ is the
smallest integer such that $2^{k_0}\ge R_0$. For any integer $k\ge
k_0$, let $f_k\equiv f\psi_k$. Then $f_k\in
h^{\rz,\,\fz,\,s}_{\oz,\,\fin,\,c}(\rn)$. Indeed, by the choice of
$\pz_k$, we know that
$$f_k=\sum_{i=1}^{N_1}\lz_i a_i+\lz_0 a_0\pz_k$$
and $\supp (f_k)\subset Q(0,2^{k+1})$. Furthermore, from
$\supp(a_0 \pz_k)\subset Q(0,2^{k+1})$ and
$$\|a_0\pz_k\|_{L^{\fz}_{\oz}(\rn)}\le\|a_0\|_{L^{\fz}_{\oz}(\rn)}
\le\frac{1}{\oz(\rn)\rz(\oz(\rn))}\le\frac{1}{\oz(Q(0,2^{k+1}))
\rz(\oz(Q(0,2^{k+1})))},$$
we deduce that $a_0 \pz_k$ is a
$(\rz,\,\fz,\,s)_{\oz}$-atom. Thus, $f_k\in
h^{\rz,\,\fz,\,s}_{\oz,\,\fin,\,c}(\rn)$. For any fixed integer
$k\ge k_0$ and any $i\in\nn$, let $\wz{f}_{k,\,i}\equiv f_k
\ast\phi_i$, where $\phi$ is as in Case 1). Similarly to the proof
of \eqref{6.7}, we have $\wz{f}_{k,\,i}\in
h^{\rz,\,\fz,\,s}_{\oz,\,\fin,\,c}(\rn)\cap C(\rn)$. For any
$q\in(q_{\oz},\fz)$, from the choice of $f_k$ and $\oz(\rn)<\fz$, we
conclude that
\begin{eqnarray}\label{6.10}
\|f-f_k\|_{L^q_{\oz}(\rn)}&\le&
\left\{\int_{Q(0,2^{k})^{\complement}}|f(x)|^q \oz(x)\,dx\r\}^{1/q}
\nonumber\\
 &\le&\|\lz_0
a_0\|_{L^{\fz}_{\oz}(\rn)}\left\{\int_{Q(0,2^{k})^{\complement}}
\oz(x)\,dx\r\}^{1/q}\to0,
\end{eqnarray}
as $k\to\fz$. Furthermore, for any fixed $k\in\zz$ with $k\ge k_0$,
similarly to the proof of \eqref{6.8}, we see that
$\|f_k-\wz{f}_{k,\,i}\|_{L^q_{\oz}(\rn)}\to0\,
\,\text{as}\,\,i\to\fz$, which together with \eqref{6.10}
implies that
$$\left\|f-\wz{f}_{k,\,i}\r\|_{L^q_{\oz}(\rn)}\to0$$
as $k,\,i\to\fz$. Without loss of generality, we may assume that
when $k$ and $i$ are large enough,
$\|f-\wz{f}_{k,\,i}\|_{L^q_{\oz}(\rn)}>0$. Let
$$c_{k,\,i}\equiv\|f-\wz{f}_{k,\,i}\|_{L^q_{\oz}(\rn)}[\oz(\rn)]^{1/q-1}
\rz(\oz(\rn))$$
and
$a_{k,\,i}\equiv\frac{f-\wz{f}_{k,\,i}}{c_{k,\,i}}$. Then,
$f-\wz{f}_{k,\,i}=c_{k,\,i} a_{k,\,i}$, $a_{k,\,i}$ is a
$(\rz,\,q)_{\oz}$-single-atom and $c_{k,\,i}\to0$ as $k,\,i\to\fz$.
Then, similarly to the proof of \eqref{6.9}, we know that
$\|f-\wz{f}_{k,\,i} \|_{h^{\bfai}_{\oz}(\rn)}\to0$ as $k,\,i\to\fz$,
which completes the proof of Case 2) and hence of Theorem
\ref{t6.4}.
\end{proof}

\begin{remark}\label{r6.5}
Let $p\in(0,1]$. We point out that Theorems \ref{t6.2}(i) and
\ref{t6.4}(i) when $\bfai(t)\equiv t^p$ for all $t\in(0,\fz)$ were
obtained by Tang \cite[Theorems\,6.1 and 6.2]{Ta1}. Theorems
\ref{t6.2}(ii) and \ref{t6.4}(ii) are new even when $\oz\equiv1$ and
$\bfai(t)\equiv t^p$ for all $t\in(0,\fz)$.
\end{remark}

\chapter{Dual spaces \label{s7}}

In this section, we introduce the $\bbmo$-type space
$\bmo^q_{\rz,\,\oz}(\rn)$ and establish the duality between
$h^{\rz,\,q,\,s}_{\oz}(\rn)$ and $\bmo^{q'}_{\rz,\,\oz}(\rn)$, where
and in what follows, $\frac{1}{q}+\frac{1}{q'}=1$. From this and
Theorem \ref{t5.6}, we deduce the duality between
$h^{\bfai}_{\oz}(\rn)$ and $\bmo_{\rz,\,\oz}(\rn)$,  and that for
$q\in[1,\frac{q_{\oz}}{q_{\oz}-1})$, $\bmo^{q}_{\rz,\,\oz}(\rn)
=\bmo_{\rz,\,\oz}(\rn)$ with equivalent norms, where the
\emph{symbol} $\bmo_{\rz,\,\oz}(\rn)$ denotes
$\bmo^1_{\rz,\,\oz}(\rn)$. We begin with some notions.

For any locally integrable function $f$ on $\rn$, we denote the {\it
minimizing polynomial} of $f$ on the cube $Q$ with degree at most
$s$ by $P^s_Q f$, namely, for all multi-indices $\tz\in\zz_+^n$ with
$0\le|\tz|\le s$,
\begin{eqnarray}\label{7.1}
\int_{Q}\left[f(x)-P^s_Q f(x)\r]x^{\tz}\,dx=0.
\end{eqnarray}
It is well known that if $f$ is locally integrable, then $P^s_Q f$
uniquely exists; see, for example, \cite{tw80}. Now, we introduce
the $\bbmo$-type space $\bmo^q_{\rz,\,\oz}(\rn)$ as follows.
\begin{definition}\label{d7.1}
Let $\bfai$ satisfy Assumption $\mathrm{(A)}$, $\oz\in
A^{\loc}_{\fz}(\rn)$, and $q_{\oz}$, $p_{\bfai}$ and $\rz$ be
respectively as in \eqref{2.4}, \eqref{2.6} and \eqref{2.7}. Let
$q\in[1,\frac{q_{\oz}}{q_{\oz}-1})$ and $s\in\zz_+$ with $s\ge\lfz
n(\frac{q_{\oz}}{p_{\bfai}}-1)\rf$. When $\oz(\rn)=\fz$, a locally
integrable function $f$ on $\rn$ is said to belong to the {\it space
$\bmo_{\rz,\,\oz}^q (\rn)$}, if
\begin{eqnarray*}
\|f\|_{\bmo_{\rz,\,\oz}^q
(\rn)}&\equiv&\sup_{Q\subset\rn,\,|Q|<1}\frac{1}{\rz(\oz(Q))}
\left\{\frac{1}{\oz(Q)}\int_{Q}\left|f(x)-P^s_Q f(x)\r|^q
[\oz(x)]^{1-q}\,dx\r\}^{1/q}\\
&&+\sup_{Q\subset\rn,\,|Q|\ge1}\frac{1}{\rz(\oz(Q))}
\left\{\frac{1}{\oz(Q)}\int_{Q}|f(x)|^q
[\oz(x)]^{1-q}\,dx\r\}^{1/q}<\fz,
\end{eqnarray*}
where the supremum is taken over all the cubes $Q\subset\rn$ and
$P^s_Q f$ is as in \eqref{7.1}. When $\oz(\rn)<\fz$, a function $f$
on $\rn$ is said to belong to the {\it space $\bmo_{\rz,\,\oz}^q
(\rn)$} if
\begin{eqnarray*}
\|f\|_{\bmo_{\rz,\,\oz}^q
(\rn)}&\equiv&\sup_{Q\subset\rn,\,|Q|<1}\frac{1}{\rz(\oz(Q))}
\left\{\frac{1}{\oz(Q)}\int_{Q}\left|f(x)-P^s_Q f(x)\r|^q
[\oz(x)]^{1-q}\,dx\r\}^{1/q}\\
&&+\sup_{Q\subset\rn,\,|Q|\ge1}\frac{1}{\rz(\oz(Q))}
\left\{\frac{1}{\oz(Q)}\int_{Q}|f(x)|^q
[\oz(x)]^{1-q}\,dx\r\}^{1/q}\\
&&+\frac{1}{\rz(\oz(\rn))} \left\{\frac{1}{\oz(\rn)}\int_{\rn}|f(x)|^q
[\oz(x)]^{1-q}\,dx\r\}^{1/q}<\fz,
\end{eqnarray*}
where the supremum is taken over all the cubes $Q\subset\rn$ and
$P^s_Q f$ is as in \eqref{7.1}.
\end{definition}

When $\oz\equiv1$, $\bfai\equiv t$ for all $t\in(0,\fz)$ and $q=1$,
the space $\bmo_{\rz,\,\oz}^q (\rn)$ is just the space $\bmo(\rn)$
introduced in \cite{Go79}.

Now, we establish the duality between $h^{\rz,\,q,\,s}_{\oz}(\rn)$
and $\bmo^{q'}_{\rz,\,\oz}(\rn)$. We begin with the notion of the
weighted atomic Orlicz-Hardy space $H^{\rz,\,q,\,s}_{\oz} (\rn)$.

\begin{definition}\label{d7.2}
Let $\bfai$ satisfy Assumption $\mathrm{(A)}$, $\oz\in
A^{\loc}_{\fz}(\rn)$, $\rz$ be as in \eqref{2.7} and
$(\rz,\,q,\,s)_{\oz}$ be admissible. A function $a$ on $\rn$ is
called an {\it $H^{\rz,\,q,\,s}_{\oz} (\rn)$-atom} if there exists a
cube $Q\subset\rn$ such that

(i) $\supp(a)\subset Q$;

(ii) $\|a\|_{L^q_{\oz}(\rn)}\le[\oz(Q)]^{1/q -1}[\rz(\oz(Q))]^{-1}$;

(iii) $\int_{\rn}a(x)x^{\az}\,dx=0$ for all multi-indices
$\az\in\zz_+^n$ with $|\az|\le s$.\\
The {\it weighted atomic Orlicz-Hardy space} $H^{\rz,\,q,\,s}_{\oz}
(\rn)$ is defined to be the space of all $f\in\cd'(\rn)$ satisfying
that $f=\sum_{i=1}^{\fz}\lz_i a_i$ in $\cd'(\rn)$, where
$\{a_i\}_{i\in\nn}$ are $H^{\rz,\,q,\,s}_{\oz} (\rn)$-atoms with
$\supp(a_i)\subset Q_i$, and $\{\lz_i\}_{i\in\nn}\subset\cc$,
satisfying
$$\sum_{i=1}^{\fz}\oz(Q_i)\bfai\left(\frac{|\lz_i|}{\oz(Q_i)
\rz(\oz(Q_i))}\r)<\fz.$$ Moreover, the \emph{quasi-norm} of $f\in
H^{\rz,\,q,\,s}_{\oz} (\rn)$ is defined by
$$\|f\|_{H^{\rz,\,q,\,s}_{\oz} (\rn)}\equiv\inf\left\{
\blz\left(\{\lz_i a_i\}_{i=1}^{\fz}\r)\r\},$$ where the infimum is
taken over all the decompositions of $f$ as above and
\begin{eqnarray*}
&&\blz(\{\lz_i a_i\}_{i=1}^{\fz})\equiv \inf\left\{\lz>0:\,
\sum_{i=1}^{\fz}\oz(Q_i)\bfai\left(\frac{|\lz_i|}{\lz\oz(Q_i)
\rz(\oz(Q_i))}\r)\le1\r\}.
\end{eqnarray*}

Furthermore, the {\it space} $H^{\rz,\,q,\,s}_{\oz,\,\fin}(\rn)$ is
defined to be the set of all finite linear combinations of
$H^{\rz,\,q,\,s}_{\oz} (\rn)$-atoms.
\end{definition}

Obviously, the space $H^{\rz,\,q,\,s}_{\oz,\,\fin}(\rn)$ is dense in
the space $H^{\rz,\,q,\,s}_{\oz}(\rn)$ with respect to the
quasi-norm $\|\cdot\|_{H^{\rz,\,q,\,s}_{\oz}(\rn)}$.

\begin{definition}\label{d7.3}
Let $\bfai$ satisfy Assumption $\mathrm{(A)}$, $\oz\in
A^{\loc}_{\fz}(\rn)$, $q_{\oz}$, $p_{\bfai}$ and $\rz$ be
respectively as in \eqref{2.4}, \eqref{2.6} and \eqref{2.7}. Let
$q\in[1,\frac{q_{\oz}}{q_{\oz}-1})$ and $s\in\zz_+$ with $s\ge\lfz
n(\frac{q_{\oz}}{p_{\bfai}}-1)\rf$. A locally integrable function
$f$ on $\rn$ is said to belong to the {\it space
$\bbmo_{\rz,\,\oz}^q (\rn)$} if
\begin{eqnarray*}
\|f\|_{\bbmo_{\rz,\,\oz}^q
(\rn)}&\equiv&\sup_{Q\subset\rn}\frac{1}{\rz(\oz(Q))}
\left\{\frac{1}{\oz(Q)}\int_{Q}\left|f(x)-P^s_Q f(x)\r|^q
[\oz(x)]^{1-q}\,dx\r\}^{1/q}<\fz,
\end{eqnarray*}
where the supremum is taken over all cubes $Q\subset\rn$ and $P^s_Q
f$ is as in \eqref{7.1}.
\end{definition}

Now, we establish the duality between $H^{\rz,\,q,\,s}_{\oz} (\rn)$
and $\bbmo_{\rz,\,\oz}^{q'} (\rn)$.
\begin{lemma}\label{l7.4}
Let $\bfai$ satisfy Assumption $\mathrm{(A)}$, $\oz\in
A^{\loc}_{\fz}(\rn)$, $q_{\oz}$ and $\rz$ be respectively as in
\eqref{2.4} and \eqref{2.7}, and $(\rz,\,q,\,s)_{\oz}$ be
admissible. Then $\left[H^{\rz,\,q,\,s}_{\oz} (\rn)\r]^{\ast}$, the
dual space of $H^{\rz,\,q,\,s}_{\oz} (\rn)$, coincides with
$\bbmo_{\rz,\,\oz}^{q'} (\rn)$ in the following sense.

$\mathrm{(i)}$ Let $g\in \bbmo_{\rz,\,\oz}^{q'} (\rn)$. Then the
linear functional $L$, which is initially defined on
$H^{\rz,\,q,\,s}_{\oz,\,\fin}(\rn)$ by
\begin{eqnarray}\label{7.2}
L(f)=\langle g,f\rangle,
\end{eqnarray}
has a unique extension to $H^{\rz,\,q,\,s}_{\oz}(\rn)$ with
$$\|L\|_{\left[H^{\rz,\,q,\,s}_{\oz} (\rn)\r]^{\ast}}\le
C\|g\|_{\bbmo_{\rz,\,\oz}^{q'} (\rn)},$$
where $C$ is a positive constant independent of $g$.

$\mathrm{(ii)}$ Conversely, for any $L\in\left[H^{\rz,\,q,\,s}_{\oz}
(\rn)\r]^{\ast}$, there exists $g\in \bbmo_{\rz,\,\oz}^{q'} (\rn)$
such that \eqref{7.2} holds for all $f\in
H^{\rz,\,q,\,s}_{\oz,\,\fin}(\rn)$ and
$$\|g\|_{\bbmo_{\rz,\,\oz}^{q'} (\rn)}\le
C\|L\|_{\left[H^{\rz,\,q,\,s}_{\oz} (\rn)\r]^{\ast}},$$
where $C$ is a positive constant independent of $L$.
\end{lemma}

\begin{proof}
We borrow some ideas from \cite{tw80} and \cite[Theorm 4.1]{L95}.
Let $(\rz,\,q,\,s)_{\oz}$ be an admissible triplet. First, we prove
(i). Let $a$ be an $H^{\rz,\,q,\,s}_{\oz} (\rn)$-atom with $\supp
(a)\subset Q$ and $g\in\bbmo_{\rz,\,\oz}^{q'} (\rn)$. Then by the
vanishing condition of $a$ and H\"older's inequality, we have
\begin{eqnarray}\label{7.3}
\left|\int_{\rn}a(x)g(x)\,dx\r|&=&\left|\int_{\rn}a(x)\left[g(x)-P^s_Q
g(x)\r]\,dx\r|\nonumber\\
&\le&\|a\|_{L^q_{\oz}(\rn)}\left\{\int_{Q}\left|g(x)-P^s_Q
g(x)\r|^{q'}[\oz(x)]^{1-q'}\,dx\r\}^{1/q'}\nonumber\\
 &\le&\|g\|_{\bbmo_{\rz,\,\oz}^{q'}(\rn)},
\end{eqnarray}
where $P^s_Q g$ is as in \eqref{7.1}. Let
$$f=\sum_{i=1}^{k_0}\lz_i a_i \in H^{\rz,\,q,\,s}_{\oz,\,\fin} (\rn),$$
where $k_0\in\nn$,
$\{\lz_i\}_{i=1}^{k_0}\subset\cc$ and for
$i\in\{1,\,2,\,\cdots,\,k_0\}$, $a_i$ is an $H^{\rz,\,q,\,s}_{\oz}
(\rn)$-atom with $\supp(a_i)\subset Q_i$. Since $\bfai$ is concave
and has upper type 1, by Remark \ref{r3.6}(iii), we know that
$\sum_{i}|\lz_i|\ls\blz(\{\lz_i a_i\}_{i=1}^{k_0})$, which together
with \eqref{7.3} implies that
\begin{eqnarray*}
\left|\int_{\rn}f(x)g(x)\,dx\r|&\le&\sum_{i=1}^{k_0}|\lz_i|\left|\int_{Q_i}a_i
(x)g(x)\,dx\r|\\
&\le&\left\{\sum_{i=1}^{k_0}|\lz_i|\r\}\|g\|_{\bbmo_{\rz,\oz}^{q'}(\rn)}
\ls\blz\left(\{\lz_i
a_i\}_{i=1}^{k_0}\r)\|g\|_{\bbmo_{\rz,\oz}^{q'}(\rn)}.
\end{eqnarray*}
Thus, by the above estimate and the fact that
$H^{\rz,\,q,\,s}_{\oz,\,\fin} (\rn)$ is dense in
$H^{\rz,\,q,\,s}_{\oz} (\rn)$ with respect to the quasi-norm
$\|\cdot\|_{H^{\rz,\,q,\,s}_{\oz} (\rn)}$, we know that
$\mathrm{(i)}$ holds.

To prove $\mathrm{(ii)}$, assume that $L\in\left[H^{\rz,\,q,\,s}_{\oz}
(\rn)\r]^{\ast}$. Let $Q\subset\rn$ be a closed cube and
$$L^q_{\oz,\,s}(Q)\equiv\left\{f\in
L^q_{\oz}(Q):\ \int_{Q}f(x)x^{\az}\,dx=0,\,\az\in\zz^n_+,\,|\az|\le
s\r\},$$ where $f\in L^q_{\oz}(Q)$ means that $f\in L^q_{\oz}(\rn)$
and $\supp(f)\in Q$. We first prove that
\begin{eqnarray}\label{7.4}
\left[H^{\rz,\,q,\,s}_{\oz}
(\rn)\r]^{\ast}\subset\left[L^q_{\oz,\,s}(Q)\r]^{\ast}.
\end{eqnarray}
Obviously, for any given $f\in L^q_{\oz,\,s}(Q)$,
$$a\equiv[\oz(Q)]^{1/q
-1}[\rz(\oz(Q))]^{-1}\|f\|_{L^q_{\oz}(Q)}^{-1}f$$ is an
$H^{\rz,\,q,\,s}_{\oz} (\rn)$-atom. Thus, $f\in
H^{\rz,\,q,\,s}_{\oz} (\rn)$ and
$$\|f\|_{H^{\rz,\,q,\,s}_{\oz}
(\rn)}\le[\oz(Q)]^{1/q'}\rz(\oz(Q))\|f\|_{L^q_{\oz}(Q)},$$ which
implies that for all $f\in L^q_{\oz,\,s}(Q)$,
$$|Lf|\le\|L\|\|f\|_{H^{\rz,\,q,\,s}_{\oz}
(\rn)}\le[\oz(Q)]^{1/q'}\rz(\oz(Q))\|L\|\|f\|_{L^q_{\oz}(Q)}.$$ That
is, $L\in[L^q_{\oz,\,s}(Q)]^{\ast}$. Thus, \eqref{7.4} holds.

From \eqref{7.4}, the Hahn-Banach theorem and the Riesz represent
theorem, it follows that there exists a $\wz{g}\in L^{q'}_{\oz}(Q)$
such that for all $f\in L^q_{\oz,\,s}(Q)$,
\begin{eqnarray}\label{7.5}
Lf=\int_{Q}f(x)\wz{g}(x)\oz(x)\,dx,
\end{eqnarray}
 where when $q=\fz$, we used the fact that $L^{\fz}_{\oz,\,s}(Q)\subset
L^{\gz}_{\oz,\,s}(Q)$ for any $\gz\in[1,\fz)$ and
$L^{\gz'}_{\oz,\,s}(Q)\subset L^1_{\oz,\,s}(Q)$. Taking a sequence
$\{Q_j\}_{j\in\nn}$ of cubes  such that for any $j\in\nn$, $Q_j
\subset Q_{j+1}$ and $\lim_{j\to\fz}Q_j =\rn$. From the above
result, it follows that for each $Q_j$, there exists a $\wz{g}_j \in
L^{q'}_{\oz}(Q_j)$ such that for all $f\in L^q_{\oz,\,s}(Q_j)$,
\begin{eqnarray}\label{7.6}
Lf=\int_{Q_j}f(x)\wz{g}_j (x)\oz(x)\,dx.
\end{eqnarray}

Now, we construct a function $g$ such that
$$Lf=\int_{Q_j}f(x)g(x)\,dx$$
for all $f\in L^q_{\oz,\,s}(Q_j)$ and
all $j\in\nn$. First, assume that $f\in L^q_{\oz,\,s}(Q_1)$. By
\eqref{7.6}, we know that there exists a $\wz{g}_1 \in
L^{q'}_{\oz}(Q_1)$ such that
$$Lf=\int_{Q_1}f(x)\wz{g}_1(x)\oz(x)\,dx.$$
Notice that $f\in L^q_{\oz,\,s}(Q_1)\subset
L^q_{\oz,\,s}(Q_2)$. By \eqref{7.6} again, there exists a $\wz{g}_2
\in L^{q'}_{\oz}(Q_2)$ such that
$$Lf=\int_{Q_1}f(x)\wz{g}_1 (x)\oz(x)\,dx=\int_{Q_2}
f(x)\wz{g}_2 (x)\oz(x)\,dx,$$ which implies that for all $f\in
L^q_{\oz,\,s}(Q_1)$,
\begin{eqnarray}\label{7.7}
\int_{Q_1}f(x)\left[\wz{g}_1 (x)-\wz{g}_2 (x)\r]\oz(x)\,dx=0.
\end{eqnarray}
For any given $h\in L^q_{\oz}(Q_1)$, let $f_1\equiv h-P_{Q_1}^s h$.
Then by \eqref{7.1}, we know that $f_1\in L^q_{\oz,\,s}(Q_1)$. For
$f_1$, by \eqref{7.7}, we have
$$\int_{Q_1}[h(x)-P_{Q_1}^s h (x)][\wz{g}_1 (x)-\wz{g}_2
(x)]\oz(x)\,dx=0,$$ which combined with the well-known fact that
$$\int_{Q_1}P^s_{Q_1}h(x)\left[\wz{g}_1 (x)-\wz{g}_2
(x)\r]\oz(x)\,dx=\int_{Q_1}h(x)P^s_{Q_1}((\wz{g}_1-\wz{g}_2)\oz)(x)\,dx$$
implies that
\begin{eqnarray}\label{7.8}
\int_{Q_1}h(x)\left\{[\wz{g}_1 (x)-\wz{g}_2 (x)]\oz(x)-P_{Q_1}^s
((\wz{g}_1 -\wz{g}_2)\oz)(x)\r\}\,dx=0.
\end{eqnarray}
For $j=1,\,2$, let $g_j\equiv \wz{g}_j \oz$. By \eqref{7.8}, we know
that for all $h\in L^q_{\oz}(Q_1)$,
$$\int_{Q_1}h(x)\left\{\frac{[g_1
(x)-g_2 (x)]-P^s_{Q_1}(g_1-g_2)(x)}{\oz(x)}\r\}\oz(x)\,dx=0,$$ which
implies that for almost every $x\in Q_1$,
$$g_1 (x)-g_2
(x)=P^s_{Q_1}(g_1-g_2)(x).$$ Let
$$g(x)\equiv\begin{cases}g_1(x),&\text{when}\,\,x\in Q_1,\\
g_1(x)+P^s_{Q_1}(g_1-g_2)(x),&\text{when}\,\, x\in (Q_2 \setminus
Q_1).\end{cases}$$ It is easy to see that for any $f \in
L^q_{\oz,\,s}(Q_j)$ with $j\in\{1,\,2\}$,
\begin{eqnarray}\label{7.9}
Lf=\int_{Q_j}f(x)g(x)\,dx.
\end{eqnarray}
In this way, we obtain a function $g$ on $\rn$ such that \eqref{7.9}
holds for any $j\in\nn$.

Finally, we show that $g\in\bbmo^{q'}_{\rz,\,\oz}(\rn)$ and for all
$f\in H^{\rz,\,q,\,s}_{\oz,\,\fin}(\rn)$,
\begin{eqnarray}\label{7.10}
Lf=\int_{\rn}f(x)g(x)\,dx.
\end{eqnarray}
Indeed, for any $H^{\rz,\,q,\,s}_{\oz}(\rn)$-atom $a$, there exists
a $j_0\in\nn$ such that $a\in L^q_{\oz,\,s}(Q_{j_0})$. By this and
the fact that \eqref{7.9} holds for any $j\in\nn$, we see that
\eqref{7.10} holds for any $f\in H^{\rz,\,q,\,s}_{\oz,\,\fin}(\rn)$.
It remains to prove that $g\in \bbmo^{q'}_{\rz,\,\oz}(\rn)$. Take
any cube $Q\subset\rn$ as well as any $f\in L^q_{\oz}(Q)$ satisfying
$\|f\|_{L^q_{\oz}(Q)}\le1$ and $\supp(f)\subset Q$. Let
\begin{eqnarray}\label{7.11}
a\equiv \wz{C}^{-1}[\oz(Q)]^{1/q'}[\rz(\oz(Q))]^{-1}(f-P^s_Q f)
\chi_Q,
\end{eqnarray}
where $\wz{C}$ is a positive constant. Obviously, $\supp (a) \subset
Q$. We choose $\wz{C}$ such that $a$ becomes an
$H^{\rz,\,q,\,s}_{\oz}(\rn)$-atom. From the equality
$$La=\int_{Q}a(x)g(x)\,dx$$
and $L\in
[H^{\rz,\,q,\,s}_{\oz}(\rn)]^{\ast}$, it follows that
\begin{eqnarray}\label{7.12}
|La|=\left|\int_{Q}a(x)[g(x)-P^s_{Q}g(x)]\,dx\r|\le
\|L\|_{\left[H^{\rz,\,q,\,s}_{\oz} (\rn)\r]^{\ast}}.
\end{eqnarray}
 By \eqref{7.11}, \eqref{7.12} and \eqref{7.1}, for all
 $f\in L^q_{\oz}(Q)$ with $\|f\|_{L^q_{\oz}(Q)}\le1$, we see that
$$[\oz(Q)]^{1/q'}[\rz(\oz(Q))]^{-1}
\left|\int_{Q}f(x)[g(x)-P^s_{Q}g(x)]\,dx\r|\ls
\|L\|_{\left[H^{\rz,\,q,\,s}_{\oz} (\rn)\r]^{\ast}},$$ which implies
that
$$[\oz(Q)]^{1/q'}[\rz(\oz(Q))]^{-1}\left\{\int_Q
|g(x)-P^s_{Q}g(x)|^{q'}[\oz(x)]^{1-q'}\,dx\r\}^{1/q'}
\ls\|L\|_{\left[H^{\rz,\,q,\,s}_{\oz} (\rn)\r]^{\ast}}.$$ Thus, $g\in
\bbmo^{q'}_{\rz,\,\oz}(\rn)$ and
$\|g\|_{\bbmo^{q'}_{\rz,\,\oz}(\rn)}\ls
\|L\|_{\left[H^{\rz,\,q,\,s}_{\oz}(\rn)\r]^{\ast}}$. This finishes the
proof of Lemma \ref{l7.4}.
\end{proof}

Now, we give the duality between $h^{\rz,\,q,\,s}_{\oz} (\rn)$ and
$\bmo_{\rz,\,\oz}^{q'} (\rn)$ by invoking Lemma \ref{l7.4}.

\begin{theorem}\label{t7.5}
Let $\bfai$ satisfy Assumption $\mathrm{(A)}$, $\oz\in
A^{\loc}_{\fz}(\rn)$, $q_{\oz}$ and $\rz$ be respectively as in
\eqref{2.4} and \eqref{2.7}, and $(\rz,\,q,\,s)_{\oz}$ be
admissible. Then $\left[h^{\rz,\,q,\,s}_{\oz} (\rn)\r]^{\ast}$, the
dual space of $h^{\rz,\,q,\,s}_{\oz} (\rn)$, coincides with
$\bmo_{\rz,\,\oz}^{q'} (\rn)$ in the following sense.

$\mathrm{(i)}$ Let $g\in \bmo_{\rz,\,\oz}^{q'} (\rn)$. Then the
linear functional $L$, which is initially defined on
$h^{\rz,\,q,\,s}_{\oz,\,\fin}(\rn)$ by
\begin{eqnarray}\label{7.13}
L(f)=\langle g,f\rangle,
\end{eqnarray}
has a unique extension to $h^{\rz,\,q,\,s}_{\oz}(\rn)$ with
$$\|L\|_{\left[h^{\rz,\,q,\,s}_{\oz} (\rn)\r]^{\ast}}\le
C\|g\|_{\bmo_{\rz,\,\oz}^{q'} (\rn)},$$
where $C$ is a positive
constant independent of $g$.

$\mathrm{(ii)}$ Conversely, for any $L\in\left[h^{\rz,\,q,\,s}_{\oz}
(\rn)\r]^{\ast}$, there exists $g\in \bmo_{\rz,\,\oz}^{q'} (\rn)$
such that \eqref{7.13} holds for all $f\in
h^{\rz,\,q,\,s}_{\oz,\,\fin}(\rn)$ and
$$\|g\|_{\bmo_{\rz,\,\oz}^{q'}
(\rn)}\le C\|L\|_{\left[h^{\rz,\,q,\,s}_{\oz} (\rn)\r]^{\ast}},$$
where $C$ is a positive constant independent of $L$.
\end{theorem}

\begin{proof}
Let $(\rz,\,q,\,s)_{\oz}$ be an admissible triplet. Obviously, the
proof of $\mathrm{(i)}$ is similar to the proof of Lemma
\ref{l7.4}(i). We omit the details.

Now, we prove $\mathrm{(ii)}$ by considering the following two cases
for $\oz$.

{\it Case i)} $\oz(\rn)=\fz$. In this case, let
 $Q\subset\rn$ be a cube with $l(Q)\in[1,\fz)$. We first
prove that
\begin{eqnarray}\label{7.14}
\left[h^{\rz,\,q,\,s}_{\oz}
(\rn)\r]^{\ast}\subset\left[L^q_{\oz}(Q)\r]^{\ast}.
\end{eqnarray}
Obviously, for any given $f\in L^q_{\oz}(Q)$,
$$a\equiv[\oz(Q)]^{1/q-1
}[\rz(\oz(Q))]^{-1}\|f\|_{L^q_{\oz}(Q)}^{-1}f\chi_Q$$ is a
$(\rz,\,q,\,s)_{\oz}$-atom. Thus, $f\in h^{\rz,\,q,\,s}_{\oz} (\rn)$
and
$$\|f\|_{h^{\rz,\,q,\,s}_{\oz}
(\rn)}\le[\oz(Q)]^{1/q'}\rz(\oz(Q))\|f\|_{L^q_{\oz}(Q)},$$ which
implies that for any $L\in\left[h^{\rz,\,q,\,s}_{\oz} (\rn)\r]^{\ast}$,
$$|Lf|\le\|L\|_{\left[h^{\rz,\,q,\,s}_{\oz} (\rn)\r]^{\ast}}
\|f\|_{h^{\rz,\,q,\,s}_{\oz}
(\rn)}\le[\oz(Q)]^{1/q'}\rz(\oz(Q))\|f\|_{L^q_{\oz}(Q)}
\|L\|_{\left[h^{\rz,\,q,\,s}_{\oz} (\rn)\r]^{\ast}}.$$ That is,
$L\in[L^q_{\oz}(Q)]^{\ast}$. Thus, \eqref{7.14} holds.

Now, assume that $L\in\left[h^{\rz,\,q,\,s}_{\oz} (\rn)\r]^{\ast}$.
Similarly to the proof of \eqref{7.5}, we know that there exists a
$\wz{g}\in L^{q'}_{\oz}(Q)$ such that for all $f\in L^q_{\oz}(Q)$,
$$Lf=\int_{Q}f(x)\wz{g}(x)\oz(x)\,dx.$$
Taking a sequence $\{Q_j\}_{j\in\nn}$ of cubes  such that for any
$j\in\nn$, $Q_j \subset Q_{j+1}$, $\lim_{j\to\fz}Q_j =\rn$ and
$l(Q_1)\in[1,\fz)$. From the above result, it follows that for each
$Q_j$, there exists a $\wz{g}_j \in L^{q'}_{\oz}(Q_j)$ such that for
all $f\in L^q_{\oz}(Q_j)$,
\begin{eqnarray}\label{7.15}
Lf=\int_{Q_j}f(x)\wz{g}_j (x)\oz(x)\,dx.
\end{eqnarray}

Now, we construct a function $g$ on $\rn$ such that
$$Lf=\int_{Q_j}f(x)g(x)\,dx$$
holds for all $f\in L^q_{\oz}(Q_j)$ and
$j\in\nn$. First, assume that $f\in L^q_{\oz}(Q_1)$. By
\eqref{7.15}, we know that there exists a $\wz{g}_1 \in
L^{q'}_{\oz}(Q_1)$ such that
$$Lf=\int_{Q_1}f(x)\wz{g}_1
(x)\oz(x)\,dx.$$
Notice that $f\in L^q_{\oz}(Q_1)\subset
L^q_{\oz}(Q_2)$. By \eqref{7.15} again, there exists a $\wz{g}_2 \in
L^{q'}_{\oz}(Q_2)$ such that
$$Lf=\int_{Q_1}f(x)\wz{g}_1 (x)\oz(x)\,dx=\int_{Q_2}f(x)
\wz{g}_2 (x)\oz(x)\,dx,$$ which implies that for all $f\in
L^q_{\oz}(Q_1)$,
\begin{eqnarray*}
\int_{Q_1}f(x)[\wz{g}_1 (x)-\wz{g}_2 (x)]\oz(x)\,dx=0.
\end{eqnarray*}
Thus, we obtain that for almost every $x\in Q_1$, $\wz{g}_1 (x)
=\wz{g}_2 (x)$. For $j=1,\,2$, let $g_j\equiv \wz{g}_j \oz$ and
$$g(x)\equiv\begin{cases}g_1(x),&\text{when}\,\,x\in Q_1,\\
g_2(x),&\text{when}\,\, x\in Q_2 \setminus Q_1.\end{cases}$$ It is
easy to see that for all $f\in L^q_{\oz}(Q_j)$ with $j\in\{1,\,2\}$,
\begin{eqnarray}\label{7.16}
Lf=\int_{Q_j}f(x)g_j (x)\,dx.
\end{eqnarray}
Continuing in this way, we obtain a function $g$ on $\rn$ such that
\eqref{7.16} holds for all $j\in\nn$.

Finally, we show that $g\in\bmo^{q'}_{\rz,\,\oz}(\rn)$ and for all
$f\in h^{\rz,\,q,\,s}_{\oz,\,\fin}(\rn)$,
\begin{eqnarray}\label{7.17}
Lf=\int_{\rn}f(x)g(x)\,dx.
\end{eqnarray}
Indeed, since $\oz(\rn)=\fz$, all $(\rz,\,q)_{\oz}$-single-atoms are
0, and for any $(\rz,\,q,\,s)_{\oz}$-atom $a$, there exists a
$j_0\in\nn$ such that $a\in L^q_{\oz}(Q_{j_0})$. By this and the
fact that \eqref{7.16} holds for all $j\in\nn$, we see that
\eqref{7.17} holds.

Now, we prove that $g \in\bmo^{q'}_{\rz,\,\oz}(\rn)$. Take any cube
$Q\subset\rn$ with $l(Q)\in[1,\fz)$ as well as any $f\in
L^q_{\oz}(Q)$ with $\|f\|_{L^q_{\oz}(Q)}\le1$. Let
$$a\equiv[\oz(Q)]^{-1/q'}[\rz(\oz(Q))]^{-1}f \chi_Q.$$
Then $a$ is a $(\rz,\,q,\,s)_{\oz}$-atom and $\supp(a)\subset Q$.
From the equality
$$La=\int_{Q}a(x)g(x)\,dx$$
and $L\in[h^{\rz,\,q,\,s}_{\oz}(\rn)]^{\ast}$, we deduce that
$$|La|=\left|\int_{Q}a(x)g(x)\,dx\r|\le
\|L\|_{\left[h^{\rz,\,q,\,s}_{\oz} (\rn)\r]^{\ast}}.$$ Thus, for any
$f\in L^q_{\oz}(Q)$ with $\|f\|_{L^q_{\oz}(Q)}\le1$, we have
$$[\oz(Q)]^{-1/q'}[\rz(\oz(Q))]^{-1}
\left|\int_{Q}f(x)g(x)\,dx\r|\le\|L\|_{\left[h^{\rz,\,q,\,s}_{\oz}
(\rn)\r]^{\ast}},$$ which implies that
\begin{eqnarray}\label{7.18}
[\oz(Q)]^{-1/q'}[\rz(\oz(Q))]^{-1}\left\{\int_Q
|g(x)|^{q'}[\oz(x)]^{1-q'}\,dx\r\}^{1/q'}
\le\|L\|_{\left[h^{\rz,\,q,\,s}_{\oz} (\rn)\r]^{\ast}}.
\end{eqnarray}
Furthermore, from $h^{\rz,\,q,\,s}_{\oz}(\rn)\supset
H^{\rz,\,q,\,s}_{\oz}(\rn)$ and
$$\|f\|_{h^{\rz,\,q,\,s}_{\oz}(\rn)}\le
\|f\|_{H^{\rz,\,q,\,s}_{\oz}(\rn)}$$
for all $f\in H^{\rz,\,q,\,s}_{\oz}(\rn)$, we deduce that
$$[h^{\rz,\,q,\,s}_{\oz}(\rn)]^{\ast}\subset[H^{\rz,\,q,\,s}_{\oz}(\rn)]
^{\ast}$$
and $L|_{H^{\rz,\,q,\,s}_{\oz}(\rn)}\in[H^{\rz,\,q,\,s}_{\oz}(\rn)]
^{\ast}$. Since \eqref{7.17} holds for all $f\in
H^{\rz,\,q,\,s}_{\oz,\,\fin}(\rn)$, by Lemma \ref{l7.4}(ii), we know
that $g\in \bbmo^{q'}_{\rz,\,\oz}(\rn)$ and
$$\|g\|_{\bbmo^{q'}_{\rz,\,\oz}(\rn)}\ls
\|L\mid_{H^{\rz,\,q,\,s}_{\oz}(\rn)}\|_{\left[H^{\rz,\,q,\,s}_{\oz}
(\rn)\r]^{\ast}}\ls\|L\|_{\left[h^{\rz,\,q,\,s}_{\oz}
(\rn)\r]^{\ast}}.$$ Thus, this estimate together with \eqref{7.18}
implies that $g\in\bmo^{q'}_{\rz,\,\oz}(\rn)$ and
$$\|g\|_{\bmo^{q'}_{\rz,\,\oz}(\rn)}\ls
\|L\|_{\left[h^{\rz,\,q,\,s}_{\oz} (\rn)\r]^{\ast}},$$ which completes
the proof of Theorem \ref{t7.5}(ii) in Case i).

{\it Case ii)} $\oz(\rn)<\fz$. In this case, let
\begin{eqnarray*}
\widetilde{h^{\rz,\,q,\,s}_{\oz}}(\rn)\equiv&&\left\{f=\sum_{i=1}^{\fz}\lz_i
a_i \,\,\text{in}\,\,\cd'(\rn):\ \text{For}\ i\in\nn,\
a_i\,\,\text{is a}\,\,
(\rz,\,q,\,s)_{\oz}\text{-atom,}\r.\\
&&\hs\ \supp (a_i)\subset Q_i, \,\lz_i\in\cc\ \text{and}\
\sum_{i=1}^{\fz}\oz(Q_i)\bfai\left(\frac{|\lz_i|}{\oz(Q_i)
\rz(\oz(Q_i))}\r)<\fz\Bigg\}
\end{eqnarray*}
and for all $f\in\widetilde{h^{\rz,\,q,\,s}_{\oz}}(\rn)$,
$$\|f\|_{\widetilde{h^{\rz,\,q,\,s}_{\oz}}(\rn)}\equiv\inf\left\{\blz\left(\{\lz_i
a_i\}_{i=1}^{\fz}\r)\r\},$$
where the infimum is taken over all the
decompositions of $f$ as above. For any $f\in L^1_{\loc}(\rn)$, let
\begin{eqnarray*}
\|f\|_{\widetilde{\bmo_{\rz,\,\oz}^{q'}} (\rn) }&\equiv&
\sup_{Q\subset\rn,\,|Q|<1}\frac{1}{\rz(\oz(Q))}
\left\{\frac{1}{\oz(Q)}\int_{Q}|f(x)-P^s_Q f(x)|^{q'}
[\oz(x)]^{1-q'}\,dx\r\}^{1/q'}\\
&&\hs\hs +\sup_{Q\subset\rn,\,|Q|\ge1}\frac{1}{\rz(\oz(Q))}
\left\{\frac{1}{\oz(Q)}\int_{Q}|f(x)|^{q'}
[\oz(x)]^{1-q'}\,dx\r\}^{1/q'}
\end{eqnarray*}
and
\begin{eqnarray*}
\widetilde{\bmo_{\rz,\,\oz}^{q'}} (\rn)\equiv\left\{f\in
L^1_{\loc}(\rn):\ \|f\|_{\widetilde{\bmo_{\rz,\,\oz}^{q'}} (\rn)
}<\fz\r\}.
\end{eqnarray*}
Similarly to the proofs of (i) and Case i), we conclude that
\begin{eqnarray}\label{7.19}
\left[\widetilde{h^{\rz,\,q,\,s}_{\oz}}(\rn)\r]^{\ast}
=\widetilde{\bmo_{\rz,\,\oz}^{q'}} (\rn).
\end{eqnarray}

Now we claim that
\begin{eqnarray}\label{7.20}
\left[h^{\rz,\,q,\,s}_{\oz}
(\rn)\r]^{\ast}\subset\left[L^q_{\oz}(\rn)\r]^{\ast}.
\end{eqnarray}
Indeed, for any  $f\in L^q_{\oz}(\rn)$, let
$$a\equiv[\oz(\rn)]^{1/q
-1}[\rz(\oz(\rn))]^{-1}\|f\|_{L^q_{\oz}(\rn)}^{-1}f.$$
Then $a$ is a
$(\rz,\,q)_{\oz}$-single-atom, which implies that $f\in
h^{\rz,\,q,\,s}_{\oz} (\rn)$ and
$$\|f\|_{h^{\rz,\,q,\,s}_{\oz}
(\rn)}\le[\oz(\rn)]^{1/q'}\rz(\oz(\rn))\|f\|_{L^q_{\oz}(\rn)}.$$
Thus, for any given $L\in\left[h^{\rz,\,q,\,s}_{\oz} (\rn)\r]^{\ast}$
and all $f\in L^q_{\oz}(\rn)$, we have
$$|Lf|\le\|L\|_{\left[h^{\rz,\,q,\,s}_{\oz}
(\rn)\r]^{\ast}}\|f\|_{h^{\rz,\,q,\,s}_{\oz} (\rn)}\le
[\oz(\rn)]^{1/q'}\rz(\oz(\rn))\|f\|_{L^q_{\oz}(\rn)}
\|L\|_{\left[h^{\rz,\,q,\,s}_{\oz} (\rn)\r]^{\ast}}.$$ That is,
$L\in[L^q_{\oz}(\rn)]^{\ast}$. Thus, \eqref{7.20} holds.

Now, assume that $L\in\left[h^{\rz,\,q,\,s}_{\oz} (\rn)\r]^{\ast}$.
From $\oz(\rn)<\fz$, it follows that
$$L^{\fz}_{\oz}(\rn)\subset
L^{\gz}_{\oz}(\rn)$$
for any $\gz\in[1,\fz)$ and
$L^{\gz'}_{\oz}(\rn)\subset L^1_{\oz}(\rn)$. By this,
\eqref{7.20}, the Hahn-Banach theorem and the Riesz represent
theorem, we conclude that there exists a $\wz{g}\in L^{q'}_{\oz}(\rn)$
such that for all $f\in L^q_{\oz}(\rn)$ with $q\in(q_{\oz},\fz]$,
$$Lf=\int_{\rn}f(x)\wz{g}(x)\oz(x)\,dx.$$
Let $g\equiv \wz{g}\oz$. Then we know that for all $f\in
L^q_{\oz}(\rn)$,
\begin{eqnarray}\label{7.21}
Lf=\int_{\rn}f(x)g(x)\,dx.
\end{eqnarray}

Finally, we prove that $g\in\bmo^{q'}_{\rz,\,\oz}(\rn)$ and
$$\|g\|_{\bmo^{q'}_{\rz,\,\oz}(\rn)}\ls\|L\|_{\left[h^{\rz,\,q,\,s}_{\oz}
(\rn)\r]^{\ast}}.$$
Obviously, \eqref{7.21} holds for all $f\in
h^{\rz,\,q,\,s}_{\oz,\,\fin}(\rn)$. For any $f\in L^q_{\oz}(\rn)$
with $\|f\|_{L^q_{\oz}(\rn)}\le1$, let
$$a\equiv[\oz(\rn)]^{-1/q'}[\rz(\oz(\rn))]^{-1}f.$$
Then $a$ is a
$(\rz,\,q)_{\oz}$-single-atom. From \eqref{7.21} with $f\equiv a$
and $L\in [h^{\rz,\,q,\,s}_{\oz}(\rn)]^{\ast}$, we deduce that
$$|La|=\left|\int_{\rn}a(x)g(x)\,dx\r|\le\|L\|_{\left[h^{\rz,\,q,\,s}_{\oz}
(\rn)\r]^{\ast}}.$$ That is,
$$[\oz(\rn)]^{-1/q'}[\rz(\oz(\rn))]^{-1}
\left|\int_{\rn}f(x)g(x)\,dx\r|\le\|L\|_{\left[h^{\rz,\,q,\,s}_{\oz}
(\rn)\r]^{\ast}},$$
which together with $\|f\|_{L^q_{\oz}(\rn)}\le1$
implies that
\begin{eqnarray}\label{7.22}
[\oz(\rn)]^{-1/q'}[\rz(\oz(\rn))]^{-1}\left\{\int_{\rn}
|g(x)|^{q'}[\oz(x)]^{1-q'}\,dx\r\}^{1/q'}
\le\|L\|_{\left[h^{\rz,\,q,\,s}_{\oz} (\rn)\r]^{\ast}}.
\end{eqnarray}
Moveover, from $h^{\rz,\,q,\,s}_{\oz}(\rn)\supset
\widetilde{h^{\rz,\,q,\,s}_{\oz}}(\rn)$ and
$$\|f\|_{h^{\rz,\,q,\,s}_{\oz}(\rn)}\le
\|f\|_{\widetilde{h^{\rz,\,q,\,s}_{\oz}}(\rn)}$$
for all $f\in
\widetilde{h^{\rz,\,q,\,s}_{\oz}}(\rn)$, we conclude that
$[h^{\rz,\,q,\,s}_{\oz}(\rn)]^{\ast}\subset[\widetilde{
h^{\rz,\,q,\,s}_{\oz}}(\rn)]^{\ast}$ and
$$L\mid_{\widetilde{h^{\rz,\,q,\,s}_{\oz}}(\rn)}
\in\left[\widetilde{h^{\rz,\,q,\,s}_{\oz}}(\rn)\r]^{\ast}.$$
Thus, by \eqref{7.19} and \eqref{7.21}, we know that
$g\in\widetilde{\bmo^{q'}_{\rz,\,\oz}}(\rn)$ and
$$\|g\|_{\widetilde{\bmo^{q'}_{\rz,\,\oz}}(\rn)}\ls
\left\|L\mid_{\widetilde{h^{\rz,\,q,\,s}_{\oz}}(\rn)}\r\|_{[
\widetilde{h^{\rz,\,q,\,s}_{\oz}}(\rn)]^{\ast}}\ls
\|L\|_{\left[h^{\rz,\,q,\,s}_{\oz} (\rn)\r]^{\ast}},$$
which together with \eqref{7.22} implies that
$g\in\bmo^{q'}_{\rz,\,\oz}(\rn)$ and
$$\|g\|_{\bmo^{q'}_{\rz,\,\oz}(\rn)}\ls\|L\|_{\left[h^{\rz,\,q,\,s}_{\oz}
(\rn)\r]^{\ast}}.$$
This finishes the proof of Theorem \ref{t7.5}.
\end{proof}

When $q=1$, we denote $\bmo^{q}_{\rz,\,\oz}(\rn)$ simply by
$\bmo_{\rz,\,\oz}(\rn)$. By Theorems \ref{t5.6} and \ref{t7.5}, we
have the following conclusions.
\begin{corollary}\label{c7.6}
Let $\bfai$ satisfy Assumption $\mathrm{(A)}$, $\oz\in
A^{\loc}_{\fz}(\rn)$, $q_{\oz}$ and $\rz$ be respectively as in
\eqref{2.4} and \eqref{2.7}. Then for
$q\in[1,\frac{q_{\oz}}{q_{\oz}-1})$,
$\bmo^{q}_{\rz,\,\oz}(\rn)=\bmo_{\rz,\,\oz}(\rn)$ with equivalent
norms.
\end{corollary}

\begin{corollary}\label{c7.7}
Let $\bfai$ satisfy Assumption $\mathrm{(A)}$, $\oz\in
A^{\loc}_{\fz}(\rn)$ and $\rz$ be as in \eqref{2.7}. Then
$[h^{\bfai}_{\oz} (\rn)]^{\ast}=\bmo_{\rz,\,\oz}(\rn)$.
\end{corollary}

\chapter{Some applications\label{s8}}

In this section, we first show that local Riesz transforms are
bounded on $h^{\bfai}_{\oz}(\rn)$. Moreover, we introduce local
fractional integrals and show that they are bounded from
$h^p_{\oz^p}(\rn)$ to $L^q_{\oz^q}(\rn)$ when $q\in[1,\fz)$, and
from $h^p_{\oz^p}(\rn)$ to $h^q_{\oz^q}(\rn)$ when $q\in(0,1]$.
Finally, we prove that some pseudo-differential operators are
bounded on $h^{\bfai}_{\oz}(\rn)$, where $\oz\in A_p (\phi)$ which
was introduced by Tang \cite{Ta2} (see also Definition \ref{d8.13}
below) and is contained in $A^{\loc}_{p}(\rn)$ for $p\in[1,\fz)$.

Now, we recall the notion of local Riesz transforms introduced by
Goldberg \cite{Go79}. In what follows, $\cs(\rn)$ denotes the
\emph{space of all Schwartz functions on $\rn$}.

\begin{definition}\label{d8.1}
Let $\phi_0\in\cd(\rn)$ such that $\phi_0\equiv1$ on $Q(0,1)$ and
$\supp(\phi_0)\subset Q(0,2)$. For $j\in\{1,\,2,\,\cdots,\,n\}$ and
$x\in\rn$, let
$$k_j (x)\equiv\frac{x_j}{|x|^{n+1}}\phi_0(x).$$
For $f\in\cs(\rn)$, the \emph{local Riesz transform} $r_j (f)$ of
$f$ is defined by $r_j (f)\equiv k_j \ast f$.
\end{definition}

We remark that $\phi_0$ in \cite{Go79} was assumed that
$\phi_0\equiv1$ in a neighborhood of the origin and
$\phi_0\in\cd(\rn)$. In this paper, for convenience, we assume
$\phi_0\equiv1$ on $Q(0,1)$ and $\supp(\phi_0)\subset Q(0,2)$. We have
the boundedness on $h^{\bfai}_{\oz}(\rn)$ of local Riesz transforms
$\{r_j\}_j$ as follows.

\begin{theorem}\label{t8.2}
Let $\bfai$ satisfy Assumption $\mathrm{(A)}$, $\oz\in
A^{\loc}_{\fz}(\rn)$ and $p_{\bfai}$ be as in \eqref{2.6}. For
$j\in\{1,\,2,\,\cdots,\,n\}$, let $r_j$ be the local Riesz operator
as in Definition \ref{d8.1}. If $p_{\bfai}=p_{\bfai}^+$ and $\bfai$
is of upper type $p_{\bfai}^+$, then there exists a positive
constant $C_0 \equiv C_0 (\bfai,\,\oz,\,n)$, depending only on
$\bfai$, $q_{\oz}$, the weight constant of $\oz$ and $n$, such that
for all $f\in h^{\bfai}_{\oz}(\rn)$,  $$\left\|r_j
(f)\r\|_{h^{\bfai}_{\oz}(\rn)}\le C_0
\|f\|_{h^{\bfai}_{\oz}(\rn)}.$$
\end{theorem}

To prove Theorem \ref{t8.2}, we need the following lemma established
in \cite[Lemma\,8.2]{Ta1}.

\begin{lemma}\label{l8.3}
For $j\in\{1,\,2,\,\cdots,\,n\}$, let $r_j$ be the local Riesz
operator as in Definition \ref{d8.1}.

$\mathrm{(i)}$ For $\oz\in A^{\loc}_p (\rn)$ with $p\in(1,\,\fz)$,
then there exists a positive constant
$$C_1 \equiv C_1(p,\,\oz,\,n),$$
depending only on $p$, the weight constant of $\oz$,
and $n$, such that for all $f\in L^p_{\oz}(\rn)$,
$$\left\|r_j (f)\r\|_{L^p_{\oz}(\rn)}\le
C_1\|f\|_{L^p_{\oz}(\rn)}.$$

$\mathrm{(ii)}$ For $\oz\in A^{\loc}_1 (\rn)$, there exists a
positive constant $C_2 \equiv C_2 (\oz,\,n)$, depending only on the
weight constant of $\oz$, and $n$, such that for all $f\in
L^1_{\oz}(\rn)$,
$$\|r_j (f)\|_{L^{1,\fz}_{\oz}(\rn)}\le
C_2\|f\|_{L^1_{\oz}(\rn)}.$$
\end{lemma}

Now, we prove Theorem \ref{t8.2} by using Theorem \ref{6.2} and
Lemma \ref{l8.3}.

\begin{proof}[Proof of Theorem \ref{t8.2}]
Let $s\equiv\lfz n(q_{\oz}/p_{\bfai}-1)\rf$, where $q_{\oz}$ and
$p_{\bfai}$ are respectively as in \eqref{2.4} and \eqref{2.6}. Then
$(n+s+1)p_{\bfai}>nq_{\oz}$, which implies that there exists
$q\in(q_{\oz},\fz)$ such that $(n+s+1)p_{\bfai}>nq$ and $\oz\in
A^{\loc}_q(\rn)$. To show Theorem \ref{t8.2}, by Theorem
\ref{t6.4}(i) and Theorem \ref{3.2}, it suffices to show that for
any $(\rz,\,q)_{\oz}$-single-atom $a$ or $(\rz,\,q,\,s)_{\oz}$-atom
$a$ supported in $Q(x_0,\,R_0)$ with $R_0\in(0,2]$,
\begin{eqnarray}\label{8.1}
\left\|\cg^0_N \left(r_j (a)\r)\r\|_{L^{\bfai}_{\oz}(\rn)}\ls1.
\end{eqnarray}

First, we prove \eqref{8.1} for any $(\rz,\,q)_{\oz}$-single-atom
$a\neq0$. In this case, $\oz(\rn)<\fz$. Since $\bfai$ is concave, by
Jensen's inequality, H\"older's inequality, Proposition
\ref{p3.2}(ii), Lemma \ref{l8.3}(i) and \eqref{2.8} with
$t\equiv\oz(\rn)$, we have
\begin{eqnarray*}
&&\int_{\rn}\bfai\left(\cg^0_N\left(r_j (a)\r)(x)\r)\oz(x)\,dx\\
&&\hs\le\oz(\rn)\bfai\left(\frac{1}{\oz(\rn)} \int_{\rn}
\cg^0_N(r_j (a))(x)\oz(x)\,dx\r)\\
&&\hs\le\oz(\rn)\bfai\left(\frac{1}{[\oz(\rn)]^{1/q}}
\left\{\int_{\rn}\left[\cg^0_N(r_j (a))(x)\r]^{q}\oz(x)\,dx\r\}^{1/q}\r)\\
&&\hs\ls\oz(\rn)\bfai\left(\frac{1}{[\oz(\rn)]^{1/q}} \|r_j
(a)\|_{L^q_{\oz}(\rn)}\r)\ls\oz(\rn)\bfai\left(\frac{1}{[\oz(\rn)]^{1/q}}
\|a\|_{L^q_{\oz}(\rn)}\r)\\
&&\hs\ls\oz(\rn) \bfai\left(\frac{1}{\oz(\rn)\rz(\oz(\rn))}\r)\sim1,
\end{eqnarray*}
which implies \eqref{8.1} in this case.

Now, let $a$ be any $(\rz,\,q,\,s)_{\oz}$-atom supported in
$Q_0\equiv Q(x_0,\,R_0)$ with $R_0\in(0,2]$. We prove \eqref{8.1}
for $a$ by considering the following two cases for $R_0$.

{\it Case 1)} $R_0 \in[1,2]$. In this case, by the definitions of
$r_j (a)$ and $\cg^0_N (r_j (a))$, we see that
$$\supp \left(\cg^0_N (r_j (a))\r)\subset
Q_0^{\ast}\equiv Q(x_0, R_0 +8).$$ From this, Jensen's inequality,
H\"older's inequality, Proposition \ref{p3.2}(ii), Lemmas \ref{l8.3}
and \ref{l2.3}(v), Remark \ref{r2.4} with $\wz{C}\equiv2$ and
\eqref{2.8} with $t\equiv\oz(Q_0)$, we infer that
\begin{eqnarray*}
&&\int_{\rn}\bfai\left(\cg^0_N(r_j (a))(x)\r)\oz(x)\,dx\\
&&\hs\le\oz(Q_0^{\ast})\bfai\left(\frac{1}{\oz(Q_0^{\ast})}
\int_{Q_0^{\ast}}
\cg^0_N(r_j (a))(x)\oz(x)\,dx\r)\\
&&\hs\le\oz(Q_0^{\ast})\bfai\left(\frac{1}{[\oz(Q_0^{\ast})]^{1/q}}
\left\{\int_{Q_0^{\ast}}\left[\cg^0_N(r_j (a))(x)\r]^{q}\oz(x)\,dx\r\}^{1/q}\r)\\
&&\hs\ls\oz(Q_0^{\ast})\bfai\left(\frac{1}{[\oz(Q_0)]^{1/q}} \left\|r_j
(a)\r\|_{L^q_{\oz}(\rn)}\r)\ls\oz(Q_0)\bfai\left(\frac{1}{[\oz(Q_0)]^{1/q}}
\|a\|_{L^q_{\oz}(\rn)}\r)\\
&&\hs\ls\oz(Q_0) \bfai\left(\frac{1}{\oz(Q_0)\rz(\oz(Q_0))}\r)\sim1,
\end{eqnarray*}
which implies \eqref{8.1} in Case 1).

{\it Case 2)} $R_0\in(0,1)$. In this case, let $\wz{Q}_0
\equiv8nQ_0$. Then
\begin{eqnarray}\label{8.2}
\int_{\rn}\bfai\left(\cg^0_N(r_j (a))(x)\r)\oz(x)\,dx
&=&\int_{\wz{Q}_0}\bfai\left(\cg^0_N(r_j
(a))(x)\r)\oz(x)\,dx+\int_{(\wz{Q}_0)^\complement}\cdots\nonumber\\
 &\equiv&\mathrm{I_1}+\mathrm{I_2}.
\end{eqnarray}

For $\mathrm{I_1}$, similarly to the proof of Case 1), we have
\begin{eqnarray}\label{8.3}
\mathrm{I_1}&\le&\oz(\wz{Q}_0)\bfai\left(\frac{1}
{[\oz(\wz{Q}_0)]^{1/q}}
\left\{\int_{\wz{Q}_0}\left[\cg^0_N(r_j (a))(x)\r]^{q}\oz(x)\,dx\r\}^{1/q}\r)\nonumber\\
 &\ls&\oz(\wz{Q}_0)\bfai\left(\frac{1}{[\oz(Q_0)]^{1/q}}
\left\|r_j (a)\r\|_{L^q_{\oz}(\rn)}\r)\nonumber\\
&\ls&\oz(Q_0)\bfai\left(\frac{1}{\oz(Q_0)\rz(\oz(Q_0))}\r)\sim1.
\end{eqnarray}

To estimate $\mathrm{I_2}$, let $x\in(\wz{Q}_0)^\complement$,
$t\in(0,1)$, $\pz\in\cd^0_N (\rn)$ and $P_{\pz}^s$ be the Taylor
expansion of $\pz$ about $(x-x_0)/t$ with degree $s$. Then by the
vanishing condition of $a$, we see that
\begin{eqnarray}\label{8.4}
|r_j (a)\ast\pz_t (x)|&=&\frac{1}{t^n}\left|\int_{\rn}r_j (a)(y)
\pz\left(\frac{x-y}{t}\r)\,dy\r|\nonumber \\
&=&\frac{1}{t^n}\left|\int_{\rn} r_j
(a)(y)\left\{\pz\left(\frac{x-y}{t}\r)-P_{\pz}^s
\left(\frac{x-y}{t}\r)\r\}\,dy\r|\nonumber\\
 &\le&\frac{1}{t^n}\int_{2\sqrt{n}Q_0} \left|r_j
(a)(y)\r|\left|\pz\left(\frac{x-y}{t}\r)-P_{\pz}^s
\left(\frac{x-y}{t}\r)\r|\,dy\nonumber\\
&&+\frac{1}{t^n}\int_{Q(x_0,\frac{|x-x_0|}{2\sqrt{n}})
\setminus{(2\sqrt{n}Q_0})}\cdots+
\frac{1}{t^n}\int_{Q(x_0,\frac{|x-x_0|}{2\sqrt{n}})^\complement}\cdots
\nonumber \\
&\equiv&\mathrm{G_1}+\mathrm{G_2}+\mathrm{G_3}.
\end{eqnarray}

To estimate $\mathrm{G_1}$, by $t\in(0,1)$ and
$x\in(\wz{Q}_0)^{\complement}$, we see that $\mathrm{G_1}\neq0$
implies that $t>\frac{3|x-x_0|}{4}$. From this, Taylor's remainder
theorem, H\"older's inequality, Lemma \ref{l8.3}(i), \eqref{2.1} and
Remark \ref{r2.2}(ii), we deduce that
\begin{eqnarray}\label{8.5}
\hs\hs\mathrm{G_1}&\ls&\frac{1}{t^{n+s+1}}\left\|r_j
(a)\r\|_{L^q_{\oz}(\rn)} \left\{\sum_{\az\in\zz^n_+,\,|\az|=s+1}
\int_{2\sqrt{n}Q_0}\left|\left(\partial^{\az}
\pz\r)\left(\frac{\xi}{t}\r)\r|^{q'}\r.\nonumber\\
 &&\times|y-x_0|
^{(s+1){q'}}[\oz(y)]^{-q'/q}\,dy\Bigg\}^{1/q'}\nonumber\\
&\ls&\frac{R_0^{s+1}}{|x-x_0|^{n+s+1}}\|a\|_{L^q_{\oz}(\rn)}
\left\{\int_{2\sqrt{n}Q_0}[\oz(y)]^{-q'/q}\,dy\r\}^{1/q'}\nonumber\\
&\ls&\frac{1}{\oz(Q_0)\rz(\oz(Q_0))}\frac{R_0^{n+s+1}}{|x-x_0|^{n+s+1}},
\end{eqnarray}
where $\tz\in(0,1)$, $\xi\equiv\tz(x-y)+(1-\tz) (x-x_0)$ and
$\frac{1}{q}+\frac{1}{q'}=1$.

To estimate $\mathrm{G_2}$,  by the definition of $k_j$ with
$j\in\{1,\,2,\,\cdots,\,n\}$, we have
\begin{eqnarray}\label{8.6}
\sum_{\az\in\zz^n_+,\,|\az|=s+1}\left|(\partial^{\az}k_j)
(z)\r|\ls\frac{1}{|z|^{n+s+1}}
\end{eqnarray}
for all $z\in\rn\setminus\{0\}$. For any fixed $y\in
(Q(x_0,\frac{|x-x_0|}{2\sqrt{n}})\setminus 2\sqrt{n}Q_0)$, let
$K_j^s$ be the Taylor expansion of $k_j (\cdot)$ at the point
$y-x_0$ with degree $s$. Moreover, it is easy to see that
$\mathrm{G_2}\neq0$ implies that $t>\frac{|x-x_0|}{2}$. From this,
Taylor's remainder theorem, \eqref{8.6}, H\"older's inequality,
Lemma \ref{l8.3}(i) and \eqref{2.1}, we conclude that
\begin{eqnarray}\label{8.7}
\hs\hs\mathrm{G_2}&\le&\frac{1}{t^n}\int_{Q(x_0,\frac{
|x-x_0|}{2\sqrt{n}})\setminus
{(2\sqrt{n}Q_0)}}\left\{\int_{Q_0}|a(z)||k_j (y-z)-K_j^s (y-z)|\,dz\r\}\nonumber\\
&&\times\left|\pz\left(\frac{x-y}{t}\r)-P_{\pz}^s \left(\frac{x-y}{t}\r)\r|\,dy\nonumber\\
&\ls&\frac{1}{t^{n+s+1}}\int_{Q(x_0,\frac{|x-x_0|}{2\sqrt{n}})\setminus
{(2\sqrt{n}Q_0)}}\left\{\int_{Q_0}|a(z)|\frac{|z-x_0|^{s+1}}
{|\xi|^{n+s+1}}\,dz\r\}\nonumber\\  &&\times|y-x_0|^{s+1}\,dy\nonumber\\
&\ls&\frac{1}{|x-x_0|^{n+s+1}}\int_{Q(x_0,\frac{|x-x_0|}{2\sqrt{n}})\setminus
{(2\sqrt{n}Q_0)}}\frac{1}{|y-x_0|^n}\left\{\int_{Q_0}|a(z)|
|z-x_0|^{s+1}\,dz\r\}\,dy\nonumber\\
&\ls&\frac{R_0^{s+1}}{|x-x_0|^{n+s+1}}\|a\|_{L^q_{\oz}(\rn)}\frac{|Q_0|}
{[\oz(Q_0)]^{1/q}}\int_{Q(x_0, \frac{|x-x_0|}{2\sqrt{n}})\setminus
{(2\sqrt{n}Q_0)}}\frac{1}{|y-x_0|^n}\,dy\nonumber\\
&\ls&\frac{1}{\oz(Q_0)\rz(\oz(Q_0))}
\frac{R_0^{n+s+1}}{|x-x_0|^{n+s+1}}\int_{\sqrt{n}R_0}^{\frac{|x-x_0|}
{2\sqrt{n}}}z^{-1}\,dz\nonumber\\
&\ls&\frac{1}{\oz(Q_0)\rz(\oz(Q_0))}
\frac{R_0^{n+s+1-\dz}}{|x-x_0|^{n+s+1-\dz}},
\end{eqnarray}
where $\xi\equiv\gz(y-z)+(1-\gz)(y-x_0)$ for some $\gz\in(0,1)$,
$\dz$ is a small positive constant which is determined later, and in
the third inequality we used the fact that for any $y\in
Q(x_0,\frac{|x-x_0|}{2\sqrt{n}})\setminus {(2\sqrt{n}Q_0)}$ and
$z\in Q_0$,
$$|(y-x_0)-\gz(z-x_0)|\ge|y-x_0|-|y-x_0|/2=|y-x_0|/2.$$

Finally, we estimate $\mathrm{G_3}$.  For any $y\in
[Q(x_0,\frac{|x-x_0|}{2\sqrt{n}} )]^\complement$, by the definition
of $P_{\pz}^s$ and the support condition $\pz$, we have
\begin{eqnarray}\label{8.8}
\frac{1}{t^n}\left|P_{\pz}^s \left(\frac{x-y}{t}\r)\r|\ls
\frac{|y-x_0|^s}{|x-x_0|^{n+s}}.
\end{eqnarray}
Thus, from the vanishing condition of $a$, Taylor's remainder
theorem, \eqref{8.6}, H\"older's inequality, \eqref{2.1} and
\eqref{8.8}, we deduce that
\begin{eqnarray}\label{8.9}
\mathrm{G_3}&\ls&\frac{1}{t^n}\int_{Q(x_0,\frac{|x-x_0|}{2\sqrt{n}}
)^\complement}
\left\{\int_{Q_0}|a(z)||k_j (y-z)-K_j^s (y-z)|\,dz\r\}\nonumber\\
 &&\times\left\{\left|\pz\left(\frac{x-y}{t}\r)\r|+\left|P_{\pz}^s
\left(\frac{x-y}{t}\r)\r|\r\}\,dy\nonumber\\
&\ls&\frac{1}{t^n}\int_{Q(x_0,\frac{|x-x_0|}{2\sqrt{n}}
)^\complement}\left\{\int_{Q_0}|a(z)|\frac{|z-x_0|^{s+1}}
{|\xi|^{n+s+1}}\,dz\r\}\nonumber\\
&&\times\left\{\left|\pz\left(\frac{x-y}{t}\r)\r|+\left|P_{\pz}^s
\left(\frac{x-y}{t}\r)\r|\r\}\,dy\nonumber\\
&\ls&\|a\|_{L^q_{\oz}(\rn)}\frac{R_0^{s+n+1}}
{[\oz(Q_0)]^{1/q}}\frac{1}{t^n}
\int_{Q(x_0,\frac{|x-x_0|}{2\sqrt{n}})^\complement}
\frac{1}{|y-x_0|^{n+s+1}}\nonumber\\
&&\times\left\{\left|\pz\left(\frac{x-y}{t}\r)\r|+\left|P_{\pz}^s
\left(\frac{x-y}{t}\r)\r|\r\}\,dy\nonumber\\
&\ls&\|a\|_{L^q_{\oz}(\rn)}\frac{R_0^{s+n+1}}
{[\oz(Q_0)]^{1/q}}\left\{\frac{1}{|x-x_0|^{n+s+1}}
\frac{1}{t^n}\int_{Q(x_0,\frac{|x-x_0|}{2\sqrt{n}})^\complement}
\left|\pz\left(\frac{x-y}{t}\r)\r|\,dy\r.\nonumber\\  &&+\frac{1}{t^n}
\int_{Q(x_0,\frac{|x-x_0|}{2\sqrt{n}})^\complement}
\frac{1}{|y-x_0|^{n+s+1}}\left|P_{\pz}^s \left(\frac{x-y}{t}\r)\r|\,dy\Bigg\}
\nonumber\\
 &\ls&\frac{1}{\oz(Q_0)\rz(\oz(Q_0))}
\left\{\frac{R_0^{n+s+1}}{|x-x_0|^{n+s+1}}\r.\nonumber\\
&&+\frac{R_0^{n+s+1}}{|x-x_0|^{n+s}}\int_{Q(x_0,
\frac{|x-x_0|}{2\sqrt{n}})^\complement}
\frac{1}{|y-x_0|^{n+1}}\,dy\Bigg\}\nonumber \\
&\ls&\frac{1}{\oz(Q_0)\rz(\oz(Q_0))}
\frac{R_0^{n+s+1}}{|x-x_0|^{n+s+1}},
\end{eqnarray}
where $\xi\equiv\gz(y-z)+(1-\gz)(y-x_0)$ for some $\gz\in(0,1)$ and
in the third inequality we used the fact that for any $y\in [Q(x_0,
\frac{|x-x_0|}{2\sqrt{n}})]^\complement$ and $z\in Q_0$,
$$|(y-x_0)-\gz(z-x_0)|\gs|y-x_0|.$$
Thus, by \eqref{8.4}, \eqref{8.5},
\eqref{8.7}, \eqref{8.9} and $|x-x_0|\ge4nR_0$, we know that
\begin{eqnarray*}
\left|r_j (a)\ast\pz_t (x)\r|&\ls&\frac{1}{\oz(Q_0)\rz(\oz(Q_0))}
\left\{\frac{R_0^{(n+s+1)}}{|x-x_0|^{n+s+1}}+
\frac{R_0^{(n+s+1-\dz)}}{|x-x_0|^{n+s+1-\dz}}\r\}\\
&\ls&\frac{1}{\oz(Q_0)\rz(\oz(Q_0))}
\frac{R_0^{(n+s+1-\dz)}}{|x-x_0|^{n+s+1-\dz}},
\end{eqnarray*}
which together with the arbitrariness of $\pz\in\cd^0_N (\rn)$
implies that for all $x\in(\wz{Q}_0)^\complement$,
\begin{eqnarray}\label{8.10}
\cg^0_N \left(r_j (a)\r)(x)\ls\frac{1}{\oz(Q_0)\rz(\oz(Q_0))}
\frac{R_0^{(n+s+1-\dz)}}{|x-x_0|^{n+s+1-\dz}}.
\end{eqnarray}
Take $\dz\in(0,\fz)$ small enough such that
$p_{\bfai}(n+s+1-\dz)>nq$. By the fact that
$$\supp(\cg^0_N (r_j
(a)))\subset Q(x_0, R_0+8)\subset Q(x_0,9)$$
and Lemma \ref{l2.3}(i),
we know that there exists an $\wz{\oz}\in A_q (\rn)$ such that
$\wz{\oz}=\oz$ on $Q(x_0,9)$. Let $m_0$ be the integer such that
$2^{m_0 -1}nR_0\le9<2^{m_0}nR_0$. By \eqref{8.10}, the lower type
$p_{\bfai}$ property of $\bfai$, Lemma \ref{l2.3}(viii) and
$p_{\bfai}(n+s+1-\dz)>nq$, we conclude that
\begin{eqnarray*}
\mathrm{I_2}&\ls&\int_{Q(x_0,9)\setminus\wz{Q}_0}\bfai\left(\cg^0_N
(r_j (a))(x)\r)\wz{\oz}(x)\,dx\\
&\ls&\sum_{j=3}^{m_0}\int_{2^{j+1}nQ_0\setminus2^j n
 Q_0}\bfai\left(\frac{1}{\oz(Q_0)\rz(\oz(Q_0))}
\frac{R_0^{(n+s+1-\dz)}}{|x-x_0|^{n+s+1-\dz}}\r)\wz{\oz}(x)\,dx\\
&\ls&\frac{1}{\oz(Q_0)}\sum_{j=3}^{m_0}\int_{2^{j+1}nQ_0\setminus2^j
nQ_0}\left(\frac{R_0^{n+s+1-\dz}}{|x-x_0|^{n+s+1-\dz}}\r)
^{p_{\bfai}}\wz{\oz}(x)\,dx\\
&\ls&\sum_{j=3}^{m_0}2^{k[(n+s+1-\dz)p_{\bfai}-nq]}\ls1,
\end{eqnarray*}
which together with \eqref{8.2} and \eqref{8.3} implies \eqref{8.1}
in Case 2). This finishes the proof of Theorem \ref{t8.2}.
\end{proof}

\begin{remark}\label{r8.4}
Theorem \ref{t8.2} when $\oz\in A^{\loc}_1 (\rn)$ and
$\bfai(t)\equiv t$ for all $t\in(0,\fz)$ was obtained by Tang
\cite[Lemma\,8.3]{Ta1}.
\end{remark}

Next, we introduce the local fractional integral and, using Theorem
\ref{t6.4}, prove that they are boundedness from $h^p_{\oz^p}(\rn)$
to $L^q_{\oz^q}(\rn)$ when $q\in[1,\fz)$, and from
$h^p_{\oz^p}(\rn)$ to $h^q_{\oz^q}(\rn)$ when $q\in(0,1]$, provided
that $\oz$ satisfies $\oz^{\frac{nr}{nr-n-r\az}}\in A^{\loc}_1
(\rn)$ for some $r\in(\frac{n}{n-\az},\fz)$ and
$\int_{\rn}[\oz(x)]^p\,dx=\fz$. We begin with some notions.

\begin{definition}\label{d8.5}
Let $\az\in[0,n)$ and $\phi_0$ be as in Definition \ref{d8.1}. For
any $f\in\cs(\rn)$ and all $x\in\rn$, the \emph{local fractional
integral} $I^{\loc}_{\az}(f)$ of $f$ is defined by
$$I^{\loc}_{\az}(f)(x)\equiv\int_{\rn}\frac{\phi_0 (y)}{|y|^{n-\az}}
f(x-y)\,dy.$$
\end{definition}

\begin{definition}\label{d8.6}
(i) If there exist $r\in(1,\fz)$ and a positive constant $C$ such
that for all cubes $Q\subset\rn$ with sidelength $l(Q)\in(0,1]$,
\begin{eqnarray}\label{8.11}
\left(\frac{1}{|Q|}\int_{Q}[\oz(x)]^r \,dx\r)^{1/r}\le C
\left(\frac{1}{|Q|}\int_{Q}\oz(x) \,dx\r),
\end{eqnarray}
then $\oz$ is said to satisfy the {\it local reverse H\"older
inequality of order $r$,} which is denoted by $\oz\in RH^{\loc}_r
(\rn)$. Furthermore, let $RH^{\loc}_r (\oz)\equiv\inf C$, where the
infimum is taken over all the positive constants $C$ satisfying
\eqref{8.11}.

(ii) Let $p,\,q\in(1,\fz)$. A locally integrable nonnegative
function $\oz$ on $\rn$ is said to belong to the \emph{class}
$A^{\loc}(p,\,q)$, if there exists a positive constant $C$ such that
for all cubes $Q\subset\rn$ with sidelength $l(Q)\in(0,1]$,
\begin{eqnarray}\label{8.12}
\left(\frac{1}{|Q|}\int_{Q}[\oz(x)]^q \,dx\r)^{1/q}
\left(\frac{1}{|Q|}\int_{Q}[\oz(x)]^{-p'} \,dx\r)^{1/p'}\le C,
\end{eqnarray}
 where and in what follows, $\frac{1}{p}+\frac{1}{p'}=1$. Furthermore,
 let
$A^{\loc}(p,\,q)(\oz)\equiv\inf\{C\}$, where the infimum is taken
over all the positive constants $C$ satisfying \eqref{8.12}.
\end{definition}

\begin{remark}\label{r8.7}
(i) Let $r$ be as in Definition \ref{d8.6}(i). If \eqref{8.11} holds
for all cubes $Q\subset\rn$, then $\oz$ is said to satisfy the {\it
reverse H\"older inequality of order $r$,} which is denoted by
$\oz\in RH_r (\rn)$ (see, for example, \cite{gr85}). Let $p,\,q$ be
as in Definition \ref{d8.6}(ii). If \eqref{8.12} holds for all cubes
$Q\subset\rn$, then $\oz$ is said to belong to the \emph{class}
$A(p,\,q)$.

(ii) For any given positive constant $A_1$, let the cube $Q$ satisfy
$l(Q)=A_1$. Similarly to the proof of Lemma \ref{l2.3}(i), we have
that for any $\oz\in RH^{\loc}_r (\rn)$, there exists an
$\wz{\oz}\in RH_r (\rn)$ such that $\oz=\wz{\oz}$ on $Q$ and $RH_r
(\wz{\oz})\ls RH^{\loc}_r (\oz)$, where $RH_r (\wz{\oz})$ is defined
similarly to $RH^{\loc}_r (\oz)$ and the implicit constant depends
only on $A_1$ and $n$.

(iii) Similarly to Remark \ref{r2.2}(ii), for any given constant
$A_2\in(0,\fz)$, the condition $l(Q)\in(0,1]$ in \eqref{8.11} can be
replaced by $l(Q)\in(0,A_2]$ with the positive constant $C$ in
\eqref{8.11} depending on $A_2$.
\end{remark}

About the relations of $A^{\loc}_{\fz}(\rn)$, $RH^{\loc}_r (\rn)$
and $A^{\loc}(p,\,q)$, we have the following conclusions.

\begin{lemma}\label{l8.8}
$\mathrm{(i)}$ Let $r\in(1,\,\fz)$. Then $\oz^r \in
A^{\loc}_{\fz}(\rn)$ if and only if $\oz\in RH_r^{\loc}(\rn)$.

$\mathrm{(ii)}$ Let $\az\in(0,\,n)$, $p\in(1,\,n/\az)$ and
$1/q=1/p-\az/n$. Then $\oz\in A^{\loc}(p,\,q)$ if and only if
$\oz^{-p'}\in A^{\loc}_{1+p'/q}(\rn)$.
\end{lemma}

\begin{proof}
We first prove (i). Let $\oz^r \in A^{\loc}_{\fz}(\rn)$.  Then by
Lemma \ref{l2.3}(i), we know that for any cube $Q\equiv Q(x_0,l(Q))$
with $l(Q)\in(0,1]$, there exists a function $\wz{\oz}$ on $\rn$
such that
\begin{eqnarray}\label{8.13}
\wz{\oz}^r \in A_{\fz}(\rn) \ \text{and}\ \wz{\oz}=\oz \  \text{on}\
Q(x_0,1).
\end{eqnarray}
Moreover, by \cite[Lemma A]{dll03}, we know that
\begin{eqnarray}\label{8.14}
\wz{\oz}^r \in A_{\fz}(\rn) \ \text{if and only if}\ \wz{\oz}\in
RH_r (\rn).
\end{eqnarray}
Thus, for any cube $Q(x_0,l(Q))$ with $l(Q)\in(0,1]$, by
\eqref{8.13} and \eqref{8.14}, we have
\begin{eqnarray*}
\left(\frac{1}{|Q|}\int_{Q}[\oz(x)]^r \,dx\r)^{1/r}=
\left(\frac{1}{|Q|}\int_{Q}[\wz{\oz}(x)]^r \,dx\r)^{1/r}
\ls\frac{1}{|Q|}\int_{Q}\wz{\oz}(x)\,dx\sim\frac{1}{|Q|}
\int_{Q}\oz(x)\,dx,
\end{eqnarray*}
which together with the arbitrariness of the cube $Q(x_0,l(Q))$
implies that $\oz\in RH^{\loc}_r (\rn)$.

Conversely, let $\oz\in RH^{\loc}_r (\rn)$.  Then by Remark
\ref{r8.7}(ii), we know that for any cube $Q(x_0,l(Q))$ with
$l(Q)\in(0,1]$, there exists a function $\wz{\oz}$ on $\rn$ such
that $\wz{\oz}\in RH_r (\rn)$ and $\wz{\oz}=\oz$ on $Q(x_0,1)$,
which together with \eqref{8.14} and the arbitrariness of the cube
$Q(x_0,l(Q))$ implies that $\oz\in A^{\loc}_{\fz}(\rn)$. This
finishes the proof of (i).

By the definitions of $A^{\loc}(p,\,q)$ and
$A^{\loc}_{1+p'/q}(\rn)$, we see that (ii) holds, which completes
the proof of Lemma \ref{l8.8}.
\end{proof}

To establish the boundedness of local fractional integrals, we need
the following technical lemma.

\begin{lemma}\label{l8.9}
Let $\az\in(0,n)$, $p\in(1,\frac{n}{\az})$ and
$\frac{1}{q}=\frac{1}{p}-\frac{\az}{n}$. For some $r\in(q,\fz)$, if
$$\oz^{-r'}\in A^{\loc}(q'/r',\,p'/r'),$$
then there exists a positive constant $C$ such that for all $f\in
L^p_{\oz^p}(\rn)$,
\begin{eqnarray}\label{8.15}
\left\|I^{\loc}_{\az}(f)\r\|_{L^q_{\oz^q}(\rn)}\le
C\|f\|_{L^p_{\oz^p}(\rn)},
\end{eqnarray}
where $p',\,q'$ and $r'$ respectively denote the conjugate indices
of $p,\,q$ and $r$.
\end{lemma}

\begin{proof}
Let $\oz^{-r'}\in A^{\loc}(q'/r',\,p'/r')$. For any
unit cube $Q\subset\rn$, from Lemmas \ref{l8.8}(ii) and
\ref{l2.3}(i), and Remark \ref{r2.4}, we deduce that there exists a
function $\wz{\oz}$ on $\rn$ such that $\wz{\oz}^{-r'}\in
A(q'/r',\,p'/r')$ and $\wz{\oz}=\oz$ on $5Q$. For
$\wz{\oz}^{-r'}\in A(q'/r',\,p'/r')$, similarly to
the proof of \cite[Theorem\,2]{dl98}, we know that for all $f\in
L^p_{\wz{\oz}^p}(\rn)$,
$$\left\|I^{\loc}_{\az}(f)\r\|_{L^q_{\wz{\oz}^q}(\rn)}\ls
\|f\|_{L^p_{\wz{\oz}^p}(\rn)},$$ which combined with the definition
of $I^{\loc}_{\az}(f)$ implies that
\begin{eqnarray}\label{8.16}
\left\|I^{\loc}_{\az}(f)\r\|_{L^q_{\oz^q}(Q)}&=&
\left\|I^{\loc}_{\az}(f\chi_{5Q})\r\|_{L^q_{\wz{\oz}^q}(Q)}\ls
\|f\chi_{5Q}\|_{L^p_{\wz{\oz}^p}(\rn)} \sim\|f\|_{L^p_{\oz^p}(5Q)}.
\end{eqnarray}
Take unit cubes $\{Q_i\}_{i=1}^{\fz}$ such that
$\cup_{i=1}^{\fz}Q_i=\rn$, their interiors are disjoint and
$$\sum_{i=1}^{\fz}\chi_{5Q_i}\le M,$$
where $M$ is a positive integer depending only on $n$. From this and
\eqref{8.16}, we infer that
$$\left\|I^{\loc}_{\az}(f)\r\|^q_{L^q_{\oz^q}(\rn)}=\sum_{i=1}^{\fz}
\left\|I^{\loc}_{\az}(f)\r\|^q_{L^q_{\oz^q}(Q_i)}\ls\sum_{i=1}^{\fz}
\|f\|^q_{L^p_{\oz^p}(5Q_i)}\ls\|f\|^q_{L^p_{\oz^p}(\rn)},$$ which
implies \eqref{8.15}. This finishes the proof of Lemma \ref{l8.9}.
\end{proof}

\begin{theorem}\label{t8.10}
Let $\az\in(0,n)$, $p\in[\frac{n}{n+\az},1]$ and
$\frac{1}{q}=\frac{1}{p}-\frac{\az}{n}$. For some
$r\in(\frac{n}{n-\az},\fz)$, if the weight $\oz$ satisfies
$\oz^{\frac{nr}{nr-n-r\az}}\in A^{\loc}_1 (\rn)$ and
$\int_{\rn}[\oz(x)]^p\,dx=\fz$, then there exists a positive
constant $C$ such that for all $f\in h^p_{\oz^p}(\rn)$,
$$\left\|I^{\loc}_{\az}(f)\r\|_{L^q_{\oz^q}(\rn)}\le
C\|f\|_{h^p_{\oz^p}(\rn)}.$$
\end{theorem}

\begin{proof}
Let $r\in(\frac{n}{n-\az},\fz)$ and $\oz$ satisfy
$\oz^{\frac{nr}{nr-n-r\az}}\in A^{\loc}_1 (\rn)$ and
$\int_{\rn}[\oz(x)]^p\,dx=\fz$. Then by
$$\oz^{\frac{nr}{nr-n-r\az}}\in A^{\loc}_1 (\rn)$$
and Lemma \ref{l2.3}(ii), we know that there exists
an $\ez_1\in(0,\fz)$ such that
\begin{equation}\label{8.17}
\oz^{\frac{nr(1+\ez_1)}{nr-n-r\az}}\in A^{\loc}_1 (\rn).
\end{equation}
Let
\begin{equation}\label{8.18}
\frac{1}{p_1}\equiv\frac{1}{r}+\frac{\az}{n}+\left(1-\frac{1}{r}-
\frac{\az}{n}\r)/(1+\ez_1)\ \text{and}\
\frac{1}{q_1}\equiv\frac{1}{p_1}-\frac{\az}{n}.
\end{equation}
Then by $r\in(\frac{n}{n-\az},\fz)$, we know that
\begin{equation}\label{8.19}
p_1\in\left(1,\frac{n}{\az}\r),\  r>q_1 \ \text{and}\ \oz^{-r'}\in
A^{\loc}\left(\frac{q_1'}{r'},\,\frac{p_1'}{r'}\r).
\end{equation}
Furthermore, from \eqref{8.17}, the fact that
$p_1<\frac{nr(1+\ez_1)}{nr-n-r\az}$ and H\"older's inequality, we
infer that
\begin{equation}\label{8.20}
\oz^{p_1}\in A^{\loc}_1 (\rn),
\end{equation}
which together with Lemma \ref{l2.3}(ii) implies that there exists
an $\ez_2\in(0,\fz)$ such that $\oz^{p_1 (1+\ez_2)}\in A^{\loc}_1
(\rn)$. Let
\begin{equation}\label{8.21}
\wz{q}\equiv p_1 (1+\ez_2).
\end{equation}
By $\frac{nr}{nr-n-r\az}>p$ and H\"older's inequality, we see that
$\oz^p \in A^{\loc}_1 (\rn)$. Let $s\equiv\lfz n(1/p -1)\rf$. To show
Theorem \ref{t8.10}, by the facts that $h^{p}_{\oz^p}(\rn)$ and
$L^q_{\oz^q}(\rn)$
 are respectively $p$-quasi-Banach space and 1-quasi-Banach space,
and Theorem \ref{t6.4}(i) with $\bfai(t)\equiv t^p$ for all
$t\in(0,\fz)$, it suffices to show that for any
$(p,\,\wz{q},\,s)_{\oz^p}$-atom $a$ supported in $Q_0\equiv
Q(x_0,R_0)$ with $R_0\in(0,2]$,
\begin{equation}\label{8.22}
\left\|I^{\loc}_{\az}(a)\r\|_{L^q_{\oz^q}(\rn)}\ls1.
\end{equation}

By $\supp (a)\subset Q_0$ and the definition of $I^{\loc}_{\az}(a)$,
we see that
\begin{equation}\label{8.23}
\supp\left(I^{\loc}_{\az}(a)\r)\subset Q(x_0,R_0 +4).
\end{equation}
Now, we prove \eqref{8.22} by considering the following two cases
for $R_0$.

{\it Case 1)} $R_0 \in[1,2]$. In this case, from \eqref{8.23},
H\"older's inequality, \eqref{8.19}, Lemma \ref{8.3}, $R_0 \in[1,2]$
and $\frac{1}{q}-\frac{1}{q_1}=\frac{1}{p}-\frac{1}{p_1}$, we deduce that
\begin{eqnarray}\label{8.24}
&&\left\{\int_{\rn}\left|I^{\loc}_{\az}(a)(x)\r|^q
[\oz(x)]^q\,dx\r\}^{1/q}\nonumber\\
 &&\hs\ls\left\{\int_{Q(x_0,R_0
+4)}\left|I^{\loc}_{\az}(a)(x)\r|^{q_1}
[\oz(x)]^{q_1}\,dx\r\}^{1/{q_1}}|Q_0|^{\frac{1}{q}-\frac{1}{q_1}}\nonumber\\
&&\hs\ls\|a\|_{L^{p_1}_{\oz^{p_1}}(\rn)}|Q_0|^{\frac{1}{p}-\frac{1}{p_1}}.
\end{eqnarray}
By \eqref{8.20} and the definition of $A^{\loc}_1 (\rn)$, we know
that $\oz^{\frac{p_1 (\wz{q}-p)}{(\wz{q}-p_1)}}\in A^{\loc}_1
(\rn)$. From this and Lemma \ref{8.2}(i), we infer that $\oz^p\in
RH^{\loc}_{\frac{p_1 (\wz{q}-p)}{p(\wz{q}-p_1)}} (\rn)$, which
implies that
\begin{eqnarray}\label{8.25}
\left\{\int_{Q_0}[\oz(x)]^p\,dx\r\}^{\frac{1}{\wz{q}}-
\frac{1}{p}}\left\{\int_{Q_0}[\oz(x)]^{\frac{p_1
(\wz{q}-p)}{(\wz{q}-p_1)}}\,dx\r\}^{\frac{1}{p_1}-
\frac{1}{\wz{q}}}\ls |Q_0|^{\frac{1}{p_1}- \frac{1}{p}}.
\end{eqnarray}
This, combined with \eqref{8.24}, H\"older's inequality and the fact
that $a$ is a $(p,\,\wz{q},\,s)_{\oz^p}$-atom, yields that
\begin{eqnarray*}
\left\|I^{\loc}_{\az}(a)\r\|_{L^q_{\oz^q}(\rn)}&\ls&
\|a\|_{L^{p_1}_{\oz^{p_1}}(\rn)}|Q_0|^{\frac{1}{p}-\frac{1}{p_1}}\\
&\ls&\left\{\int_{Q_0}|a(x)|^{\wz{q}}[\oz(x)]^p\,dx\r\}^{1/\wz{q}}
\left\{\int_{Q_0}[\oz(x)]^{\frac{p_1
(\wz{q}-p)}{(\wz{q}-p_1)}}\,dx\r\}^{\frac{1}{p_1}-
\frac{1}{\wz{q}}}|Q_0|^{\frac{1}{p}-\frac{1}{p_1}}\\
&\ls& \left\{\int_{Q_0}[\oz(x)]^p\,dx\r\}^{\frac{1}{\wz{q}}-
\frac{1}{p}}\left\{\int_{Q_0}[\oz(x)]^{\frac{p_1
(\wz{q}-p)}{(\wz{q}-p_1)}}\,dx\r\}^{\frac{1}{p_1}-
\frac{1}{\wz{q}}}|Q_0|^{\frac{1}{p}-\frac{1}{p_1}}\ls1.
\end{eqnarray*}
This shows \eqref{8.22} in Case 1).

{\it Case 2)} $R_0\in(0,1)$. In this case, let
$\wz{Q}_0\equiv4nQ_0$. From \eqref{8.23}, it follows that
\begin{eqnarray}\label{8.26}
\left\|I^{\loc}_{\az}(a)\r\|_{L^q_{\oz^q}(\rn)}&\le&
\left\{\int_{\wz{Q}_0}\left|I^{\loc}_{\az}(a)(x)\r|^q
[\oz(x)]^q\,dx\r\}^{1/q}\nonumber\\
&&+\left\{\int_{Q(x_0,R_0 +4)\setminus
\wz{Q}_0}\left|I^{\loc}_{\az}(a)(x)\r|^q [\oz(x)]^q\,dx\r\}^{1/q}\nonumber\\
 &\equiv&\mathrm{I_1}+\mathrm{I_2}.
\end{eqnarray}

To estimate $\mathrm{I_1}$, by H\"older's inequality, \eqref{8.15}
and \eqref{8.25}, we conclude that
\begin{eqnarray}\label{8.27}
\mathrm{I_1}&\le& \left(\int_{Q_1}\left|I^{\loc}_{\az}(a)(x)\r|^{q_1}
[\oz(x)]^{q_1}\,dx\r)^{1/{q_1}}|Q_0|^{\frac{1}{p}-\frac{1}{p_1}}\nonumber\\
&\ls&\|a\|_{L^{p_1}_{\oz^{p_1}}(\rn)}|Q_0|^{\frac{1}{p}-\frac{1}{p_1}}\nonumber\\
 &\ls&\left\{\int_{Q_0}[\oz(x)]^p\,dx\r\}^{\frac{1}{\wz{q}}-
\frac{1}{p}}\left\{\int_{Q_0}[\oz(x)]^{\frac{p_1
(\wz{q}-p)}{(\wz{q}-p_1)}}\,dx\r\}^{\frac{1}{p_1}-
\frac{1}{\wz{q}}}|Q_0|^{\frac{1}{p}-\frac{1}{p_1}}\ls1.
\end{eqnarray}

To estimate $\mathrm{I_2}$, for any fixed $x\in Q(x_0,R_0
+4)\setminus \wz{Q}_0$, let $E^s$ be the Taylor expansion of
$\frac{\phi_0 (\cdot)}{|\cdot|^{n-\az}}$ about $x-x_0$ with degree
$s$. Let $m_0$ be the integer such that
$$2^{m_0 -1}nR_0\le R_0+4<2^{m_0}nR_0.$$
Since $\oz^p\in A^{\loc}_1 (\rn)\subset
A^{\loc}_{\wz{q}}(\rn)$, by \eqref{2.1}, we have
$$\left(\int_{Q_0}[\oz(x)]^p\,dx\r)^{\frac{1}{\wz{q}}-
\frac{1}{p}} \left(\int_{Q_0}[\oz(x)]^{-\frac{p\wz{q}'}{\wz{q}}}
\,dx\r)^{\frac{1}{\wz{q}'}}\ls\left(\int_{Q_0}[\oz(x)]^p
\,dx\r)^{\frac{1}{p}}|Q_0|.$$ From this, the vanishing condition of
$a$, Minkowski's inequality, Taylor's remainder theorem, the fact
that
$$\sum_{\az\in\zz_+^n,\,|\az|=s+1}\left|\partial^{\az}
\left(\frac{\phi_0 (\cdot)}{|\cdot|^{n-\az}}\r)
(z)\r|\ls\frac{1}{|z|^{n+s+1-\az}}$$ for all
$z\in\rn\setminus\{0\}$, and H\"older's inequality, we deduce that
\begin{eqnarray}\label{8.28}
\hs\mathrm{I_2}&\le&\left(\sum_{k=2}^{m_0}\int_{2^{k+1}nQ_0\setminus2^{k}nQ_0}
\left\{\int_{Q_0}\left|\frac{\phi_0 (x-y)}{|x-y|^{n-\az}}-E^s
(x-y)\r|\r.\r. \nonumber\\  && \,\times|a(y)|\,dy\Bigg\}^q
[\oz(x)]^q \,dx\Bigg)^{1/q}
\nonumber\\
&\le&\sum_{k=2}^{m_0}\int_{Q_0}|a(y)|\left\{\int_{2^{k+1}nQ_0\setminus2^{k}nQ_0}
\left|\frac{\phi_0 (x-y)}{|x-y|^{n-\az}}-E^s (x-y)\r|^q
[\oz(x)]^q\,dx\r\}^{1/q}\,dy\nonumber\\
&\ls&\sum_{k=2}^{m_0}\int_{Q_0}|a(y)|\left\{\int_{2^{k+1}nQ_0\setminus2^{k}nQ_0}
\left(\frac{|y-x_0|^{s+1}}{|\tz(x-y)-(1-\tz)(x-x_0)|^{n+s+1-\az}}\r)^q\r.
\nonumber\\
&& \,\times[\oz(x)]^q\,dx\Bigg\}^{1/q}\,dy\nonumber\\
&\ls&\sum_{k=2}^{m_0}\int_{Q_0}|a(y)|\left\{\int_{2^{k+1}nQ_0\setminus2^{k}nQ_0}
\left(\frac{|y-x_0|^{s+1}}{|x-x_0|^{n+s+1-\az}}\r)^q
[\oz(x)]^q\,dx\r\}^{1/q}\,dy\nonumber\\
&\ls&\sum_{k=2}^{m_0}\frac{R_0^{\az-n}}{2^{k(n+s+1-\az)}}
\left\{\int_{Q_0}|a(y)|\,dy\r\}\left\{\int_{2^{k+1}nQ_0}[\oz(x)]^q
\,dx\r\}^{1/q}\nonumber\\
&\ls&\sum_{k=2}^{m_0}\frac{R_0^{\az-n}}{2^{k(n+s+1-\az)}}
\left\{\int_{Q_0}|a(y)|^{\wz{q}}[\oz(y)]^p\,dy\r\}^{1/\wz{q}}\nonumber\\
 &&
\,\times\left\{\int_{Q_0}[\oz(y)]^{-p\wz{q}^{'}/\wz{q}}\,dy\r\}^{1/
\wz{q}^{'}}\left\{\int_{2^{k+1}nQ_0}[\oz(x)]^q \,dx\r\}^{1/q}\nonumber\\
 &\ls&\sum_{k=2}^{m_0}\frac{R_0^{\az}}{2^{k(n+s+1-\az)}}
\left\{\int_{Q_0}[\oz(x)]^p\,dx\r\}^{-1/p}\left\{\int_{2^{k+1}nQ_0}[\oz(x)]^q
\,dx\r\}^{1/q},
\end{eqnarray}
where $\tz\in(0,1)$ and in the fourth inequality we used the fact
that for any $y\in Q_0$ and $x\in2^{k+1}nQ_0\setminus2^{k}nQ_0$ with
$k\in\{2,\,\cdots,\,m_0\}$, $|(x-x_0)-\tz(y-x_0)|\gs|x-x_0|$.

From $\frac{nr(1+\ez_1)}{nr-n-r\az}=\frac{rq_1}{r-q_1}>
\frac{rq}{r-q}$, \eqref{8.17} and H\"older's inequality, it follows
that
\begin{eqnarray}\label{8.29}
\oz^\frac{rq}{r-q}\in A^{\loc}_1 (\rn).
\end{eqnarray}
By Lemma \ref{l2.3}(i) and Remark \ref{r2.4} with $\wz{C}\equiv20n$,
we know that there exists a function $\wz{\oz}$ on $\rn$ such that
$\wz{\oz}^\frac{rq}{r-q}\in A_1 (\rn)$ such that $\wz{\oz}=\oz$ on
$Q(x_0,20n)$, which together with $2^{m_0 +1}nQ_0\subset Q(x_0,20n)$
and Lemma \ref{l2.3}(viii) implies that for any
$k\in\{1,\,2,\,\cdots,\,m_0\}$,
$$\int_{
2^k nQ_0}[\oz(x)]^{\frac{rq}{r-q}}\,dx=\int_{2^k
nQ_0}[\wz{\oz}(x)]^{\frac{rq}{r-q}}\,dx\ls2^{kn}
\int_{Q_0}[\wz{\oz}(x)]^{\frac{rq}{r-q}}\,dx\ls2^{kn}\int_{
Q_0}[\oz(x)]^{\frac{rq}{r-q}}\,dx.$$ By this estimate and H\"older's
inequality, we have
\begin{eqnarray}\label{8.30}
\left\{\int_{2^{k+1}nQ_0}[\oz(x)]^q \,dx\r\}^{1/q}\ls
R_0^{n/r}2^{\frac{kn}{q}}
\left\{\int_{Q_0}[\oz(x)]^{\frac{rq}{r-q}}\,dx\r\}^{\frac{1}{q}-\frac{1}{r}}.
\end{eqnarray}
Moreover, by \eqref{8.29} and Lemma \ref{l8.8}(i), we know that
$\oz^p \in RH^{\loc}_{\frac{rq}{p(r-q)}}(\rn)$. Thus, we have
\begin{eqnarray*}
\left\{\int_{Q_0}[\oz(x)]^p\,dx\r\}^{-1/p}
\left\{\int_{Q_0}[\oz(x)]^{\frac{rq}{r-q}}\,dx\r\}
^{\frac{1}{q}-\frac{1}{r}}\ls R_0^{-\frac{n}{r}-\az},
\end{eqnarray*}
which together with \eqref{8.28} and \eqref{8.30} implies that
$$\mathrm{I_2}\ls\sum_{k=2}^{m_0}2^{-k(n+s+1-\az-n/q)}.$$
From $\frac{1}{q}=\frac{1}{p}-\frac{\az}{n}$ and
$r>\frac{n}{n-\az}$, we deduce that $n+s+1-\az-n/q>n+s+1-n/p$, which
together with $s=\lfz n(1/p -1)\rf$ implies that $n+s+1-\az-n/q>0$.
Thus,
$$\mathrm{I_2}\ls\sum_{k=2}^{m_0}2^{-k(n+s+1-\az-n/q)}\ls1.$$
This combined with \eqref{8.26} and \eqref{8.27} proves \eqref{8.22}
in Case 2), which completes the proof of Theorem \ref{t8.10}.
\end{proof}

\begin{theorem}\label{t8.11}
Let $\az\in(0,1)$, $p\in(0,\frac{n}{n+\az}]$ and
$\frac{1}{q}=\frac{1}{p}-\frac{\az}{n}$. For some
$r\in(\frac{n}{n-\az},\fz)$, if the weight $\oz$ satisfies
$\oz^{\frac{nr}{nr-n-r\az}}\in A^{\loc}_1 (\rn)$ and
$\int_{\rn}[\oz(x)]^p\,dx=\fz$, then there exists a positive
constant $C$ such that for all $f\in h^p_{\oz^p}(\rn)$,
$$\left\|I^{\loc}_{\az}(f)\r\|_{h^q_{\oz^q}(\rn)}\le
C\|f\|_{h^p_{\oz^p}(\rn)}.$$
\end{theorem}

\begin{proof}
Let $p_1,\,q_1$ and $\wz{q}$ be respectively as in \eqref{8.18} and
\eqref{8.21}. To show Theorem \ref{t8.11}, by the facts that
$h^{p}_{\oz^p}(\rn)$ and $h^q_{\oz^q}(\rn)$ are respectively the
$p$-quasi-Banach space and the $q$-quasi-Banach space, Theorem
\ref{t6.4}(i) with $\bfai(t)\equiv t^p$ for all $t\in(0,\fz)$ and
Theorem \ref{t3.14}, it suffices to show that for any
$(p,\,\wz{q},\,s)_{\oz^p}$-atom $a$ supported in $Q_0\equiv
Q(x_0,R_0)$ with $R_0\in(0,2]$,
\begin{eqnarray}\label{8.31}
\left\|\cg^0_N\left(I^{\loc}_{\az}(a)\r)\r\|_{L^{\bfai}_{\oz^q}(\rn)}\ls1.
\end{eqnarray}

By $\supp(a)\subset Q_0$ and definitions of $I^{\loc}_{\az}(a)$ and
$\cg^0_N(I^{\loc}_{\az}(a))$, we know that
\begin{eqnarray}\label{8.32}
\supp\left(\cg^0_N(I^{\loc}_{\az}(a))\r)\subset Q(x_0,R_0 +8).
\end{eqnarray}
Now, we prove \eqref{8.31} by considering the following two cases
for $R_0$.

{\it Case 1)} $R_0 \in[1,2]$. In this case, by \eqref{8.17}, the
fact that $\frac{nr(1+\ez_1)}{nr-n-r\az}>q_1$ and H\"older's
inequality, we know that $\oz^{q_1}\in A^{\loc}_1 (\rn)$, where
$\ez_1$ is as in \eqref{8.17}. From this, \eqref{8.32}, H\"older's
inequality, Proposition \ref{p3.2}(ii), Lemma \ref{l8.9} and
\eqref{8.25}, it follows that
\begin{eqnarray*}
&&\int_{\rn}\left|\cg^0_N\left(I^{\loc}_{\az}(a)\r)(x)\r|^q
[\oz(x)]^q\,dx\\
&&\hs\ls|Q_0|^{1-\frac{q}{q_1}}\left\{\int_{Q(x_0,R_0 +8)}
\left|\cg^0_N\left(I^{\loc}_{\az}(a)\r)(x)\r|^{q_1}
[\oz(x)]^{q_1}\,dx\r\}^{q/q_1}\\
&&\hs\ls|Q_0|^{1-\frac{q}{q_1}}\left\{\int_{Q(x_0,R_0
+8)}\left|I^{\loc}_{\az}(a)(x)\r|^{q_1}
[\oz(x)]^{q_1}\,dx\r\}^{q/q_1}\\
&&\hs\ls|Q_0|^{1-\frac{q}{q_1}}\left\{\int_{Q_0}|a(x)|^{p_1}
[\oz(x)]^{p_1}\,dx\r\}^{q/p_1}\\
&&\hs\ls|Q_0|^{1-\frac{q}{q_1}}\left\{\int_{Q_0}|a(x)|^{\wz{q}}
[\oz(x)]^{p}\,dx\r\}^{q/\wz{q}}\left\{\int_{Q_0}[\oz(x)]^{\frac{p_1
(\wz{q}-p)}{\wz{q}-p_1}}\,dx\r\}^{(\frac{1}{p_1}-\frac{1}{\wz{q}})q}\\
&&\hs\ls\left\{|Q_0|^{\frac{1}{q}-\frac{1}{q_1}}\left(\int_{Q_0}[\oz(x)]^p
\,dx\r)^{\frac{1}{\wz{q}}-\frac{1}{p}}\left(\int_{Q_0}[\oz(x)]^{\frac{p_1
(\wz{q}-p)}{\wz{q}-p_1}}\,dx\r)^
{\frac{1}{p_1}-\frac{1}{\wz{q}}}\r\}^q\ls1,
\end{eqnarray*}
which proves \eqref{8.31} in Case 1).

{\it Case 2)} $R_0\in(0,1)$. In this case, let $\wz{Q}_0
\equiv8nQ_0$. Then from \eqref{8.32}, we conclude that
\begin{eqnarray}\label{8.33}
\int_{\rn}\left[\cg^0_N\left(I^{\loc}_{\az} a\r)(x)\r]^q [\oz(x)]^q\,dx
&=&\int_{\wz{Q}_0}\left[\cg^0_N\left(I^{\loc}_{\az} a\r)(x)\r]^q
[\oz(x)]^q\,dx\nonumber\\  &&\hs+\int_{Q(x_0,R_0
+8)\setminus\wz{Q}_0} \cdots \equiv\mathrm{I_1}+\mathrm{I_2}.
\end{eqnarray}
For $\mathrm{I_1}$, similarly to the proof of Case 1), we have
\begin{eqnarray}\label{8.34}
\mathrm{I_1}&\ls&|Q_0|^{1-\frac{q}{q_1}}\left\{\int_{\wz{Q}_0}
\left[\cg^0_N\left(I^{\loc}_{\az}(a)\r)(x)\r]^{q_1}
[\oz(x)]^{q_1}\,dx\r\}^{q/q_1}\nonumber\\
&\ls&|Q_0|^{1-\frac{q}{q_1}}\left\{\int_{\rn}\left|I^{\loc}_{\az}(a)(x)\r|^{q_1}
[\oz(x)]^{q_1}\,dx\r\}^{q/q_1}\nonumber\\
&\ls&|Q_0|^{1-\frac{q}{q_1}}\left\{\int_{Q_0}|a(x)|^{p_1}
[\oz(x)]^{p_1}\,dx\r\}^{q/p_1}\nonumber\\
&\ls&|Q_0|^{1-\frac{q}{q_1}}\left\{\int_{Q_0}|a(x)|^{\wz{q}}
[\oz(x)]^{p}\,dx\r\}^{q/\wz{q}}\left\{\int_{Q_0}[\oz(x)]^{\frac{p_1
(\wz{q}-p)}{\wz{q}-p_1}}\,dx\r\}^
{(\frac{1}{p_1}-\frac{1}{\wz{q}})q}\nonumber\\
&\ls&\left\{|Q_0|^{\frac{1}{q}-\frac{1}{q_1}}\left(\int_{Q_0}[\oz(x)]^p
\,dx\r)^{\frac{1}{\wz{q}}-\frac{1}{p}}\left(\int_{Q_0}[\oz(x)]^{\frac{p_1
(\wz{q}-p)}{\wz{q}-p_1}}\,dx\r)^
{\frac{1}{p_1}-\frac{1}{\wz{q}}}\r\}^q\ls1.
\end{eqnarray}
To estimate $\mathrm{I_2}$, let $x\in(\wz{Q}_0)^\complement$,
$\pz\in\cd^0_N (\rn)$, $t\in(0,1)$ and $P_{\pz}^s$ be the Taylor
expansion of $\pz$ about $(x-x_0)/t$ with degree $s$, where
$s\equiv\lfz n(1/p-1)\rf$. Then we have
\begin{eqnarray}\label{8.35}
\left|I^{\loc}_{\az}(a)\ast\pz_t
(x)\r|&=&\frac{1}{t^n}\left|\int_{\rn}I^{\loc}_{\az}
(a)(y)\pz\left(\frac{x-y}{t}\r)\,dy\r|\nonumber\\
 &=&\frac{1}{t^n}\left|\int_{\rn}I^{\loc}_{\az}
(a)(y)\left[\pz\left(\frac{x-y}{t}\r)-P_{\pz}^s
\left(\frac{x-y}{t}\r)\r]\,dy\r|\nonumber\\
 &\le&\frac{1}{t^n}\int_{2\sqrt{n}Q_0} \left|I^{\loc}_{\az}
(a)(y)\r|\left|\pz\left(\frac{x-y}{t}\r)-P_{\pz}^s
\left(\frac{x-y}{t}\r)\r|\,dy\nonumber\\
&&+\frac{1}{t^n}\int_{Q(x_0,\frac{|x-x_0|}{2\sqrt{n}}
)\setminus2\sqrt{n}Q_0}\cdots+
\frac{1}{t^n}\int_{[Q(x_0,\frac{|x-x_0|}{2\sqrt{n}})]^\complement}\cdots
\nonumber\\
&\equiv&\mathrm{E_1}+\mathrm{E_2}+\mathrm{E_3}.
\end{eqnarray}
To estimate $\mathrm{E_1}$, by $x\in(\wz{Q}_0)^\complement$ and
$t\in(0,1)$, we see that $\mathrm{E_1}\neq0$ implies that
$t>\frac{|x-x_0|}{2}$. From
$$\oz^{\frac{nr}{nr-n-r\az}}\in
A^{\loc}_1 (\rn)$$
and the definition of $A^{\loc}_1 (\rn)$, it
follows that $\oz\in A^{\loc}_1 (\rn)$. Let $q_2 \equiv \frac{2q_1
-1}{q_1}$. Then by $\oz\in A^{\loc}_1 (\rn)\subset
A^{\loc}_{q_2}(\rn)$ and Lemma \ref{l2.3}(iv), we know that
$\oz^{-q'_1}=\oz^{1-q'_2}\in A^{\loc}_{q'_2}(\rn)$. From these
facts, Taylor's remainder theorem, H\"older's inequality, Lemmas
\ref{l8.9} and \ref{l2.3}(v), Remark \ref{r2.4} with
$\wz{C}=2\sqrt{n}$, \eqref{8.25} and \eqref{2.2}, we infer that
\begin{eqnarray}\label{8.36}
\mathrm{E_1}&\ls&\frac{1}{t^{n+s+1}}\left\|I^{\loc}_{\az}
(a)\r\|_{L^{q_1}_{\oz^{q_1}}(\rn)}
\left\{\sum_{\az\in\zz^n_+,\,|\az|=s+1}\int_{2\sqrt{n}Q_0}
\left|\partial^{\az}\pz\left(\frac{\xi}{t}\r)\r|^{q_1'}\r.\nonumber\\
&&\,\times|y-x_0|
^{(s+1){q_1'}}[\oz(y)]^{-q_1'}\,dy\Bigg\}^{1/q_1'}\nonumber\\
&\ls&\frac{R_0^{s+1}}{|x-x_0|^{n+s+1}}\|a\|_{L^{p_1}
_{\oz^{p_1}}(\rn)} \left\{\int_{2\sqrt{n}Q_0}[\oz(y)]^{-q_1'}\,dy\r\}^{1/q_1'}
\nonumber\\
&\ls&\frac{R_0^{s+1}}{|x-x_0|^{n+s+1}}\left\{\int_{Q_0}|a(x)|^{\wz{q}}
[\oz(x)]^{p}\,dx\r\}^{1/\wz{q}}\nonumber\\
&&\,\times\left\{\int_{Q_0}[\oz(x)]^{\frac{p_1
(\wz{q}-p)}{\wz{q}-p_1}}\,dx\r\}^{\frac{1}{p_1}-
\frac{1}{\wz{q}}}\left\{\int_{Q_0}[\oz(y)]^{-q_1'}\,dy\r\}^{1/q_1'}\nonumber\\
 &\ls&\frac{R_0^{s+1}}{|x-x_0|^{n+s+1}}|Q_0|
^{\frac{1}{q_1}-\frac{1}{q}}
\left\{\int_{Q_0}[\oz(y)]^{-q_1'}\,dy\r\}^{1/q_1'}\nonumber\\
&\ls&
\frac{R_0^{s+1}}{|x-x_0|^{n+s+1}}|Q_0|
^{1-\frac{1}{q}}\left[\einf_{z\in Q_0}\oz(z)\r]^{-1},
\end{eqnarray}
where $\gz\in(0,1)$ and $\xi\equiv\gz(x-y)+(1-\gz) (x-x_0)$.
Similarly to the estimates of $G_2$ and $G_3$ in the proof of
Theorem \ref{t8.2}, we have
\begin{eqnarray}\label{8.37}
\max\{\mathrm{E_2},\,\mathrm{E_3}\}\ls\frac{R_0^{n+s+1}}{|x-x_0|
^{n+s+1-\az}}|Q_0|^{-1/p}\left[\einf_{z\in Q_0}\oz(z)\r]^{-1}.
\end{eqnarray}
Thus, from \eqref{8.35}, \eqref{8.36}, \eqref{8.37} and the facts
that $|x-x_0|\ge4nR_0$ and $\frac{1}{q}=\frac{1}{p}-\frac{\az}{n}$,
we deduce that
\begin{eqnarray*}
\left|\left[I^{\loc}_{\az} (a)\r]\ast\pz_t
(x)\r|&\ls&\frac{R_0^{s+1}}{|x-x_0|^{n+s+1}}|Q_0|
^{1-\frac{1}{q}}\left[\einf_{z\in
Q_0}\oz(z)\r]^{-1}\\
&&+\frac{R_0^{n+s+1}}{|x-x_0|
^{n+s+1-\az}}|Q_0|^{-1/p}\left[\einf_{z\in
Q_0}\oz(z)\r]^{-1}\\
&\ls&\frac{R_0^{n+s+1}}{|x-x_0|
^{n+s+1-\az}}|Q_0|^{-1/p}\left[\einf_{z\in Q_0}\oz(z)\r]^{-1},
\end{eqnarray*}
which together with the arbitrariness of $\pz\in\cd^0_N (\rn)$
implies that for any $x\in(\wz{Q}_0)^\complement$,
\begin{eqnarray}\label{8.38}
\cg^0_N \left(I^{\loc}_{\az} (a)\r)(x)\ls\frac{R_0^{n+s+1}}{|x-x_0|
^{n+s+1-\az}}|Q_0|^{-1/p}\left[\einf_{z\in Q_0}\oz(z)\r]^{-1}.
\end{eqnarray}
By $s=\lfz n(1/p-1)\rf$ and $\frac{1}{q}=\frac{1}{p}-\frac{\az}{n}$,
we know that $(n+s+1-\az)q-n>0$. Let $m_0$ be the integer such that
$2^{m_0 }nR_0\le R_0+8<2^{m_0+1}nR_0$. By
$$\oz^{\frac{nr}{nr-n-r\az}} \in A^{\loc}_1 (\rn)$$
and the definition
of $A^{\loc}_1 (\rn)$, we see that $\oz^q \in A^{\loc}_1 (\rn)$,
which together Lemma \ref{l2.3}(i) implies that there exists a
function $\wz{\oz}$ on $\rn$ such that $\wz{\oz}=\oz$ on $Q(x_0,
R_0+8)$ and $\wz{\oz}^q \in A_1 (\rn)$. From this,
$\frac{1}{q}=\frac{1}{p}-\frac{\az}{n}$, \eqref{8.38}, (i) and
(viii) of Lemma \ref{l2.3}, the definition of $A^{\loc}_1 (\rn)$ and
$(n+s+1-\az)q-n>0$, we infer that
\begin{eqnarray*}
\mathrm{I_2}&\ls&\int_{Q(x_0,R_0
+8)\setminus\wz{Q}_0}\left\{\frac{R_0^{n+s+1}}{|x-x_0|
^{n+s+1-\az}}|Q_0|^{-1/p}\left[\einf_{z\in Q_0}\oz(z)\r]^{-1}\r\}^q
[\wz{\oz}(x)]^q\,dx\\
&\ls&\sum_{k=3}^{m_0}\int_{2^{k+1}nQ_0\setminus2^k nQ_0}
\left\{\frac{R_0^{n+s+1-n/q-\az}}{(2^k R_0) ^{n+s+1-\az}}\left[\einf_{z\in
Q_0}\oz(z)\r]^{-1}\r\}^q
[\wz{\oz}(x)]^q\,dx\\
&\ls&\sum_{k=3}^{m_0}\frac{R_0^{-n}}{2^{k[(n+s+1-\az)q-n]}}\left[\einf_{z\in
Q_0}\oz(z)\r]^{-q}\int_{Q_0}[\oz(x)]^q\,dx\\
&\ls&\sum_{k=3}^{m_0}\frac{R_0^{-n}}{2^{k[(n+s+1-\az)q-n]}}|Q_0|\left[\einf_{z\in
Q_0}\oz(z)\r]^{-q}\einf_{z\in Q_0}[\oz(z)]^q\\
&\ls&\sum_{k=3}^{m_0}2^{-k[(n+s+1-\az)q-n]}\ls1,
\end{eqnarray*}
which together with \eqref{8.33} and \eqref{8.34} implies
\eqref{8.31} in Case 2). This finishes the proof of Theorem
\ref{t8.11}.
\end{proof}

The pseudo-differential operators have been extensively studied in
the literature, and they are important in the study of partial
differential equations and harmonic analysis; see, for example,
\cite{St93,Tay91,S99,Ta2}. Now, we recall the notion of the
pseudo-differential operators of order zero.

\begin{definition}\label{d8.12}
Let $\dz$ be a real number. A {\it  symbol} in $S^0_{1,\dz}(\rn)$ is
a smooth function $\sz(x,\xi)$ defined on $\rn\times\rn$ such that
for all multi-indices $\az$ and $\bz$, the following estimate holds:
$$\left|\partial^{\az}_x \partial^{\bz}_{\xi}\sz(x,\xi)\r|\le
C({\az,\,\bz})(1+|\xi|)^{-|\bz|+\dz|\az|},$$ where $C(\az,\,\bz)$ is
a positive constant independent of $x$ and $\xi$. Let $f$ be a
Schwartz function and $\hat{f}$ denote its Fourier transform. The
operator $T$ given by setting, for all $x\in\rn$,
$$Tf(x)\equiv\int_{\rn}\sz(x,\xi)e^{2\pi ix\xi}\hat{f}(\xi)\,d\xi
$$
is called a {\it  pseudo-differential operator} with symbol
$\sz(x,\xi)\in S^0_{1,\dz}(\rn)$.
\end{definition}

In the rest of this section, let
\begin{eqnarray}\label{8.39}
\phi(t)\equiv(1+t)^{\az}
\end{eqnarray}
for all $\az\in(0,\fz)$  and $t\in(0,\fz)$. Recall that a
\emph{weight} always means a locally integrable function which is
positive almost everywhere. The following notion of the weight class
$A_p (\phi)$ was introduced by Tang \cite{Ta2}.

\begin{definition}\label{d8.13}
A weight $\oz$ is said to belong to the \emph{class} $A_p (\phi)$
for $p\in(1,\fz)$, if there exists a positive constant $C$ such that
for all cubes $Q\equiv Q(x, r)$,
$$\left(\frac{1}{\phi(|Q|)|Q|}\int_{Q}\oz(y)\,dy\r)
\left(\frac{1}{\phi(|Q|)|Q|}\int_{Q}[\oz(y)]^{-\frac{1}{p-1}}\,dy\r)^{p-1}\le
C.$$ A weight $\oz$ is said to belong to the \emph{class} $A_1
(\phi)$, if there exists a positive constant $C$ such that for all
cubes $Q\subset\rn$ and almost every $x\in\rn$, $M_{\phi}(\oz)(x)\le
C\oz(x)$, where for all $x\in\rn$,
$$M_{\phi}(\oz)(x)\equiv\sup_{Q\ni x}\frac{1}{\phi(|Q|)|Q|}
\int_{Q}|f(y)|\,dy,$$ and the supremum is taken over all cubes
$Q\subset\rn$ and $Q\ni x$.
\end{definition}

\begin{remark}\label{r8.14}
By the definition of $A_p (\phi)$, we see that $A_p (\phi)\subset
A^{\loc}_p (\rn)$, and that $\phi(t)\ge1$ for all $t\in(0,\fz)$
implies that $A_p (\rn)\subset A_p (\phi)$ for all $p\in[1,\fz)$.
Moreover, if $\oz\in A_p(\phi)$, $\oz(x)\,dx$ may not be a doubling
measure; see the remark of Section 7 in \cite{Ta1} for the details.
\end{remark}

Similarly to the classical Muckenhoupt weights, we recall some
properties of weights $\oz\in A_{\fz}(\phi)\equiv \cup_{1\le p<\fz}
A_p(\phi)$. The following Lemmas \ref{l8.15} and \ref{l8.16} are
respectively Lemmas 7.3 and 7.4 in \cite{Ta1}.

\begin{lemma}\label{l8.15}
$\mathrm{(i)}$ If $1\le p_1 <p_2<\fz$, then $A_{p_1}(\phi)\subset
A_{p_2}(\phi)$.

$\mathrm{(ii)}$ For $p\in(1,\fz)$, $\oz\in A_{p}(\phi)$ if and only
if $\oz^{-\frac{1}{p-1}}\in A_{p'}(\phi)$, where $1/p +1/p' =1$.

$\mathrm{(iii)}$ If $\oz\in A_p (\phi)$ for $p\in[1,\fz)$, then
there exists a positive constant $C$ such that for any cube
$Q\subset\rn$ and measurable set $E\subset Q$,
$$\frac{|E|}{\phi(|Q|)|Q|}\le C\left(\frac{\oz(E)}{\oz(Q)}\r)^{1/p}.$$
\end{lemma}

\begin{lemma}\label{l8.16}
Let $T$ be an $S^0_{1,\,0}(\rn)$ pseudo-differential operator. Then
for $\oz\in A_p (\phi)$ with $p\in(1,\fz)$, there exists a positive
constant $C(p,\,\oz)$ such that for all $f\in L^p_{\oz}(\rn)$,
$$\|Tf\|_{L^p_{\oz}(\rn)}\le
C(p,\,\oz)\|f\|_{L^p_{\oz}(\rn)}.$$
\end{lemma}

The following Lemma \ref{l8.17} is just \cite[Lemma\,6]{Go79}.

\begin{lemma}\label{l8.17}
Let $T$ be an $S^0_{1,\,0}(\rn)$ pseudo-differential operator. If
$\pz\in\cd(\rn)$, then $T_t f=\pz_t \ast Tf$ has a symbol $\sz_t$
which satisfies that
$$\left|\pat^{\bz}_x \pat^{\az}_{\xi}\sz(x,\xi)\r|\le
C(\az,\,\bz)(1+|\xi|)^{-|\az|},$$
and a kernel $K_t (x,\,z)$ which satisfies that
$$\left|\pat^{\bz}_x \pat^{\az}_{z}K_t (x,z)\r|\le
C(\az,\,\bz)|z|^{-n-|\az|},$$
where $ C({\az,\,\bz})$ is independent
of $t$ when $t\in(0,1)$.
\end{lemma}

Now, we establish the boundedness on $h^{\bfai}_{\oz}(\rn)$ of the
$S^0_{1,\,0}(\rn)$ pseudo-differential operators as follows.

\begin{theorem}\label{t8.18}
Let $T$ be an $S^0_{1,\,0}(\rn)$ pseudo-differential operator,
$\bfai$ satisfy Assumption $\mathrm{(A)}$, $\oz\in A_{\fz}(\phi)$
and $p_{\bfai}$ be as in \eqref{2.6}. If $p_{\bfai}=p_{\bfai}^+$ and
$\bfai$ is of upper type $p_{\bfai}^+$, then there exists a positive
constant $C(\bfai,\,\oz)$, depending only on $\bfai$, $q_{\oz}$ and
the weight constant of $A_{\fz}(\phi)$, such that for all $f\in
h^{\bfai}_{\oz}(\rn)$,
$$\|Tf\|_{h^{\bfai}_{\oz}(\rn)}\le
C(\bfai,\,\oz) \|f\|_{h^{\bfai}_{\oz}(\rn)}.$$
\end{theorem}

\begin{proof}
Since $\oz\in A_{\fz}(\phi)$, then $\oz\in A_{q}(\phi)$ for some
$q\in(1,\fz)$. To prove Theorem \ref{t8.18}, by the fact that
$\oz\in A_{\fz}(\phi)\subset A^{\loc}_{\fz}(\rn)$, Theorem
\ref{t6.4}(i) and Theorem \ref{t3.14}, it suffices to show that for
all $(\rz,\,q)_{\oz}$-single-atoms and $(\rz,\,q,\,s)_{\oz}$-atoms
$a$ supported in $Q_0\equiv Q(x_0,\,R_0)$ with $R_0\in(0,2]$,
\begin{eqnarray}\label{8.40}
\left\|\cg^0_N \left(Ta\r)\r\|_{L^{\bfai}_{\oz}(\rn)}\ls1.
\end{eqnarray}
By Theorem \ref{t5.6}, we may assume that $s$ satisfies
$(n+s+1)p_{\bfai}>nq_{\oz}(1+\az)$, where $p_{\bfai},\,q_{\oz}$ and
$\az$ are respectively as in \eqref{2.6}, \eqref{2.4} and
\eqref{8.39}.

First, we prove \eqref{8.40} for any $(\rz,\,q)_{\oz}$-single-atom
$a\neq0$. In this case, $\oz(\rn)<\fz$. Since $\bfai$ is concave, by
Jensen's inequality, H\"older's inequality, Proposition
\ref{p3.2}(ii) and Lemma \ref{l8.16} and \eqref{2.8} with
$t\equiv\oz(\rn)$, we have
\begin{eqnarray*}
&&\int_{\rn}\bfai\left(\cg^0_N(T
a)(x)\r)\oz(x)\,dx\\
&&\hs\le\oz(\rn)\bfai\left(\frac{1}{\oz(\rn)} \int_{\rn}
\cg^0_N(Ta)(x)\oz(x)\,dx\r)\\
&&\hs\le\oz(\rn)\bfai\left(\frac{1}{[\oz(\rn)]^{1/q}}
\left\{\int_{\rn}\left[\cg^0_N(Ta)(x)\r]^{q}\oz(x)\,dx\r\}^{1/q}\r)\\
&&\hs\ls\oz(\rn)\bfai\left(\frac{1}{[\oz(\rn)]^{1/q}}
\|Ta\|_{L^q_{\oz}(\rn)}\r)\ls\oz(\rn)\bfai\left(\frac{1}{[\oz(\rn)]^{1/q}}
\|a\|_{L^q_{\oz}(\rn)}\r)\\
&&\hs\ls\oz(\rn) \bfai\left(\frac{1}{\oz(\rn)\rz(\oz(\rn))}\r)\sim1,
\end{eqnarray*}
which shows \eqref{8.40} in this case.

Now, let $a$ be any $(\rz,\,q,\,s)_{\oz}$-atom supported in
$Q_0\equiv Q(x_0,\,R_0)$ with $R_0\in(0,2]$. Let
$\wz{Q}_0\equiv2Q_0$. Then from Jensen's inequality, H\"older's
inequality, Proposition \ref{p3.2}(ii), Lemma \ref{l8.16}, Lemma
\ref{l8.15}(iii) and \eqref{2.8} with $t\equiv\oz(\wz{Q}_0)$, it
follows that
\begin{eqnarray}\label{8.41}
&&\int_{\wz{Q}_0}\bfai\left(\cg^0_N(T a)(x)\r)\oz(x)\,dx\nonumber\\
&&\hs\le\oz(\wz{Q}_0)\bfai\left(\frac{1}{\oz(\wz{Q}_0)} \int_{\wz{Q}_0}
\cg^0_N(Ta)(x)\oz(x)\,dx\r)\nonumber\\
&&\hs\le\oz(\wz{Q}_0)\bfai\left(\frac{1}{[\oz(\wz{Q}_0)]^{1/q}}
\left\{\int_{\wz{Q}_0}\left[\cg^0_N(Ta)(x)\r]^{q}\oz(x)\,dx\r\}^{1/q}\r)\nonumber\\
 &&\hs\ls\oz(\wz{Q}_0
)\bfai\left(\frac{1}{[\oz(\wz{Q}_0)]^{1/q}}
\|Ta\|_{L^q_{\oz}(\rn)}\r)\ls\oz(\wz{Q}_0)\bfai\left(\frac{1}{[\oz(\wz{Q}_0)]^{1/q}}
\|a\|_{L^q_{\oz}(\rn)}\r)\nonumber \\ &&\hs\ls\oz(\wz{Q}_0)
\bfai\left(\frac{1}{\oz(\wz{Q}_0)\rz(\oz(\wz{Q}_0))}\r)\sim1.
\end{eqnarray}

For any  $\pz\in\cd_N^0 (\rn)$ and $t\in(0,1)$, let $K_t (x,x-z)$ be
the kernel of $T_t a(x)\equiv\pz_t \ast Ta(x)$. To estimate
$\int_{\rn\setminus \wz{Q}_0}\bfai[\cg^0_N(T a)(x)]\oz(x)\,dx$, we
consider the following two cases for $R_0$.

{\it Case 1)} $R_0\in(0,1)$.  In this case, we expand $K_t (x,x-z)$
into a Taylor series about $z=x_0$ such that for any
$x\in(\rn\setminus \wz{Q}_0)$,
$$\pz_t \ast Ta(x)=\int_{\rn}K_t (x,x-z)a(z)\,dz=\int_{Q_0}
\sum_{\gfz{\az\in\zz^n_+,}{|\az|=s+1}} (\pat^{\az}_z K_t)
(x,x-\xi)(z-x_0)^{\az}a(z)\,dz,$$ where $\xi\equiv\tz z+(1-\tz)x_0$
for some $\tz\in(0,1)$. By $z,\,x_0\in Q_0$, we know that $\xi\in
Q_0$. Thus, for any $x\in(\rn\setminus \wz{Q}_0)$,
$|x-\xi|\sim|x-x_0|$. From the above facts and Lemma \ref{l8.17}, we
deduce that
$$|\pz_t \ast Ta(x)|\ls |x-\xi|^{-(n+s+1)}R_0^{
s+1}\|a\|_{L^1 (\rn)}\ls
|x-x_0|^{-(n+s+1)}|Q_0|^{\frac{s+1}{n}}\|a\|_{L^1 (\rn)},$$
which together with the arbitrariness of $\pz\in\cd^0_N (\rn)$ implies
that for all $x\in(\rn\setminus \wz{Q}_0)$,
\begin{eqnarray*}
\cg_N^0(Ta)(x)\ls| x-x_0|^{-(n+s+1)}|Q_0|^{\frac{s+1}{n}}\|a\|_{L^1
(\rn)}.
\end{eqnarray*}
This, combined with H\"older's inequality, Lemma \ref{l8.15}(iii)
and the definition of $A_p (\phi)$, yields that
\begin{eqnarray}\label{8.42}
&&\hs\int_{\rn\setminus \wz{Q}_0}\bfai\left(\cg^0_N(T a)(x)\r)\oz(x)\,dx\nonumber\\
 &&\hs\hs\le C\int_{\rn\setminus \wz{Q}_0}\bfai\left(|
x-x_0|^{-(n+s+1)}|Q_0|^{\frac{s+1}{n}}\|a\|_{L^1 (\rn)}\r)\oz(x)\,dx\nonumber\\
 &&\hs\hs\le C\int_{\rn\setminus
\wz{Q}_0}\bfai\left(|Q_0|^{\frac{s+1}{n}}
|x-x_0|^{-(n+s+1)}\|a\|_{L^q_{\oz}(\rn)}\r.\nonumber\\
&&\hs\hs\hs\times\phi(|Q_0|)
|Q_0|[\oz(Q_0)]^{-1/q}\Big)\oz(x)\,dx\nonumber\\
 &&\hs\hs\le C\int_{\rn\setminus
\wz{Q}_0}\bfai\left(|Q_0|^{\frac{s+1}{n}}
|x-x_0|^{-(n+s+1)}\frac{\phi(|Q_0|)|Q_0|}
{\oz(Q_0)\rz(\oz(Q_0))}\r)\oz(x)\,dx\nonumber\\
 &&\hs\hs\le C\sum_{k=1}^{\fz}\int_{2^k
Q_0}\bfai\left(|Q_0|^{\frac{s+1}{n}} (2^k
R_0)^{-(n+s+1)}\frac{\phi(|Q_0|)|Q_0|} {\oz(Q_0)\rz(\oz(Q_0))}\r)\oz(x)\,dx\nonumber \\
&&\hs\hs\le C\left\{\sum_{k=1}^{m_0}\int_{2^k
Q_0}\bfai\left(|Q_0|^{\frac{s+1}{n}} (2^k R_0)^{-(n+s+1)}\frac{|Q_0|}
{\oz(Q_0)\rz(\oz(Q_0))}\r)\oz(x)\,dx\r.\nonumber\\  &&\hs\hs\hs
+\left.\sum_{k=m_0 +1}^{\fz}\int_{2^k Q_0}\cdots\r\} \equiv
C(\mathrm{I_1}+\mathrm{I_2}),
\end{eqnarray}
where the integer $m_0$ satisfies $2^{m_0
-1}\le\frac{1}{R_0}<2^{m_0}$.

To estimate $\mathrm{I_1}$, for any $k\in\{1,\,2,\,\cdots,\,m_0\}$,
by the choice of $m_0$ and $R_0\in(0,1)$, we know that $2^k
R_0^n\le1$, which, together with Jensen's inequality, the lower type
$p_{\bfai}$ property of $\bfai$, Lemma \ref{l8.15}(iii) and the fact
that $(n+s+1)p_{\bfai}>nq(1+\az)$, implies that
\begin{eqnarray}\label{8.43}
\mathrm{I_1}&\ls&\sum_{k=1}^{m_0}\oz(2^k
Q_0)\bfai\left(\frac{1}{\oz(2^k
Q_0)}\int_{2^k Q_0}2^{-k(n+s+1)}\{\oz(Q_0)\rz(\oz(Q_0))\}^{-1}\oz(x)\,dx\r)
\nonumber\\
 &\ls&\sum_{k=1}^{m_0}2^{-k(n+s+1)p_{\bfai}}\frac{\oz(2^k
Q_0)}{\oz(Q_0)}
\ls\sum_{k=1}^{m_0}2^{knq}[\phi(|2^k Q_0|)]^q 2^{-k(n+s+1)p_{\bfai}}
\nonumber \\
&\ls&\sum_{k=1}^{m_0}2^{-k[(n+s+1)p_{\bfai}-nq]}\ls1.
\end{eqnarray}
For $\mathrm{I_2}$, similarly to the estimate of $\mathrm{I_1}$, we
have
\begin{eqnarray*}
\mathrm{I_2}&\ls&\sum_{k=m_0 +1}^{\fz}\oz(2^k
Q_0)\bfai\left(\frac{1}{\oz(2^k Q_0)}\int_{2^k
Q_0}2^{-k(n+s+1)}\{\oz(Q_0)\rz(\oz(Q_0))\}^{-1}\oz(x)\,dx\r)\\
&\ls&\sum_{k=m_0+1}^{\fz}2^{-k(n+s+1)p_{\bfai}}\frac{\oz(2^k
Q_0)}{\oz(Q_0)} \ls\sum_{k=m_0 +1}^{\fz}2^{knq}[\phi(|2^k Q_0|)]^q
2^{-k(n+s+1)p_{\bfai}}\\
&\ls&\sum_{k=m_0 +1}^{\fz}2^{-k[(n+s+1)p_{\bfai}-q(\az+1)]}\ls1,
\end{eqnarray*}
which together with \eqref{8.43}, \eqref{8.42} and \eqref{8.41}
implies \eqref{8.40} in Case 1).

{\it Case 2)} $R_0\in[1,2]$. In this case, for any
$x\in(\rn\setminus \wz{Q}_0)$ and $z\in Q_0$, we have
$$|x-z|\sim|x-x_0|$$
and $|x-x_0|>1$. From this and
\cite[p.\,235,\,(9)]{St93}, we infer that for any positive integer
$M$, there exists a positive constant $C(M)$ such that
$$|K_t
(x,x-z)|\le C(M) |x-x_0|^{-M},$$ which implies that for any
$x\in(\rn\setminus \wz{Q}_0)$,
$$\left|\pz_t \ast Ta(x)\r|=\int_{\rn}|K_t (x,x-z)a(z)|\,dz\ls|x-x_0|^{-M}
\|a\|_{L^1 (\rn)}.$$
This, combined with the arbitrariness of
$\pz\in\cd^0_N (\rn)$, yields that for any $x\in(\rn\setminus
\wz{Q}_0)$,
\begin{eqnarray}\label{8.44}
\cg^0_N (Ta)(x)\ls|x-x_0|^{-M} \|a\|_{L^1 (\rn)}.
\end{eqnarray}
Take $M>\frac{nq(1+\az)}{p_{\bfai}}$. By Jensen's inequality,
\eqref{8.44}, H\"older's inequality and Lemma \ref{l8.15}(iii), we
have
\begin{eqnarray*}
&&\int_{\rn\setminus \wz{Q}_0}\bfai\left(\cg^0_N(T a)(x)\r)\oz(x)\,dx\\
 &&\hs\ls\int_{\rn\setminus \wz{Q}_0}\bfai\left(|
x-x_0|^{-M}\|a\|_{L^1 (\rn)}\r)\oz(x)\,dx\\
&&\hs\ls\sum_{k=1}^{\fz}\int_{2^k Q_0}\bfai\left((2^k
R_0)^{-M}\frac{\phi(|Q_0|)|Q_0|}{\oz(Q_0)\rz(\oz(Q_0))}\r)\oz(x)\,dx\\
&&\hs\ls\sum_{k=1}^{\fz}\oz(2^k Q_0)\bfai\left((2^k
R_0)^{-M}\frac{1}{\oz(Q_0)\rz(\oz(Q_0))}\r)\\
&&\hs\ls\sum_{k=1}^{\fz}2^{-kMp_{\bfai}}R_0^{-Mp_{\bfai}}\frac{\oz(2^k
Q_0)}{\oz(Q_0)}
\ls\sum_{k=1}^{\fz}2^{-k(Mp_{\bfai}-nq)}\phi(|2^k Q_0|)R_0^{-Mp_{\bfai}}\\
&&\hs\ls\sum_{k=1}^{\fz}2^{-k[Mp_{\bfai}-nq(1+\az)]}
R_0^{-(Mp_{\bfai}-nq\az)}
\ls\sum_{k=1}^{\fz}2^{-k[Mp_{\bfai}-nq(1+\az)]}\ls1,
\end{eqnarray*}
which together with \eqref{8.41} implies \eqref{8.40} in Case 2).
This finishes the proof of Theorem \ref{t8.18}.
\end{proof}

\begin{remark}\label{r8.19}
Let $p\in(0,1]$. Theorem \ref{t8.18} when $\oz\equiv1$ and
$\bfai(t)\equiv t^p$ for all $t\in(0,\fz)$ was obtained by Goldberg
\cite[Theorem\,4]{Go79}; moreover, Theorem \ref{t8.18} when
$\bfai(t)\equiv t^p$ for all $t\in(0,\fz)$ was obtained by Tang
\cite[Theorem\,7.3]{Ta1}. Also, Theorem \ref{t8.18} when $\oz\in
A_1(\rn)$ and $\bfai(t)\equiv t$ for all $t\in(0,\fz)$ was obtained
by Lee, C-C. Lin and Y.-C. Lin \cite[Theorem\,2]{lll10}.
\end{remark}

By Theorems \ref{t8.18} and \ref{t7.5}, \cite[p.\,233,\,(4)]{St93}
and the proposition in \cite[p.\,259]{St93}, we have the following
result.

\begin{corollary}\label{c8.20}
Let $T$ be an $S^0_{1,\,0}(\rn)$ pseudo-differential operator,
$\bfai$ satisfy Assumption $\mathrm{(A)}$, $\oz\in A_{\fz}(\phi)$
and $p_{\bfai}$ be as in \eqref{2.6}. If $p_{\bfai}=p_{\bfai}^+$ and
$\bfai$ is of upper type $p_{\bfai}^+$, then there exists a positive
constant $C (\bfai,\,\oz)$ such that for all
$f\in\bmo_{\rz,\,\oz}(\rn)$,
$$\|Tf\|_{\bmo_{\rz,\,\oz}(\rn)}\le C
(\bfai,\,\oz) \|f\|_{\bmo_{\rz,\,\oz}(\rn)}.$$
\end{corollary}



\end{document}